\documentclass[review]{elsarticle}

\usepackage{lineno,hyperref}
\modulolinenumbers[5]

\journal{}

\usepackage{amsmath}
\usepackage{graphics}
\usepackage{epsfig}
\usepackage{verbatim}
\usepackage{stmaryrd}
\usepackage{epstopdf}
\usepackage{xcolor}
\usepackage{color}
\usepackage{bbm}
\usepackage{amssymb}
\usepackage{mathrsfs}
\usepackage{caption}
\usepackage{amsthm}
\usepackage{nomencl}
\usepackage{multirow}
\usepackage{chngpage}
\usepackage{array}
\usepackage{comment}
\usepackage{enumerate}  
\usepackage{graphicx}
\usepackage{transparent}
\usepackage{mathtools}
\usepackage{tikz}
\usepackage{longtable}
\usepackage[noend]{algpseudocode} 
\usepackage[section]{algorithm}
\usetikzlibrary{shapes.geometric, arrows}
\usetikzlibrary{matrix}
\biboptions{numbers,sort&compress}
\usepackage[top=1in, bottom=1in, left=1.25in, right=1.25in]{geometry}
\makenomenclature
\allowdisplaybreaks

\theoremstyle{plain}
\newtheorem{thm}{\bf Theorem}[section]
\newtheorem{lem}[thm]{\bf Lemma} 

\newtheorem{pro}[thm]{\bf Proposition}

\newtheorem{MNFalg}[thm]{\bf MNF Algorithm}

\theoremstyle{definition}
\newtheorem{rem}[thm]{\bf Remark}
\newtheorem{defn}[thm]{\bf Definition}

\newcommand{\tr}{\transparent{1}}

\newcommand\Tstrut{\rule{0pt}{2.6ex}} 
\newcommand{\red}[1]{\textcolor{black}{#1}}
\newcommand{\blue}[1]{\textcolor{black}{#1}}

\makeatletter
\g@addto@macro\normalsize{%
	\setlength\abovedisplayskip{3pt}
	\setlength\belowdisplayskip{3pt}
	\setlength\abovedisplayshortskip{4pt}
	\setlength\belowdisplayshortskip{4pt}
}
\makeatother

\algnewcommand\Stepa{\item[\textbf{Step 1:}]}
\algnewcommand\Stepb{\item[\textbf{Step 2:}]} 
\algnewcommand\Stepc{\item[\textbf{Step 3:}]} 
\algnewcommand\Stepd{\item[\textbf{Step 4:}]} 
\algnewcommand\Stepe{\item[\textbf{Step 5:}]} 
\algnewcommand\Result{\item[\textbf{Result:}]}
\begin{document}

\begin{frontmatter}

\title{From Morse Triangular Form of ODE Control Systems to Feedback Canonical Form of DAE Control Systems}

\author[First]{Yahao Chen} 
\author[Second]{Witold Respondek} 

\address[First]{Bernoulli Institute for Mathematics, Computer Science, and
		Artificial Intelligence, University of Groningen, The Netherlands
		(email:yahao.chen@rug.nl).}
\address[Second]{Normandie Universit\'e, INSA-Rouen, LMI, 76801 Saint-Etienne-du-Rouvray, France (e-mail: witold.respondek@insa-rouen.fr).}

\begin{abstract}
 In this	paper, we relate the feedback canonical form \textbf{FNCF}  \cite{loiseau1991feedback} of differential-algebraic control systems (DACSs) with  the famous Morse canonical form \textbf{MCF} \cite{morse1973structural},\cite{molinari1978structural} of ordinary differential equation control systems (ODECSs). First, a procedure called \red{an} explicitation (with driving variables) is proposed to connect the \red{two above categories} of control systems by attaching \red{to} a DACS  a class of  ODECSs with two kinds of inputs (the original control input $u$ and a vector of driving variables $v$). Then, we show that any ODECS with two kinds of inputs can be transformed into its extended \textbf{MCF} via two intermediate forms: the extended Morse triangular form  and the extended Morse normal form. Next, we illustrate that the \textbf{FNCF} of a DACS and the  extended \textbf{MCF} of the explicitation system \red{have} a perfect one-to-one correspondence. At last, an algorithm is proposed to transform a given DACS into its \textbf{FBCF} via the explicitation procedure and a numerical example is given to show the efficiency of the proposed algorithm.

\end{abstract}
\end{frontmatter}
\section{Introduction}\label{section:1}
Consider a linear differential-algebraic control system (DACS) of the form
\begin{align}\label{Eq:DAEcontrol1}
\Delta^u:E\dot{x}=Hx+Lu,
\end{align}
where {$x\in\mathscr X\cong \mathbb{R}^n$} is called the ``generalized'' state, $u\in \mathbb {R}^m$ is the vector of control inputs, and where $E\in \mathbb{R}^{l\times n}$, $H\in \mathbb{R}^{l\times n}$ and $L\in \mathbb{R}^{l\times m}$.  A linear  DACS of  {the} form (\ref{Eq:DAEcontrol1})  will be denoted by $\Delta^u_{l,n,m}=(E,H,L)$ or, simply, $\Delta^u$.  \blue{In the case of the control $u$ being absent, the system becomes a}  linear differential-algebraic equation (DAE) $E\dot x=Hx$, which is called regular if $l=n$ and   $sE-H\in \mathbb R^{n\times n}[s]\backslash0$.
A detailed exposition of the theory of linear DAEs and DACSs can be {consulted} in the textbooks \cite{dai1989singular},\cite{campbell1980singular} and the survey paper \cite{lewis1986survey}. Early {results} on linear DAEs can be {traced back to}  two famous canonical forms of {the} matrix pencil $sE-H$ given by Weierstrass \cite{Weierstrass1868} and Kronecker \cite{kronecker1890algebraische}.   The following literature discusses the normal forms and canonical forms of linear DAE systems.  The authors of \cite{helmke1989canonical} proposed a  canonical from for controllable and regular DACSs.  Several forms for regular systems based on their controllability and impulse controllability were  given in  \cite{glusing1990feedback}. In \cite{ozcaldiran1990complete}, a canonical form of general DACSs was discussed.  More recently,  a normal form {based} on  impulse-controllability and impulse-observability of DACSs {was proposed in} \cite{trenn2009normal}, and a quasi-Weierstrass and a quasi-Kronecker triangular/normal forms of DAEs were given in \cite{BERGER20124052} and \cite{Berger2012}, respectively.  In the present paper, we discuss the feedback canonical form \textbf{FBCF} {obtained} in \cite{loiseau1991feedback} (we restate it as Theorem \ref{Cor:FCF} of the present paper) for general linear DACSs, which plays an important role in, e.g. controllability analysis \cite{berger2013controllability},  regularization problems \cite{bunse1999feedback},\cite{berger2015regularization}, pole assignment \cite{loiseau2009pole},\cite{bonilla1993external} and stabilization \cite{berger2014zero}. The \textbf{FBCF} of DACSs is actually an extension of the Kronecker canonical form  of general linear DAEs. Some methods (most are numerical) of transforming a DAE into its Kronecker canonical form can be {found} in \cite{VANDOOREN1979103},\cite{varga1996computation},\cite{beelen1988improved}.

\red{In \cite{chen2021geometric}, we proposed a notion, called explicitation, to connect DAEs with control systems. In the present paper, we will propose a new explicitation procedure  called \emph{explicitation with driving variables} (see Definition \ref{Def:Qvexpl_lin}), and differences and relations of the two explicitation methods are \blue{discussed} in Remark \ref{Rem:dr}. }  Since {the vector of driving variables {$v$} enters} {statically into} the system  ({similarly} as the control input $u$),  we can regard it as another kind of input. More specifically, the \emph{explicitation with driving variables} of a DACS is a class of ODECSs with two kinds of inputs of the form:
\begin{align}\label{Eq:ODEcontrol}
\Lambda^{uv}:\left\lbrace  {\begin{array}{*{20}{l}}
	\dot x=Ax+B^uu+B^vv\\
	y = Cx+D^uu,
	\end{array}}\right.	
\end{align}
where $A\in \mathbb{R}^{n\times n}$, $B^u\in \mathbb{R}^{n\times m}$, $B^v\in \mathbb{R}^{n\times s}$, $C\in \mathbb{R}^{p\times n}$ and $D^u\in \mathbb{R}^{p\times m}$, where $u\in \mathbb R^m$ is the vector of control variables and $v\in\mathbb R^s$ is the vector of driving variables. An ODECS of  {the} form (\ref{Eq:ODEcontrol}) will be denoted by $\Lambda^{uv}_{n,m,s,p}=(A,B^u,B^v,C,D^u)$ or, simply, $\Lambda^{uv}$. Note that although both $u$ and $v$ {may be considered} as inputs of system (\ref{Eq:ODEcontrol}),  we distinguish {them because they play different roles for the system  and, as a consequence, their feedback transformation rules are different} (see Remark \ref{rem:EM-equi}).   	 \red{Observe that we can express an ODECS  $\Lambda^{uv}$ of the form  (\ref{Eq:ODEcontrol}), as a classical ODECS $\Lambda^w=(A,B^w,C,D^w)$ of  {the}  form 
	\begin{align}\label{Eq:ODEcontrol2}
	\Lambda^u:\left\lbrace {\begin{array}{l}
		\dot x=Ax+B^ww\\
		y = Cx+D^ww,
		\end{array}}\right.	
	\end{align}	
	by denoting $w=[u^T,v^T]^T$, $B^w=[ {\begin{matrix}
		B^u&B^v
		\end{matrix}} ]$ and $D^w=[ {\begin{matrix}
		D^u&0
		\end{matrix}} ]$. Throughout the paper, depending on the context, we will use either $\Lambda^{uv}$ or $\Lambda^w$ to denote an ODECS with two kinds of inputs.}
	
	We use Figure \ref{Fig:1} to show the relations of the results of the paper. The purpose of this paper is to find {an} efficient \emph{geometric} way to transform a DACS {$\Delta^u$} into its \blue{feedback canonical form} \textbf{FBCF} via the explicitation procedure.  As we have pointed out, the \textbf{FBCF} is a generalization, on one hand, of the classical Kronecker form (because a DACS is a differential-algebraic equation) and one the other hand, of the Brunovsky canonical form \cite{brunovsky1970classification} (because a DACS is a control system). The explicitation procedure allows us to attach to a DACS a control system $\Lambda^{uv}$ with an output $y$ (defining the algebraic constraint as $y=0$) and to study the double nature of a DACS (differential-algebraic and control-theoretic) simultaneously by analyzing $\Lambda^{uv}$. More specifically, instead of using transformations directly on a DACS, we will {first transform an ODECS $\Lambda^{uv}$, given by the explicitation of our DACS, into its canonical form (called the extended Morse canonical form \textbf{EMCF}, see Theorem~\ref{Thm:EMCF})}. Then by the {relation} between DACSs and ODECSs given in Section \ref{sec:2}, we can easily get the \textbf{FBCF}  from the \textbf{EMCF}. Moreover, inspired by the quasi-Kronecker triangular form of \cite{Berger2012},  we will propose a Morse triangular form \textbf{MTF} (see Proposition~\ref{Pro:MTF}) to transform an ODECS (with one {type of} controls) into its Morse normal form \textbf{MNF} (see Proposition~\ref{Pro:MMNF}).    Note that a procedure of transforming an ODECS $\Lambda^u$ into its \textbf{MCF} was given by Morse \cite{morse1973structural} for $D^u=0$ and by Molinari \cite{molinari1978structural} for the general case $D^u\ne0$. We propose to do it via two intermediate normal forms \textbf{MTF} and \textbf{MNF}. 
\tikzstyle{process} = [rectangle, minimum width=1.2cm, minimum height=0.4cm, text centered, draw=black]
\tikzstyle{arrow} = [thick,double,double distance=1pt,->,>=stealth] 
\begin{figure}[htp!]
\centering
\begin{tikzpicture}[node distance=1.5cm]
	\node[process](A11){$\Delta^u$};
	\node[process, below of = A11, yshift = -0.3cm](A21){$\Lambda^{uv}$};
	\node[process, right of = A11, xshift = 9cm](A14){\textbf{FBCF} \cite{loiseau1991feedback}};
	\node[process, below of = A21, yshift = -0.3cm](A31){$\Lambda^{u}$};
	\node[process, right of = A21, xshift =2cm](A22){\textbf{EMTF}};
	\node[process, right of = A22, xshift =2cm](A23){\textbf{EMNF}};
	\node[process, right of = A23, xshift =2cm](A24){\textbf{EMCF}};
	\node[process, right of = A31, xshift =2cm](A32){\textbf{MTF}};
	\node[process, right of = A32, xshift =2cm](A33){\textbf{MNF}};
	\node[process, right of = A33, xshift =2cm](A34){\textbf{MCF} \cite{morse1973structural},\cite{molinari1978structural}};
	\coordinate (point1) at (-3cm, -6cm);
	\draw [arrow] (A11) -- node[right]{explicitation, see Def.\ref{Def:Qvexpl_lin}}(A21);
	\draw [arrow] (A24) -- node[left]{implicitation, see Sec.\ref{sec:4}}(A14);
	\draw [arrow] (A11) -- node[below]{Theorem \ref{Cor:FCF}}(A14);
	\draw [arrow] (A31) --node[right]{extension}(A21);
	\draw [arrow] (A32) --node[right]{extension}(A22);
	\draw [arrow] (A33) --node[right]{extension}(A23);
	\draw [arrow] (A34) --node[right]{extension}(A24);
	\draw [arrow] (A21) --node[above]{Thm.\ref{Thm:EMTF}}(A22);
	\draw [arrow] (A22) --node[above]{Thm.\ref{Thm:EMNF}}(A23);
	\draw [arrow] (A23) --node[above]{Thm.\ref{Thm:EMCF}}(A24);
	\draw [arrow] (A31) --node[above]{Prop.\ref{Pro:MTF}}(A32);
	\draw [arrow] (A32) --node[above]{Prop.\ref{Pro:MMNF}}(A33);
	\draw [arrow] (A33) --node[above]{\cite{molinari1978structural}}(A34);
\end{tikzpicture} 
\caption{The relations of the results in the  paper} \label{Fig:1}
\end{figure}
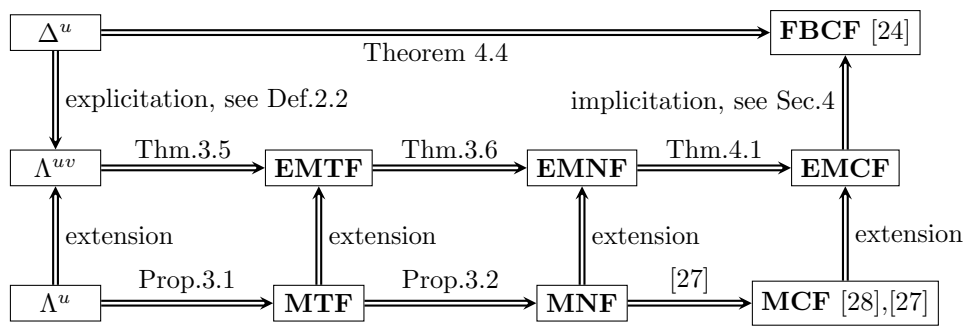 

\blue{We use the following  abbreviations throughout the paper:}\\
\begin{small}
	\begin{tabular}[htp!]{p{0.07\columnwidth}p{0.45\columnwidth}@{\ \   }p{0.07\columnwidth}p{0.3\columnwidth}}
		DAE& differential-algebraic equation& \textbf{MCF}&Morse canonical form\\ 
		DACS& differential-algebraic control system& \textbf{EMTF}&extended Morse triangular form  \\  
	  ODECS&ordinary differential equation control system&\textbf{EMNF}& extended Morse normal form \\
	    \textbf{MTF}& Morse triangular form& \textbf{EMCF}& extended Morse canonical form  \\ 
	    \textbf{MNF}& Morse normal form&  \textbf{FBCF}& feedback canonical form
	\end{tabular}
\end{small}
\\

This paper is organized as follows.  In Section \ref{sec:2}, we introduce the explicitation with driving variables procedure and build geometric connections between DACSs and ODECSs. In Section \ref{sec:3}, we show {a method} of constructing the \textbf{MTF} and the \textbf{MNF} for classical ODECSs of the form (\ref{Eq:ODEcontrol2}), then we extend them to the \textbf{EMTF} and the \textbf{EMNF} for ODECSs  (with two kinds of inputs) of the form (\ref{Eq:ODEcontrol}).  In Section \ref{sec:4}, 
we propose \blue{the} \textbf{EMCF} for ODECSs of the form (\ref{Eq:ODEcontrol}), which allows to  construct the \textbf{FBCF} of DACSs as a corollary \blue{and we} formulate the construction of the  \textbf{FBCF} via the explication procedure as an algorithm. In Section \ref{sec:5}, we give a numerical example to show the  {efficiency} of the algorithm. Section \ref{sec:proofs} and~\ref{sec:7} {contain} proofs and conclusions  of the paper, respectively. The definitions of  geometric {invariant} subspaces for ODECSs and DACSs {are given} in Appendix. \blue{Throughout, we will use the  following notations:}

\begin{small}
	\begin{tabular}[htp!]{p{0.15\columnwidth}p{0.75\columnwidth}}
		$\mathcal C^k$& the class of $k$-times  differentiable functions\\
		$\mathbb{N}$ &  the set of natural numbers with zero and $\mathbb{N}^+=\mathbb{N}\backslash \{0 \}$\\
		${\mathbb{R}^{n \times m}}$ & the set of real valued matrices with $n$ rows and $m$ columns \\
		$Gl\left( {n,\mathbb{R}} \right)$ & the group of nonsingular matrices of $\mathbb{R}^{n \times n}$\\	
		$\ker A$ & the kernel of the map given by {a} matrix $A$\\	
		${\mathop{\rm Im\,}\nolimits} A$ & the image of the map given by {a} matrix $A$\\
		${\rm rank\,}A$ & the rank of {a} matrix $A$\\
		$I_n$ &  {the} identity matrix of size $n\times n$ for $n\in \mathbb{N}^+$\\
		$0_{n\times m}$ &  {the} zero matrix of size $n\times m$ for $n,m\in \mathbb{N}^+$\\
		${A^T}$ & the transpose of {a} matrix $A$\\		
		${A^{-1}}$ & the inverse of {a} matrix $A$\\		
		${A\mathscr B}$ &$\{Ax\,|\,x\in \mathscr B\} $, the image of  {a space} $\mathscr B$ under a map given by a matrix $A$\\
		${A^{-1}\mathscr B}$ &$\{x\in \mathbb R^n\,|\,Ax\in \mathscr B\} $, the preimage of  {a space} $\mathscr B$ under a map given by a matrix $A$\\	
		${A^{-T}}\mathscr B$ &  $(A^T)^{-1}\mathscr B$\\
		$\mathscr A^{\bot}$	& $\{ x\in \mathbb R^n\,|\,\forall a\in \mathscr A: x^Ta=0\}$, the orthogonal complement of  {a subspace $\mathscr A\subseteq \mathbb R^n$} \\
		$ A^{\dagger}$	&  the right inverse of a full row rank matrix $A\in \mathbb R^{n\times m}$, i.e., $AA^{\dagger}=I_n$\\
		$x^{(k)}$& {$k$-th-order} derivative of a {function} $x(t)$
	\end{tabular}
\end{small}   
\section{Explicitation with driving variables for linear DACSs}\label{sec:2}
 {A} solution of $\Delta^u$ is a map $(x(t),u(t)):\mathbb R\rightarrow{\mathscr X}\times \mathbb R^m$ {with $x(t)\in\mathcal C^1$ and $u(t)\in \mathcal C^0$ satisfying} $E\dot x(t)=Hx(t)+Lu(t)$.   {Notice that to some $\mathcal C^0$-controls $u(t)$, there may not correspond any $\mathcal C^1$-solution $x(t)$ because of algebraic relations between $u_i$'s and $x_j$'s present in $\Delta^u$ of the form (\ref{Eq:DAEcontrol1}).} 
\begin{defn}\label{Def:ex-fb-eq}
	Two DACSs $ \Delta^u_{l,n,m}=(E,H,L)$ and  $\tilde \Delta^{\tilde u}_{l,n,m}=(\tilde E,\tilde H,\tilde L) $ are called externally feedback equivalent, shortly ex-fb-equivalent, if there exist matrices $Q\in Gl(l,\mathbb{R})$, $P\in Gl(n,\mathbb{R})$, $F\in \mathbb{R}^{m\times n}$ and $G\in Gl(m,\mathbb{R})$ such that 
	\begin{align}\label{Eq:ex-fb-eq}
	\begin{array}{ccc}
	{\tilde{E}=QEP^{-1},}&\tilde{H}=Q(H+LF)P^{-1},&\tilde{L}=Q LG.
	\end{array}
	\end{align}
	We denote the ex-fb-equivalence of two DACSs as $ \Delta^u\mathop  \sim \limits^{ex-fb} \tilde \Delta^{\tilde u}$.
\end{defn}
 Now we introduce the \emph{explicitation with driving variables} procedure for $ \Delta^u$ as follows.
\begin{itemize}
	\item Denote the rank of $E$ by $q\in \mathbb N$, define $s=n-q$ and $p=l-q$. Then there exists a matrix $Q\in Gl(l,\mathbb R)$ such that 
	$
	QE=\left[ {\begin{smallmatrix}
		E_1\\
		0
		\end{smallmatrix}} \right]$, where $E_1\in \mathbb{R}^{q\times n}$ and ${\rm rank\,} E_1=q$. {Via $Q$,} DACS $\Delta^u$ is ex-fb-equivalent to
	\begin{align}\label{Eq:expl1}
	\left[ {\begin{matrix}
		E_1\\
		0
		\end{matrix}} \right]\dot x=\left[ {\begin{matrix}
		H_1\\
		H_2
		\end{matrix}} \right]x+\left[ {\begin{matrix}
		L_1\\
		L_2
		\end{matrix}} \right]u,
	\end{align} 
	where $QH=\left[ {\begin{smallmatrix}
		H_1\\
		H_2
		\end{smallmatrix}} \right]$, $QL=\left[ {\begin{smallmatrix}
		L_1\\
		L_2
		\end{smallmatrix}} \right]$, and where $H_1\in \mathbb{R}^{q\times n}$, $H_2\in \mathbb{R}^{(l-q)\times n}$, $L_1\in \mathbb{R}^{q\times m}$, $L_2\in \mathbb{R}^{(l-q)\times m}$.
	\item   Consider the differential part of (\ref{Eq:expl1}): \setcounter{equation}{4}
		\begin{subequations}
			\begin{align}\label{Eq:expl1a}
			E_1\dot x=H_1x+L_1u.
			\end{align}
		\end{subequations}
		\setcounter{equation}{5}
The matrix $E_1$ is of full row rank $q$, so let $E^{\dagger}_1\in \mathbb{R}^{n\times q}$ denote its right inverse, i.e., $E_1E^{\dagger}_1=I_q$. Set  $A=E^{\dagger}_1H_1$ and $B^u=E^{\dagger}_1L_1$. In general, $w\in\mathbb R^n$ satisfies the linear equation $E_1w=b$, where $E_1:\mathbb R^{n}\to \mathbb R^q$ is of full row rank $q$, if and only if $w\in E^{\dagger}_1b+\ker E_1$. It follows that $x(t)$ satisfies (\ref{Eq:expl1a}) if and only if 
\begin{equation}\label{Eq:inclusion}
\dot x\in Ax+B^uu+\ker E_1.
\end{equation}
\item Choose a full column rank matrix $B^v\in \mathbb R^{n\times s}$ such that ${\rm Im\,} B^v=\ker E_1=\ker E$ (note {that} the kernels of $E_1$ and $E$ {coincide} since {any} invertible $Q$ preserves the kernel).  {Then the vector $v\in \mathbb R^s$ of driving variables (\blue{see Remark \ref{Rem:dr} for a control-theory interpretation of $v$})   parameterizes the subspace $\ker E_1={\rm Im\,} B^v$ via $B^vv$ and the solutions of the differential inclusion (\ref{Eq:inclusion}), and thus of (\ref{Eq:expl1a}), correspond to the solutions of}
\begin{align}\label{Eq:ODEcontrolnooutput1}
	\dot x= Ax+B^uu+B^vv .
	\end{align} 
\item We claim, see Proposition \ref{Pro:solution}  below, that all solutions of (\ref{Eq:expl1}) (and thus of the original DAE $\Delta^u$) are in one-to-one correspondence with all solutions (corresponding to all {driving variables} $v(t)$) of 
	\begin{align}\label{Eq:expl3}
	\left\{ {\begin{array}{*{20}{l}}
		\dot x=Ax+B^uu+B^vv\\
		0 = Cx+D^uu,
		\end{array}}\right.	
	\end{align}
	where $C=H_2\in \mathbb{R}^{p\times n}$ and $D^u=L_2\in \mathbb{R}^{p\times m}$. Recall {that} a control system of the form (\ref{Eq:ODEcontrol}) {is} denoted by $\Lambda^{uv}_{n,m,s,p}=(A,B^u,B^v,C,D^u)$. It is {immediately to} see that equation (\ref{Eq:expl3}) can be {obtained from the}  ODECS $\Lambda^{uv}$  by setting {the} output $y=0$. In the above way, we attach an ODECS $\Lambda^{uv}$ to a DACS $\Delta^u$.
\end{itemize}
The above procedure of attaching a control system $\Lambda^{u,v}$ to a DACS $\Delta^u$ will be called \emph{explicitation with driving variables} and is formalized as follows.
\begin{defn}\label{Def:Qvexpl_lin}
	Given a DACS $\Delta^u_{l,n,m}=(E,H,L)$, by a $(Q,v)$-explicitation, we will call a control system $\Lambda^{uv}=(A,B^u,B^v,C,D^u)$, with
	$$ 
	A=E^{\dagger}_1H_1, \ \ B^u=E^{\dagger}_1L_1, \ \ {\rm Im\,}B^v=\ker E_1=\ker E, \ \ C=H_2, \ \ D^u=L_2,
	$$ 
	where $$
	QE=\left[ {\begin{smallmatrix}
		E_1\\
		0
		\end{smallmatrix}} \right], \ \  QH=\left[ {\begin{smallmatrix}
		H_1\\
		H_2
		\end{smallmatrix}} \right], \ \ QL=\left[ {\begin{smallmatrix}
		L_1\\
		L_2
		\end{smallmatrix}} \right].$$
	The {class} of {all $(Q,v)$-explicitations} will be called the  \emph{explicitation with driving variables class} or, shortly \emph{explicitation class}, of $\Delta^u$, denoted by $\mathbf{Expl}(\Delta^u)$. If a particular ODECS $\Lambda^{uv}$ belongs to the explicitation class $\mathbf{Expl}(\Delta^u)$, we will write $\Lambda^{uv}\in ~\mathbf{Expl}(\Delta^u)$.
\end{defn}
The definition of the  {explicitation class} $\mathbf{Expl}(\Delta^u)$ suggests {that} a given $\Delta^u$ has many $(Q,v)$-explicitations. Indeed, the construction of $\Lambda^{uv}\in \mathbf{Expl}(\Delta^u)$ is not unique at three stages: there is a freedom in choosing $Q$, $E_1^{\dagger}$, and $B^v$.   We show in the following proposition  that $\mathbf{Expl}(\Delta^u)$ is actually an ODECS defined up to a $v$-feedback transformation, an \emph{output injection} and  \emph{an output {transformation}}, that is,
a class of ODECSs. 

\begin{pro}\label{Pro:LinDAEexpl}
 {Assume that an ODECS $\Lambda^{uv}_{n,m,s,p}=(A,B^u,B^v,C,D^u)$ is a $(Q,v)$-explicitation of a DACS $\Delta^u_{l,n}=(E,H,L)$ corresponding to a choice of invertible matrix $Q$, right inverse $E_1^{\dagger}$, and matrix $B^v$. Then $\tilde \Lambda^{u\tilde v}_{n,m,s,p}=(\tilde A,\tilde B^u,\tilde B^{\tilde v},\tilde C,\tilde D^u)$ is a $(\tilde Q,\tilde v)$-explicitation of  $\Delta^u$ corresponding to a choice of invertible matrix $\tilde Q$, right inverse $\tilde E_1^{\dagger}$, and matrix $\tilde B^{\tilde v}$  if and only if $\Lambda^{uv}$ and $\tilde \Lambda^{u\tilde v}$ are equivalent via a $v$-feedback transformation of the form $v=F_vx+Ru+T^{-1}_v\tilde v$, an output injection $Ky=K(Cx+D^uu)$ and an output multiplication $\tilde y=T_yy$, 
	which map }
\begin{align}\label{Eq:mapDACS}
	\begin{array}{c}
A\mapsto \tilde A=A+KC+B^vF_v, \ \ \ B^u\mapsto  \tilde B^u=B^u+B^vR+KD^u, \ \ \ B^v\mapsto  \tilde B^{\tilde v}=B^vT^{-1}_v,\\
C\mapsto \tilde C=T_yC, \ \ \ D^u\mapsto  \tilde D^u=T_yD^u,
\end{array}  
\end{align}
 where $F_v,K, R,T_v,T_y$ are matrices of appropriate sizes, and $T_v$ and $T_y$ are invertible. 
\end{pro}
{The following proposition shows that solutions of any DACS are in one-to-one correspondence with solutions of its $(Q,v)$-explicitations.}
 \begin{pro}\label{Pro:solution}
Consider  $\Delta^u_{l,n,m}=(E,H,L)$ {and} let an ODECS  $\Lambda^{uv}_{n,m,s,p}=(A,B^u,B^v,C,D^u)$ be {a} $(Q,v)$-explicitation of $\Delta^u$, i.e., $\Lambda^{uv}\in \mathbf{Expl}(\Delta^u)$. Then a curve $(x(t),u(t))$ with $x(t)\in \mathcal C^1$ and $u(t)\in \mathcal C^0$ is a solution of $\Delta^u$ if and only if there exists $v(t)\in \mathcal C^0$ such that $(x(t),u(t),v(t))$ is a solution {of $\Lambda^{uv}$ respecting the output constraints $y =0$,  i.e., a solution of  (\ref{Eq:expl3}).}
 \end{pro} 
 The proofs of Proposition \ref{Pro:LinDAEexpl} and Proposition \ref{Pro:solution} will be given in Section \ref{ProofThm:equivalence}.   
\begin{rem}\label{Rem:dr}
 Notice that the definition of $(Q,v)$-explicitation in the present paper is different  in two aspects from the $(Q,P)$-explicitation of \cite{chen2021geometric} (or see Chapter II of \cite{chen2019geometric}).  First, in this paper we consider the explicitation of DACSs while in \cite{chen2021geometric} we dealt with DAEs (with no controls).  The second difference is that in $(Q,v)$-explicitation, we keep the original generalized state variables $x$ and add new driving variables $v$ while in $(Q,P)$-explicitation of \cite{chen2021geometric}, we look for a partition $(z_1,z_2)=z=Px$ into state- and control- variables. More specifically,
 consider a DACS  $\Delta^u_{l,n,m}=(E,H,L)$, then via two invertible matrices $Q$ and $P$, the system $\Delta^u$ is ex-fb-equivalent with $F=0$ and $G=I_m$ (or ex-equivalent, according to the terminology of \cite{chen2021geometric}, since here we do not use feedback transformation for $\Delta^u$)  to a pure semi-explicit PSE DACS 
		$$\Delta^u_{PSE}:\left[ \begin{matrix}
			I&0\\
			0&0
		\end{matrix}\right] \left[ \begin{matrix}
			\dot z^1\\
			\dot z^2
		\end{matrix}\right]= \left[ \begin{matrix}
			H_1&H_2\\
			H_3&H_4
		\end{matrix}\right]\left[ \begin{matrix}
			z^1\\
			z^2
		\end{matrix}\right]+\left[ \begin{matrix}
		L_1\\
		L_2
		\end{matrix}\right]u,$$
with $z=\left[ \begin{smallmatrix}
	z^1\\
	z^2
	\end{smallmatrix}\right]=\left[ \begin{smallmatrix}
	P_1x\\
	P_2x
	\end{smallmatrix}\right]=Px$, where $P$ is any invertible map such that $\ker P_1=\ker E$. We attach to $\Delta^u_{PSE}$, the control system 
	\begin{align}\label{Eq:z2u}
	\Lambda^{uz^2}:
\left\lbrace \begin{array}{c@{\ }l}
\dot z^1&=H_1z^1+H_2z^2+L_1u\\
y&=H_3z^1+H_4z^2+L_2u,
\end{array} \right.
	\end{align}
 where $z^2\in \mathscr Z_2=\ker E$ is \red{the} vector of free variables (which perform like inputs), $z^1\in \mathscr Z_1$ is the state such that $\mathscr Z_1\oplus\mathscr Z_2=\mathscr X\cong\mathbb R^n$, and $y$ is the output. The system $\Lambda^{uz^2}$ is called a $(Q,P)$-explicitation of $\Delta^u$ and we will write $\Lambda^{uz^2}\in {\rm Expl}(\Delta^u)$, where ${\rm Expl}(\Delta^u)$ is the explicitation class consisting of all $(Q,P)$-explicitations of $\Delta^u$ (clearly, for a given $\Delta^u$, its $(Q,P)$-explicitation is not unique). Now by adding the equation $\dot z^2=v$, we obtain the (dynamical) prolongation  $\mathbf \Lambda^{uv}$ of $\Lambda^{uz^2}$ 
\begin{align}\label{Eq:prol_z2u}
\mathbf{\Lambda}^{uv}:
\left\lbrace \begin{array}{c@{\ }l}
\dot z^1&=H_1z^1+H_2z^2+L_1u\\
\dot z^2&=v\\
y&=H_3z^1+H_4z^2+L_2u,
\end{array} \right.
\end{align}
  which is actually an $(I_l,v)$-explicitation of $\Delta^u_{PSE}$.  We can summarize the relations between the notions of $(Q,P)$-explicitation and $(Q,v)$-explicitation by the following diagram.  
 
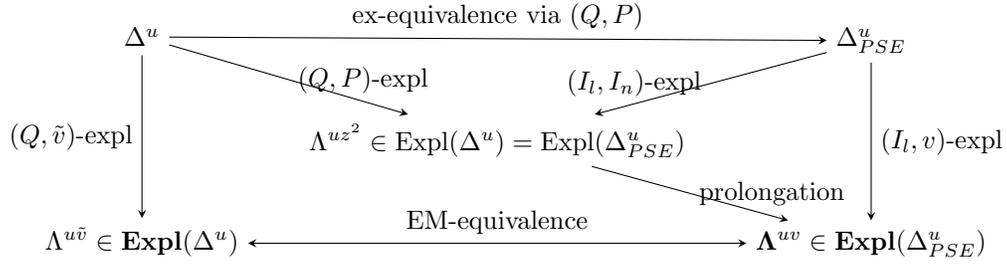
\begin{center}
 \begin{tikzpicture}
\matrix (m) [matrix of math nodes,row sep=2em,column sep=2em,minimum width=1em]
{
	\Delta^u & &\Delta^u_{PSE} \\
	& \Lambda^{uz^2}\in {\rm Expl}(\Delta^u)={\rm Expl}(\Delta^u_{PSE})&\\
	\Lambda^{u\tilde v}\in \mathbf{Expl}(\Delta^u) & &\mathbf{\Lambda}^{uv}\in\mathbf{Expl}(\Delta^u_{PSE})\\};
\path[-stealth]
(m-1-1) edge  node [left] {$(Q,\tilde v)$-expl} (m-3-1)
edge  node [above] {ex-equivalence via $(Q,P)$} (m-1-3)
edge  node [right] {$(Q,P)$-expl} (m-2-2)
(m-3-1.east|-m-3-3) edge [<->,>=stealth]
node [above] {EM-equivalence} (m-3-3)
(m-1-3) edge   node [right] {$(I_l,v)$-expl}(m-3-3)
(m-1-3) edge  node [left] {$(I_l,I_n)$-expl}(m-2-2)
(m-2-2) edge  node [right] {prolongation}(m-3-3);
\end{tikzpicture}
 \end{center}
 The systems $\Delta^u$ and $\Delta^u_{PSE}$ above are DACSs and their ex-equivalence is $(Q,P)$-equivalence of DACSs. The system $\Lambda^{u\tilde v}$ and $\Lambda^{uv}$  \red{at the bottom} are control  systems and their EM-equivalence is the extended Morse equivalence given in Definition \ref{Def:EM-equi}.  Note that the implication that the $(Q,\tilde v)$-explicitation $\Lambda^{u\tilde v}$ of $\Delta^u$ is EM-equivalent to the prolongation system $\mathbf{\Lambda}^{uv}$ is a corollary of Theorem~\ref{Thm:equivalence} below since $\mathbf{\Lambda}^{uv}\in\mathbf{Expl}(\Delta^u_{PSE})$, $\Lambda^{u\tilde v}\in \mathbf{Expl}(\Delta^u)$, and $\Delta^u_{PSE}\mathop\sim\limits^{ex} \Delta^u$. 
\end{rem}
\begin{rem}
 The above explicitation (via driving variables) procedure can also be applied to   more  \blue{general} DAE systems such as DACSs \blue{with} time delays (see e.g., \cite{ascher1995numerical}) and external disturbances (see e.g., \cite{berger2017disturbance}). For example, take a DACS of the following form 
\begin{align}\label{Eq:DACSdd}
E\dot x(t)=Hx(t)+Lu(t)+Tx(t-\tau)+Sd(t), 
\end{align}
where $\tau$ represents a time delay and $d(t)$ is a vector of external disturbances. It is always possible to find an invertible matrix $Q$ such that $E_1$ of $QE=\left[ \begin{smallmatrix}
E_1\\
0
\end{smallmatrix}\right] $ is of full row rank. Then we denote  
$$
QH=\left[ \begin{smallmatrix}
H_1\\
H_2
\end{smallmatrix}\right], ~~ QL=\left[ \begin{smallmatrix}
L_1\\
L_2
\end{smallmatrix}\right], ~~QT=\left[ \begin{smallmatrix}
T_1\\
T_2
\end{smallmatrix}\right], ~~QS=\left[ \begin{smallmatrix}
S_1\\
S_2
\end{smallmatrix}\right].
$$
Choose $B^v$ such that ${\rm Im} B^v=\ker E_1$ and \red{a right inverse $E^{\dagger}_1$ of $E_1$, and} define 
$$
A:=E^{\dagger}_1H_1,\ \ B^u:=E^{\dagger}_1L_1,\ \ M:=E^{\dagger}_1T_1,\ \ N:=E^{\dagger}_1S_1,\ \ C:=H_2,\ \ D^u=:L_2, \ \ J:=T_2,\ \ K:=S_2.
$$
 With the above defined matrices, we can attach the following ODECS \blue{with} time delays and external disturbance to (\ref{Eq:DACSdd}): 
\begin{align}\label{Eq:ODECS}
\left\{ {\begin{aligned}
	\dot x(t)&=Ax(t)+B^uu(t)+B^vv(t)+Mx(t-\tau)+Nd(t)\\
	y(t) &= Cx(t)+D^uu(t)+Jx(t-\tau)+Kd(t).
	\end{aligned}}\right.	
\end{align}
 It is clear that if   DACS (\ref{Eq:DACSdd}) is \emph{not} time-delayed, i.e. $T=0$ (hence $M=0$) and thus  $x(t-\tau)$ is absent, then the results of Proposition \ref{Pro:solution} still hold for (\ref{Eq:DACSdd}) and (\ref{Eq:ODECS}),  meaning that solutions $(x(\cdot),d(\cdot),u(\cdot))$ of (\ref{Eq:DACSdd}) \red{are in} a one-to-one correspondence with solutions $(x(\cdot),u(\cdot),d(\cdot),v(\cdot))$ of (\ref{Eq:ODECS}) with outputs $y=0$. While if \red{a} delayed term is present, the analysis of solutions is  more complicated because  for delayed DAE systems, the existence of solutions  depends on the initial condition $x(t)=\phi(t)$, for  $t\in [-\tau,0]$   (see some studies on solutions of regular delay   DAEs  in  \cite{campbell1980singular,fridman2002stability}).  A particular case is that if the matrices $E$ and $T$ of (\ref{Eq:DACSdd}) satisfy  $\ker E\subseteq \ker T$, implying that there are no delayed free variables in the generalized state $x$, then it is clear that solutions of (\ref{Eq:DACSdd}) and  \red{those} of (\ref{Eq:ODECS}) still have a one-to-one correspondence.  We will not give further discussions on solutions of delayed DAE/DACSs since the purpose of this paper is to study canonical forms  \red{but} the application of the explicitation method to such systems \red{seem to} be an interesting subject for further research. 
\end{rem}
Since the explicitation  of $\Delta^u$ is a class of ODECSs of the form (\ref{Eq:ODEcontrol}), we give the following definition of equivalence for ODECSs of the form (\ref{Eq:ODEcontrol}). This definition is a natural extension of the Morse equivalence (\cite{morse1973structural},  {extended by Molinari \cite{molinari1978structural}, see also \cite{chen2021geometric}}) of classical ODECSs of the form (\ref{Eq:ODEcontrol2}). 
\begin{defn}[extended Morse equivalence and extended Morse transformations]\label{Def:EM-equi}
	Two ODECSs
	$$\Lambda^{uv}_{n,m,s,p}=(A,B^u,B^v,C,D^u), \ \ \ \tilde \Lambda^{\tilde u\tilde v}_{n,m,s,p}=(\tilde A,\tilde B^{\tilde u},\tilde B^{\tilde v},\tilde C,\tilde D^{\tilde u})$$ 
	are called extended Morse equivalent, shortly EM-equivalent, denoted by $\Lambda^{uv} \mathop  \sim \limits^{EM} \tilde \Lambda^{\tilde u\tilde v} $, if there exist matrices  $T_x\in Gl(n,\mathbb R)$, $T_u\in Gl(m,\mathbb R)$, $T_v\in Gl(s,\mathbb R)$, $T_y\in Gl(p,\mathbb R)$, $F_u\in \mathbb{R}^{m\times n}$, $F_v\in \mathbb{R}^{s\times n}$, $R\in \mathbb{R}^{s\times m}$, $K\in \mathbb R^{n\times p}$  such that the system matrices of $\Lambda^{uv}$ and $\tilde \Lambda^{\tilde u\tilde v}$ satisfy:
	\begin{align}\label{Eq:EME}
	\left[ {\begin{smallmatrix}
		{\tilde A}&{\tilde B^{\tilde u}}&{\tilde B^{\tilde v}}\\
		{\tilde C}&{\tilde D^{\tilde u}}&0
		\end{smallmatrix}} \right] = \left[ {\begin{smallmatrix}
		T_x&{T_xK}\\
		0&{T_y}
		\end{smallmatrix}} \right]\left[ {\begin{smallmatrix}
		{A}&{B^u}&{B^v}\\
		{C}&{D^u}&0
		\end{smallmatrix}} \right]\left[ {\begin{smallmatrix}
		{T_x^{-1}}&0&0\\
		{F_uT_x^{-1}}&{T^{-1}_u}&0\\
		(F_v+RF_u)T_x^{-1}&RT^{-1}_u&T^{-1}_v
		\end{smallmatrix}} \right].
	\end{align}
	{An 8-tuple $(T_x,T_u,T_v,T_y,F_u,F_v,R,K)$, acting on the system according to  (\ref{Eq:EME}), will be called an extended Morse transformation and denoted by $EM_{tran}$}.
\end{defn} 
The matrices $T_x$, $T_u$, $T_v$ and $T_y$ are coordinates transformations in the, respectively, state space $\mathscr{X}=\mathbb R^n$, input subspace $\mathscr{U}_u=\mathbb R^m$, input subspace $\mathscr{U}_v=\mathbb R^s$ and,  output space $\mathscr{Y}=\mathbb R^p$,  where $F_u$ defines a state feedback of $u$, $F_v$ and $R$ define a feedback of $v$,  $K$ defines an output injection. 
\begin{rem}\label{rem:EM-equi}
	(i) {An} extended Morse transformation, {whose action is} given by (\ref{Eq:EME}), includes two kinds of feedback transformations:
	\begin{align}\label{Eq:twofb}
	v=F_vx+Ru+T_v^{-1}\tilde v  \ \ \ {\rm and}\ \ \ u=F_ux+T^{-1}_u\tilde u.
	\end{align}
	The vector of driving variables $v$ is ``stronger'' than the original control vector $u$ since {when transforming} $v$ we can use both $u$ and $x$ as feedback, but {when transforming $u$ we are} not allowed to use $v$. {This is expressed by the triangular form of the matrix multiplying on the right  {in} (\ref{Eq:EME}).}	
	
	(ii) Recall the definition of the Morse equivalence and the Morse transformation \cite{morse1973structural} (and their generalization by Molinari \cite{molinari1978structural} for $D^u\ne0$, see also \cite{chen2021geometric}): for two ODECSs $\Lambda^u=(A,B^u,C,D^u)$ and $\tilde \Lambda^{\tilde u}=(\tilde A,\tilde B^{\tilde u},\tilde C,\tilde D^{\tilde u})$ of the form (\ref{Eq:ODEcontrol2}), if 
	\begin{align*}
	\left[ {\begin{smallmatrix}
		{\tilde A}&\tilde B^{\tilde u}\\
		{\tilde C}&\tilde D^{\tilde u}
		\end{smallmatrix}} \right] = \left[ {\begin{smallmatrix}
		T_x&{T_xK}\\
		0&{T_y}
		\end{smallmatrix}} \right]\left[ {\begin{smallmatrix}
		{A}&B^u\\
		{C}&D^u
		\end{smallmatrix}} \right]\left[ {\begin{smallmatrix}
		{T_x^{-1}}&0\\
		{F_uT_x^{-1}}&{T^{-1}_u}\\
		\end{smallmatrix}} \right],
	\end{align*}
	then $\Lambda^u$  and  $\tilde \Lambda^{\tilde u}$ are called Morse equivalent (shortly M-equivalent) and the  Morse transformation $(T_x,T_u,T_y,F_u,K)$ is {denoted} by $M_{tran}$. {Clearly, M-equivalence is an equivalence relation for ODECSs of the form (\ref{Eq:ODEcontrol2}), defined by a 4-tuples $(A,B^u,C,D^u)$ and EM-equivalence is for ODECSs of the form (\ref{Eq:ODEcontrol}), defined by a 5-tuples $(A,B^u,B^v,C,D^u)$.} Observe that if the vector of driving variables $v$ is {of} dimension zero ($B^v$ {is absent}), then the EM-equivalence reduces to the M-equivalence.
	
	(iii) Recall that we can express an ODECS of {the} form $\Lambda^{uv}=(A,B^u,B^v,C,D^u)$ as a {standard} ODECS $\Lambda^w=(A,B^w,C,D^w)$ of the form (\ref{Eq:ODEcontrol2}) {with one type of controls $w$, where $w=[u^T,v^T]^T$}. Now let 
	\begin{align*}
	F_w=\left[ {\begin{smallmatrix}
		F_u\\
		F_v+RF_u
		\end{smallmatrix}} \right], \ \ \ T_w^{-1}=\left[ {\begin{smallmatrix}
		T_u^{-1}&0\\
		RT_u^{-1}&T_v^{-1}
		\end{smallmatrix}} \right],	
	\end{align*}    
	{then} we {conclude} the following equation from (\ref{Eq:EME}) (notice that $T_w$ has a block-triangular structure):
	\begin{align}\label{Eq:l-f.bsys-equi1}
	\left[ {\begin{smallmatrix}
		{\tilde A}&\tilde B^w\\
		{\tilde C}&\tilde D^w
		\end{smallmatrix}} \right] = \left[ {\begin{smallmatrix}
		T_x&{T_xK}\\
		0&{T_y}
		\end{smallmatrix}} \right]\left[ {\begin{smallmatrix}
		{A}&B^w\\
		{C}&D^w
		\end{smallmatrix}} \right]\left[ {\begin{smallmatrix}
		{T_x^{-1}}&0\\
		{F_wT_x^{-1}}&{T^{-1}_w}\\
		\end{smallmatrix}} \right],
	\end{align}
	which is exactly the expression of the M-equivalence {for systems $\Lambda^w$} ({compare Remark \ref{rem:EM-equi}(ii) above}). It implies that the EM-equivalence can be expressed as {a} form of the M-equivalence with a triangular {matrix} $T_w$ (input coordinates transformation matrix). This triangular form {is a consequence of} two kinds of feedback transformation  shown in equation (\ref{Eq:twofb}).
\end{rem}
Now we give the main result of this subsection:
\begin{thm}\label{Thm:equivalence}
	Consider two DACSs  $ \Delta^u_{l,n,m}=(E,H,L)$ and  $\tilde \Delta^{\tilde u}_{l,n,m}=(\tilde E,\tilde H,\tilde L) $ as well as two ODECSs $\Lambda^{uv}_{n,m,s,p}=(A,B^u,B^v,C,D^u)$ and $\tilde \Lambda^{\tilde u \tilde v}_{n,m,s,p}=(\tilde A,\tilde B^{\tilde u},\tilde B^{\tilde v},\tilde C,\tilde D^{\tilde u})$ {satisfying}  $\Lambda^{uv}\in \mathbf{Expl}(\Delta^u) $ and $\tilde \Lambda^{\tilde u \tilde v}\in \mathbf{Expl}(\tilde \Delta^{\tilde u})$. Then, $ \Delta^u\mathop  \sim \limits^{ex-fb} \tilde \Delta^{\tilde u}$ if and only if  $\Lambda^{uv} \mathop  \sim \limits^{EM} \tilde \Lambda^{\tilde u \tilde v}$.
\end{thm}
The proof will be given in Section \ref{ProofThm:equivalence}. In {the} Appendix, we recall the definitions of  geometric subspaces for DACSs and ODECSs. More specifically, for a DACS $\Delta^u$, we recall the augmented Wong sequences $\mathscr V_i$ and $\mathscr W_i$, together with {$\hat{\mathscr W}_i$}  (see \cite{berger2013controllability},\cite{lewis1992tutorial});  for an ODECS $\Lambda^w$, we recall the subspaces sequences $\mathcal V_i$ and $\mathcal W_i$ (see \cite{wonham1970decoupling},\cite{wonham1974linear},\cite{basile1992controlled}), {whose limits are controlled and conditioned invariant subspaces, respectively,} and we introduce  {a} subspaces sequence {$\hat{\mathcal W}_i$}. 
\begin{pro}\label{Pro:subspacesrelation}
	Given $\Delta^u_{l,n,m}=(E,H,L)$ and $\Lambda^{uv}_{n,m,s,p}=(A,B^u,B^v,C,D^u)$  (or equivalently, $\Lambda^{w}_{n,m+s,p}=(A,B^w,C,D^w)$), consider the subspaces  $\mathscr V_i$, $\mathscr W_i$, $\hat{\mathscr W}_i$ of  $\Delta^u$, given by Definition \ref{Def:augWong} {and} the subspaces $\mathcal V_i$, $\mathcal W_i$, $\hat{\mathcal W}_i$ of $\Lambda^{w}$, given by Lemma \ref{Lem:invarsubspaceseq} in the Appendix. Assume that $\Lambda^{uv}\in \mathbf{Expl}(\Delta^u)$. {Then we have  for $i\in \mathbb N$, $$\mathscr V_i(\Delta^u)=\mathcal V_i(\Lambda^{w}), \ \ \ \  \mathscr W_i(\Delta^u)={\mathcal W_i}(\Lambda^{w}),$$ and for $i\in \mathbb N^+$,}   $$\hat{\mathscr W}_i(\Delta^u)=\hat{\mathcal W}_i(\Lambda^{w}) .$$
\end{pro}
The proof will be given in Section \ref{ProofPro:subspacesrelation}. Note that Theorem \ref{Thm:equivalence} and Proposition \ref{Pro:subspacesrelation} are  fundamental results for the {remaining part of the paper}. The above proposition shows \red{the} importance of the notion of $(Q,v)$-explicitation. Namely, the augmented Wong sequences of any DACS $\Delta^u$ and the invariant subspaces of its $(Q,v)$-explicitation $\Lambda^{w}$ coincide (in particular, they are subspaces of the same generalized state-space $\mathscr X$). If we use the $(Q,P)$-explicitation, we need to establish relations between subspaces of different spaces $\mathscr X$ and $\mathscr Z_1$ (see Remark \ref{Rem:dr}).  {Our purpose} is to find the \textbf{FBCF} of {DACSs} via explicitation.  {We have proven in Theorem \ref{Thm:equivalence} that the ex-fb-equivalence for  DACSs  corresponds to the EM-equivalence for their explicitations.}  Thus rather than transforming a DACS $\Delta^u$ directly into its \textbf{FBCF} under ex-fb-equivalence,  we will look for the canonical form {for} $\Lambda^{uv}\in \mathbf{Expl}(\Delta^u)$ under EM-equivalence. 
\section{The Morse triangular form and its extension}\label{sec:3}
In the beginning of this section, we show that the normal form {given} in \cite{molinari1978structural} (called Morse normal form \textbf{MNF} in the present paper) for  {the} 4-tuple ODECS $ \Lambda^{u}$, given by equation (\ref{Eq:ODEcontrol2}), can be constructed  through a Morse triangular form \textbf{MTF} {that we propose}. Although the constructed normal form is the same as the one in \cite{molinari1978structural},  we will  {provide} explicit transformations with the help of the invariant subspaces given in Lemma \ref{Lem:invarsubspaceseq} of the Appendix, which makes the normalizing procedure {simple} and transparent.
\begin{pro}[Morse triangular form \textbf{MTF}]\label{Pro:MTF}
	For an ODECS $\Lambda^u_{n,m,p}=(A,B^u,C,D^u)$, consider the subspaces $\mathcal V^*$, $\mathcal U_u^*$, $\mathcal W^*$, $\mathcal Y^*$ given by Definition \ref{Def:invariantsub} of {the} Appendix. Choose full rank matrices $T_s^1\in\mathbb{R}^{n\times n_1}$, $T_s^2\in\mathbb{R}^{n\times n_2}$, $T_s^3\in\mathbb{R}^{n\times n_3}$, $T_s^4\in\mathbb{R}^{n\times n_4}$, $T_i^1\in\mathbb{R}^{m\times m_1}$, $T_i^3\in\mathbb{R}^{m\times m_3}$,
	$T_o^3\in\mathbb{R}^{p\times p_3}$, $T_o^4\in\mathbb{R}^{p\times p_4}$ such that 
	\begin{align*}
	\begin{array}{ll}
	{{\mathop{\rm Im\,}\nolimits} {T_s^1} = {\mathcal V^*} \cap {\mathcal W^*},}&{{\mathcal V^*} \cap {\mathcal W^*} \oplus {\mathop{\rm Im\,}\nolimits} {T_s^2} = {\mathcal V^*},}\\
	{{\mathcal V^*} \cap {\mathcal W^*} \oplus {\mathop{\rm Im\,}\nolimits} {T_s^3} = {\mathcal W^*},}&{\left( {{\mathcal V^*} + {\mathcal W^*}} \right) \oplus {\mathop{\rm Im\,}\nolimits} {T_s^4} =\mathscr X= {\mathbb R^n},}\\
	{{\mathop{\rm Im\,}\nolimits} {T_i^1} = {\mathcal U_u^*},}&{\mathop{\rm Im\,}\nolimits} {T_i^3}  \oplus {\rm Im\,}  {T_i^1}={ \mathscr U_u}={\mathbb R^m},\\
	{{\mathop{\rm Im\,}\nolimits} {T_o^3} = {\mathcal Y^*},}&{\mathop{\rm Im\,}\nolimits} {T_o^4}\oplus {\rm Im\,}  {T_o^3}={{ \mathscr Y}}={\mathbb R^p},
	\end{array}
	\end{align*}
	where $n=n_1+n_2+n_3+n_4$, $m=m_1+m_3$, $p=p_3+p_4$. Then 
\begin{align}\label{Eq:mtftm}
	\begin{array}{c}
	T_s = [ {\begin{smallmatrix}
			{{T_s^1}}&{{T_s^2}}&{{T_s^3}}&{{T_s^4}}
	\end{smallmatrix}} ]^{-1}\in Gl(n,\mathbb{R}), ~~ T_i = [ {\begin{smallmatrix}
			{{T_i^1}}&{{T_i^3}}
	\end{smallmatrix}} ]^{-1}\in Gl(m,\mathbb{R}), ~~ T_o= [ {\begin{smallmatrix}
			{{T_o^3}}&{{T_o^4}}
	\end{smallmatrix}} ]^{-1}\in Gl(p,\mathbb{R}), 
\end{array}
\end{align}
{and} there exist matrices $F_{MT}\in \mathbb{R}^{m\times n }$ and $K_{MT}\in \mathbb{R}^{n\times p }$ such that the Morse transformation $M_{tran}=(T_s,T_i,T_o,F_{MT},K_{MT})$ brings $\Lambda^u$ {into $\tilde \Lambda^{\tilde u}=M_{tran}(\Lambda^u)$, represented in the Morse triangular form \textbf{MTF}, that is given by $\tilde \Lambda^{\tilde u}=(\tilde A,\tilde B^{\tilde u},\tilde C,\tilde D^{\tilde u})$, where} 
		\begin{align}\label{Eq:MTF}
	\left[ \begin{smallmatrix}
	{\tilde{A}}&\tilde B^{\tilde u}\\
	{\tilde{C}}&{\tilde D^{\tilde u}}
	\end{smallmatrix} \right] = \left[ {\begin{smallmatrix}
		{\tilde A_1}&{\tilde A_1^2}&{\tilde A_1^3}&{\tilde A_1^4}&\vline& {\tilde B_1}&\tilde B_1^2\\
		0&{\tilde A_2}&0&{\tilde A_2^4}&\vline& 0&0\\
		0&0&{\tilde A_3}&{\tilde A_3^4}&\vline& 0&{\tilde B_3}\\
		0&0&0&{\tilde A_4}&\vline& 0&0\\
		\hline  
		0&0&{\tilde C_3}&\tilde C_3^4&\vline& 0&{\tilde D_3} \Tstrut\\
		0&0&0&{\tilde C_4}&\vline& 0&0
		\end{smallmatrix}} \right].
	\end{align}
In the above \textbf{MTF}, the pair $(\tilde A_1, \tilde B_1)$ is controllable,
	the pair $(\tilde C_4,\tilde A_4)$ is observable and the 4-tuple $(\tilde A_3,\tilde B_3,\tilde C_3,\tilde D_3)$ is prime \footnote{ {A control system is called prime if it is M-equivalent to $m_3$ independent chains of integrators, see \cite{morse1973structural} and \cite{molinari1978structural}.}}. 
\end{pro}
The proof {is}  given in Section \ref{ProofPro:MTF}. In {the} next proposition, we describe a way to transform the above \textbf{MTF} into  {the} Morse normal form \textbf{MNF}, {which is a further simplification of the \textbf{MTF}}. We will use the same notations as {in} Proposition \ref{Pro:MTF}.
\begin{pro}[Morse normal form \textbf{MNF}]\label{Pro:MMNF} 
	There exists a feedback transformation matrix $F_{MN}\in \mathbb{R}^{m\times n }$, an output injection matrix $K_{MN}\in \mathbb{R}^{n\times p}$ and a state space coordinate transformation matrix $T_{MN}\in Gl(n,\mathbb R)$, {which can be chosen by \textbf{MNF} Algorithm \ref{Alg:1} below,} such that the Morse transformation $M_{tran}=(T_{MN},I_u,I_y,F_{MN},K_{MN})$ brings $\tilde \Lambda^{\tilde u}$ of Proposition \ref{Pro:MTF},  {given by (\ref{Eq:MTF})},  into $\bar \Lambda^{\bar u}=M_{tran}(\tilde \Lambda^{\tilde u})$, represented in the Morse normal form \textbf{MNF},  that is given by $\bar \Lambda^{\bar u}=(\bar A,\bar B^{\bar u},\bar C,\bar D^{\bar u})$, where
	\begin{align}\label{Morsenorform}
	\left[ {\begin{smallmatrix}
		{\bar{A}}&{\bar B^{\bar u}}\\
		{\bar{C}}&{\bar D^{\bar u}}
		\end{smallmatrix}} \right] = \left[ {\begin{smallmatrix}
		{\bar A_1}&0&0&0&\vline& {\bar B_1}&0\\
		0&{\bar A_2}&0&0&\vline& 0&0\\
		0&0&{\bar A_3}&0&\vline& 0&{\bar B_3}\\
		0&0&0&{\bar A_4}&\vline& 0&0\\
		\hline
		0&0&{\bar C_3}&0&\vline& 0&{\bar D_3} \Tstrut\\
		0&0&0&{\bar C_4}&\vline& 0&0
		\end{smallmatrix}} \right].
	\end{align}
	In the above \textbf{MNF}, the pair $(\bar A_1, \bar B_1)$ is controllable, the pair $(\bar C_4,\bar A_4)$ is observable, and the 4-tuple $(\bar A_3,\bar B_3,\bar C_3,\bar D_3)$ is prime.  
\end{pro}
The proof of Proposition \ref{Pro:MMNF} will be given in Section \ref{ProofPro:MMNF} and in that proof, we will use the construction of transformation matrices $F_{MN}$, $K_{MN}$ and $T_{MN}$,  which is formulated in the following algorithm.
\begin{MNFalg}\label{Alg:1}
	Step 1: \blue{Given the matrix (\ref{Eq:MTF}),} choose $F_{MN}$ and $K_{MN}$: 
	\begin{align*}
	F_{MN}=\left[ \begin{smallmatrix}
	F_{MN}^1&0&0&0\\
	0&0&{F_{MN}^2}&F_{MN}^3
	\end{smallmatrix} \right],\ \ \ {K_{MN}} = \left[ {\begin{smallmatrix}
		{K_{MN}^1}&0\\
		0&0\\
		{K_{MN}^2}&0\\
		0&{K_{MN}^3}
		\end{smallmatrix}} \right],
	\end{align*}
	such that  {the spectra of $\bar A_1$, $\bar A_2$, $\bar A_3$ and $\bar A_4$  \blue{defined by} the equation below are mutually disjoint}
	(notice that ${F_{MN}}$ and ${K_{MN}}$ preserve the zero blocks of $\tilde \Lambda^{\tilde u}=(\tilde A,\tilde B^{\tilde u},\tilde C,\tilde D^{\tilde u})$): 
	\begin{align*}
	\left[ {\begin{smallmatrix}
		{{I_n}}&{{K_{MN}}}\\
		0&{{I_p}}
		\end{smallmatrix}} \right]\left[ {\begin{smallmatrix}
		{\tilde A}&{\tilde B^{\tilde u}}\\
		{\tilde C}&{\tilde D^{\tilde u}}
		\end{smallmatrix}} \right]\left[ {\begin{smallmatrix}
		{{I_n}}&0\\
		{{F_{MN}}}&{{I_m}}
		\end{smallmatrix}} \right] = \left[ {\begin{smallmatrix}
		{\bar A_1}&{\bar A_1^2}&{\bar A_1^3}&{\bar A_1^4}&\vline& {\bar B_1}&\bar B_1^2\\
		0&{\bar A_2}&0&{\bar A_2^4}&\vline& 0&0\\
		0&0&{\bar A_3}&{\bar A_3^4}&\vline& 0&{\bar B_3}\\
		0&0&0&{\bar A_4}&\vline& 0&0\\
		\hline
		0&0&{\bar C_3}&\bar C_3^4&\vline& 0&{\bar D_3}\Tstrut\\
		0&0&0&{\bar C_4}&\vline& 0&0
		\end{smallmatrix}} \right].
	\end{align*}
	Step 2: {Find} matrices $T^1_{MN}$, $T^2_{MN}$, $T^3_{MN}$, $T^4_{MN}$, $T^5_{MN}$ via the following (constrained) Sylvester equations:
	\begin{align}\label{Eq:Slyvester}
	\begin{array}{c}
	{\bar A_1{T^1_{MN}} - {T^1_{MN}}\bar A_2 =  - \bar A_1^2}, \ \ \ \ \ {\bar A_2{T^4_{MN}} - {T^4_{MN}}\bar A_4 =  - \bar A_2^4},\\{\bar A_1{T^3_{MN}} - {T^3_{MN}}\bar A_4 =  - \bar A_1^4 - \bar A_1^2{T^4_{MN}} - \bar A_1^3{T^5_{MN}}};
	\end{array} 
	\end{align} 
	\begin{align}\label{Eq:CSlyvester}
	\begin{array}{ll}
	\bar A_1{T^2_{MN}} - {T^2_{MN}}\bar A_3 =  - \bar A_1^3,& T^2_{MN}\bar B_3=-\bar B_1^2,\\  	{\bar A_3{T^5_{MN}} - {T^5_{MN}}\bar A_4 =  - \bar A_3^4},&{\bar C_3T^5_{MN}=-\bar C_4}.
	\end{array}
	\end{align}
	Step 3: Set
	\begin{align*}
	{{T_{MN}} = {{\left[ {\begin{smallmatrix}
					I&{{T^1_{MN}}}&{{T^2_{MN}}}&{{T^3_{MN}}}\\
					0&I&0&{{T^4_{MN}}}\\
					0&0&I&{{T^5_{MN}}}\\
					0&0&0&I
					\end{smallmatrix}} \right]}^{ - 1}}}.
	\end{align*}
\end{MNFalg}
\begin{rem}
It is not surprising that Propositions \ref{Pro:MTF} and \ref{Pro:MMNF} describe  results {similar} to those of Theorem 2.3 and Theorem 2.6 of \cite{Berger2012},  as we have shown in \cite{chen2021geometric} that there are direct connections between the geometric subspaces (the Wong sequences) of a DAE  $\Delta:E\dot x=Hx$ and {invariant subspaces of} a control system $\Lambda=(A,B,C,D)\in {\rm Expl}(\Delta)$. There are, however, differences between Propositions \ref{Pro:MTF} and \ref{Pro:MMNF} and {results} of \cite{Berger2012}. In particular, in Theorem 2.6 of \cite{Berger2012}, one has to solve generalized Sylvester equations, while in Propositions \ref{Pro:MMNF} we use (constrained) Sylvester equations. \red{In addition, our transformations  differ from those proposed in the original paper \cite{morse1976system} and \cite{molinari1978structural} for the \textbf{MNF} and seem to be more transparent and explicit.}
\end{rem}
Recall that the explicitation of {a} DACS $\Delta^u$ is a class of ODECSs with two kinds of inputs of the form (\ref{Eq:ODEcontrol}). In the following theorems, we will extend the results {of} Proposition \ref{Pro:MTF} and \ref{Pro:MMNF} to ODECSs with two kinds of inputs. 
\begin{thm}[extended Morse triangular form \textbf{EMTF}]\label{Thm:EMTF} 
	For a DACS $$\Lambda^{uv}_{n,m,s,p}=(A,B^u,B^v,C,D^u),$$ there exists an extended Morse transformation $EM_{tran}$ {bringing $\Lambda^{uv}$ into $EM_{tran}(\Lambda^{uv})=\tilde \Lambda^{\tilde u \tilde v}$ represented in the extended Morse triangular form \textbf{EMTF}, that is given by $\tilde \Lambda^{\tilde u \tilde v}_{n,m,s,p}=(\tilde A,\tilde B^{\tilde u},\tilde B^{\tilde v},\tilde C,\tilde D^{\tilde u})$, where}
	\begin{align}\label{Eq:EMTF}
	\left[ {\begin{smallmatrix}
		{\tilde{A}}&{\tilde B^{\tilde u}}&\tilde B^{\tilde v}\\
		{\tilde{C}}&{\tilde D^{\tilde u}}&0
		\end{smallmatrix}} \right] = \left[ {\begin{smallmatrix}
		{\tilde A_{1}}&{\tilde A_{12}}&{\tilde A_{13}}&{\tilde A_{14}}&\vline& {\tilde B^{\tilde u}_{1}}&{\tilde B^{\tilde u}_{12}}&\vline& {\tilde B^{\tilde v}_1}&{\tilde B^{\tilde v}_{12}}\\
		0&{\tilde A_{2}}&0&{\tilde A_{24}}&\vline& 0&0&\vline& 0&0\\
		0&0&{\tilde A_{3}}&{\tilde A_{34}}&\vline& 0&{\tilde B^{\tilde u}_{3}}&\vline& 0&{\tilde B^{\tilde v}_{3}}\\
		0&0&0&{\tilde A_{4}}&\vline& 0&0&\vline& 0&0\\
		\hline
		0&0&{\tilde C_{3}}&{\tilde C_{34}}&\vline& 0&{\tilde D^{\tilde u}_{3}}&\vline& 0&0 \Tstrut\\
		0&0&0&{\tilde C_{4}}&\vline& 0&0&\vline& 0&0
		\end{smallmatrix}} \right].
	\end{align}
 {In the above \textbf{EMTF}, }the pair $(\tilde A_{1}, \tilde B^{\tilde w}_{1})$ is controllable, where $\tilde B^{\tilde w}_{1}=[\tilde B^{\tilde u}_{1},\tilde  B^{\tilde v}_{1}]$; the pair $(\tilde C_{4},\tilde A_{4})$ is observable ;
	the 4-tuple $(\tilde A_{3},\tilde B^{\tilde w}_{3},\tilde C_{3},\tilde D^{\tilde w}_{3})$ is prime, where $\tilde B^{\tilde w}_{3}=[\tilde B^{\tilde u}_{3},\tilde  B^{\tilde v}_{3}]$, $\tilde D^{\tilde w}_{3}=[\tilde D^{\tilde u}_{3},0]$. 
\end{thm}
\begin{thm}[extended Morse normal form \textbf{EMNF}]\label{Thm:EMNF}
	For $\tilde \Lambda^{\tilde u \tilde v}_{n,m,s,p}=(\tilde A,\tilde B^{\tilde u},\tilde B^{\tilde v},\tilde C,\tilde D^{\tilde u})$ {in the \textbf{EMTF}, as given by} Theorem \ref{Thm:EMTF}, there exists an extended Morse transformation $EM_{tran}$ {bringing $\tilde \Lambda^{\tilde u\tilde v}$ into $\bar \Lambda^{\bar u \bar v}=EM_{tran}(\tilde \Lambda^{\tilde u\tilde v})$ represented in the extended Morse normal form \textbf{EMNF}, that is given by $\bar \Lambda^{\bar u \bar v}_{n,m,s,p}=(\bar A,\bar B^{\bar u},\bar B^{\bar v},\bar C,\bar D^{\bar u})$, }
	where
	\begin{align}\label{Eq:EMNF}
	\left[ {\begin{smallmatrix}
		{\bar{A}}&{\bar B^{\bar u}}&\bar B^{\bar v}\\
		{\bar{C}}&{\bar D^{\bar u}}&0
		\end{smallmatrix}} \right] = \left[ {\begin{smallmatrix}
		{\bar A_{1}}&0&0&0&\vline& {\bar B^{\bar u}_{1}}&0&\vline& {\bar B^{\bar v}_{1}}&0\\
		0&{\bar A_{2}}&0&0&\vline& 0&0&\vline& 0&0\\
		0&0&{\bar A_{3}}&0&\vline& 0&{\bar B^{\bar u}_{3}}&\vline& 0&{\bar B^{\bar v}_{3}}\\
		0&0&0&{\bar A_{4}}&\vline& 0&0&\vline& 0&0\\
		\hline
		0&0&{\bar C_{3}}&0&\vline& 0&{\bar D^{\bar u}_{3}}&\vline& 0&0\Tstrut\\
		0&0&0&{\bar C_{4}}&\vline& 0&0&\vline& 0&0
		\end{smallmatrix}} \right].
	\end{align}
 {In the above \textbf{EMNF},}  the pair $(\bar A_{1}, \bar B^{\bar w}_{1})$ is controllable, where $\bar B^{\bar w}_{1}=[\bar B^{\bar u}_{1},\bar  B^{\bar v}_{1}]$; the pair $(\bar C_{4},\bar A_{4})$ is observable; the 4-tuple $(\bar A_{3},\bar B^{\bar w}_{3},\bar C_{3},\bar D^{\bar w}_{3})$ is prime, where $\bar B^{\bar w}_{3}=[\bar B^{\bar u}_{3},\bar  B^{\bar v}_{3}]$, $\tilde D^{\bar w}_{3}=[\tilde D^{\bar u}_{3},0]$. 
\end{thm}
 {The proofs of Theorem \ref{Thm:EMTF} and Theorem \ref{Thm:EMNF} are given in} Section \ref{ProofThm:EMTFNF}.
\section{From the extended Morse normal form \textbf{EMNF} to the feedback canonical form  \textbf{FBCF}}\label{sec:4} 
We show that, with a suitable choice of an extended Morse transformation for each subsystem in the \textbf{EMNF} of Theorem \ref{Thm:EMNF}, we can bring the \textbf{EMNF} into {the} extended Morse canonical form \textbf{EMCF}. Below the upper indices refer to: $c$ to controllable, $nn$ to non-controllable and non-observable, $p$ to prime, $o$ to observable.
If an ODECS $\Lambda_{EM}^{uv}=(A_{EM},B^u_{EM},B^v_{EM},C_{EM},D^u_{EM})$ is in the \textbf{EMCF}, then the matrices $A_{EM},B^u_{EM},B^v_{EM},C_{EM},D^u_{EM}$  are   given by
\begin{align}\label{Eq:EMCF}
	\left[ {\begin{smallmatrix}
		{ A_{EM}}&{B_{EM}^u}&{B_{EM}^v}\\  
		{ C_{EM}}&{ D_{EM}^u}&0
		\end{smallmatrix}}\right] =\left[ \begin{smallmatrix}
	{ A^{cu}}&0&0&0&0&0&\vline& {B^{cu}}&0&\vline& 0&0\\
	0&A^{cv}&0&0&0&0&\vline& 0&0&\vline& { B^{cv}}&0\\
	0&0&A^{nn}&0&0&0&\vline& 0&0&\vline& 0&0\\
	0&0&0&A^{pu}&0&0&\vline& 0&{B^{pu}}&\vline& 0&0\\
	0&0&0&0&A^{pv}&0&\vline& 0&0&\vline& 0&B^{pv}\\
	0&0&0&0&0&A^o&\vline& 0&0&\vline& 0&0\\ 
	\hline
	\\ 
	 0&0&0&C^{pu}&0&0&\vline& 0&D^{pu}&\vline& 0&0\\
	0&0&0&0&C^{pv}&0&\vline& 0&0&\vline& 0&0\\
	0&0&0&0&0&C^o&\vline& 0&0&\vline& 0&0
	\end{smallmatrix}\right],
	\end{align}
with {the matrices and their invariants of the following form:}
\begin{enumerate}[(i)]
	\item 
	$A^{cu}={\rm diag}\{ A^{cu}_{\epsilon_1 },..., A^{cu}_{\epsilon_a }\}$,  $B^{cu}={\rm diag}\{ B^{cu}_{\epsilon_1 },..., B^{cu}_{\epsilon_a }\}$, $A^{cv}={\rm diag}\{ A^{cv}_{\bar {\epsilon}_b },..., A^{cv}_{\bar {\epsilon}_b }\}$,\\ $B^{cv}={\rm diag}\{B^{cv}_{\bar {\epsilon}_1 },..., B^{cv}_{\bar {\epsilon}_b }\}$, where
	\begin{align*}
	A_{\epsilon}^{cu} = \left[ {\begin{smallmatrix}
		0&I_{\epsilon-1}\\
		0&0
		\end{smallmatrix}} \right] \in {\mathbb{R}^{\epsilon  \times \epsilon }},\  B_\epsilon ^{cu} = \left[ {\begin{smallmatrix}
		0\\
		1
		\end{smallmatrix}} \right]\in \mathbb{R}^{\epsilon},\ A_{\bar {\epsilon}}^{cv} = \left[ {\begin{smallmatrix}
		0&I_{\bar \epsilon-1}\\
		0&0
		\end{smallmatrix}} \right] \in {\mathbb{R}^{\bar {\epsilon}  \times \bar {\epsilon} }},\  B_{\bar {\epsilon}} ^{cv} = \left[ {\begin{smallmatrix}
		0\\
		1
		\end{smallmatrix}} \right]\in \mathbb{R}^{\bar {\epsilon}}.
	\end{align*}
	The integers $\epsilon_1,...,\epsilon_a\in \mathbb{N^+}$ are the controllability indices of $( A^{cu},B^{cu})$, the integers $\bar {\epsilon}_1,...,\bar {\epsilon}_b\in \mathbb{N^+}$ are the controllability indices of $( A^{cv},B^{cv})$.
	\item $A^{nn}\in \mathbb R^{n_2\times n_2}$ is unique up to similarity {and can always be put in the real Jordan form}.
	\item Both the  4-tuple $(A^{pu}, B^{pu}, C^{pu}, D^{pu})$ and the triple $(A^{pv},B^{pv},C^{pv})$ are {prime, and thus} controllable and observable. That is,
	\[\left[ {\begin{smallmatrix}
		A^{pu}&B^{pu}\\
		C^{pu}&D^{pu}
		\end{smallmatrix}} \right] = \left[ {\begin{smallmatrix}
		{{\hat  A^{pu}}}&\vline& {\hat  B^{pu}}&0\\
		\hline
		{\hat  C^{pu}}&\vline& 0&0\\
		0&\vline& 0&{{I_\delta }}
		\end{smallmatrix}} \right],\]
	where $\left[ {\begin{smallmatrix}
		{{\hat  A^{pu}}}&{{\hat  B^{pu}}}\\
		{{\hat  C^{pu}}}&0
		\end{smallmatrix}} \right]$ is square and invertible and $\delta={\rm rank\,} \hat  D^{pu}\in \mathbb N $, and the matrices 
	$$
	\begin{array}{ccc}
	\hat  A^{pu}={\rm diag}\{\hat  A^{pu}_{\sigma_{1} },...,\hat  A^{pu}_{\sigma_c }\}, & \hat  B^{pu}={\rm diag}\{\hat  B^{pu}_{{\sigma_{1} }},...,\hat  B^{pu}_{{\sigma_c }}\}, & \hat  C^{pu}={\rm diag}\{\hat  C^{pu}_{{{\sigma_{1} } }},...,\hat  C^{pu}_{{{\sigma_{c} } }}\}, \\ A^{pv}={\rm diag}\{ A^{pv}_{{\bar {\sigma}}_{1} },..., A^{pv}_{{\bar {\sigma}}_d}\}, & B^{pv}={\rm diag}\{ B^{pv}_{{{\bar {\sigma}}_1}},...,B^{pv}_{{{\bar {\sigma}}_d }}\}, & C^{pv}={\rm diag}\{ C^{pv}_{{{{\bar {\sigma}}_{1} } }},..., C^{pv}_{{{{\bar {\sigma}}_{d} } }}\},
	\end{array}
	$$ 
	where
	\begin{align*}
	\begin{array}{lll}
	\hat  A^{pu}_{{{\sigma } }} = \left[ {\begin{smallmatrix}
		0&I_{\sigma-1}\\
		0&0
		\end{smallmatrix}} \right] \in {\mathbb{R}^{\sigma  \times \sigma }},& \hat  B^{pu}_{{{\sigma } }}= \left[ {\begin{smallmatrix}
		0\\
		1
		\end{smallmatrix}} \right]\in \mathbb{R}^{\sigma \times 1 },& \hat  C^{pu}_{{{\sigma } }} = \left[ {\begin{smallmatrix}
		1&0
		\end{smallmatrix}} \right] \in {\mathbb{R}^{1 \times\sigma}},\\ 
	A^{pv}_{{{{\bar {\sigma}} } }} = \left[ {\begin{smallmatrix}
		0&I_{\bar {\sigma}-1}\\
		0&0
		\end{smallmatrix}} \right] \in {\mathbb{R}^{{\bar {\sigma}}  \times {\bar {\sigma}} }},& B^{pv}_{{{{\bar {\sigma}} } }}= \left[ {\begin{smallmatrix}
		0\\
		1
		\end{smallmatrix}} \right]\in \mathbb{R}^{{\bar {\sigma}} \times 1 },& C^{pv}_{{{{\bar {\sigma}} } }} = \left[ {\begin{smallmatrix}
		1&0
		\end{smallmatrix}} \right] \in {\mathbb{R}^{{1\times\bar {\sigma}}}}.
	\end{array}
	\end{align*}
	The integers $\sigma_{1},...,\sigma_c\in \mathbb{N}^+$ are the controllability indices of the pair $(\hat  A^{pu},\hat  B^{pu})$ and they are equal to the observability indices of the pair $(\hat  C^{pu}, \hat  A^{pu}) $. The integers ${\bar {\sigma}}_{1},...,{\bar {\sigma}}_d\in \mathbb{N^+}$ are the controllability indices of the pair $(A^{pv}, B^{pv})$ and they are equal to the observability indices of the pair $( C^{pv},  A^{pv}) $.
	\item 
	$ A^o={\rm diag}\{ A^o_{\eta_1 },...,A^o_{\eta_e }\}$, $ C^o={\rm diag}\{ C^o_{\eta_1 },..., C^o_{\eta_e }\}$, 
	where 
	$$
	\begin{smallmatrix}
	A^o_{\eta } = \left[ {\begin{smallmatrix}
		0&I_{\eta-1}\\
		0&0
		\end{smallmatrix}} \right] \in {\mathbb{R}^{\eta  \times \eta}}, \ \ \ \   C^o_{\eta }= \left[
	{\begin{smallmatrix}
		1&0
		\end{smallmatrix}} \right] \in {\mathbb{R}^{1 \times \eta}}.
	\end{smallmatrix}
	$$
	The integers $\eta_1,...,\eta_e\in \mathbb{N}^+$ are the observability indices of the pair $( C^o, A^o)$.
\end{enumerate}
\begin{thm}[extended Morse canonical form \textbf{EMCF}]\label{Thm:EMCF}   For any $$\Lambda^{uv}=\Lambda^{uv}_{n,m,s,p}=(A,B^u,B^v,C,D^u),$$ there exists an extended Morse transformation $EM_{tran}$ bringing $\Lambda^{uv}$ into $$\Lambda_{EM}^{uv}=(A_{EM},B^u_{EM},B^v_{EM},C_{EM},D^u_{EM})=EM_{tran}(\Lambda^{uv}),$$ represented \blue{by} the extended Morse canonical form \textbf{EMCF}.
\end{thm}
The proof will be given in Section \ref{ProofThm:EMCF}.  Throughout if we only consider the differential equation of (\ref{Eq:ODEcontrol}) (meaning (\ref{Eq:ODEcontrol}) without {the} output $y$), we denote it as $ \Lambda^{uv}_{n,m,s}=(A,B^u,B^v)$.  
Now we introduce the \emph{driving variables $v$-reduction} and \emph{implicitation}  (compare \cite{chen2021geometric}) to reduce the driving variables $v$ and implicit the \textbf{EMCF} to a DACS.
\begin{defn}[$v$-reduction and implicitation]\label{Def:red_impl}
For	a control system $\Lambda^{uz^2}$ and	its  prolongation  $\mathbf \Lambda^{uv}$, given by (\ref{Eq:z2u}) and (\ref{Eq:prol_z2u}), \blue{respectively},	
 {the inverse operation of prolongation will be called the $v$-reduction, that is, the $v$-reduction of} $	\mathbf{\Lambda}^{uv}$ is ${\Lambda}^{uz^2}$. For an ODECS ${\Lambda}^{uz^2}$, the implicitation of ${\Lambda}^{uz^2}$ is a DACS ${\rm Impl}({\Lambda}^{uz^2})$ constructed by setting the output $y=0$, that is, 
	$${\rm Impl}({\Lambda}^{uz^2}):\left[ \begin{matrix}
	I&0\\
	0&0
	\end{matrix}\right] \left[ \begin{matrix}
	\dot z^1\\
	\dot z^2
	\end{matrix}\right]= \left[ \begin{matrix}
	H_1&H_2\\
	H_3&H_4
	\end{matrix}\right]\left[ \begin{matrix}
	z^1\\
	z^2
	\end{matrix}\right]+\left[ \begin{matrix}
	L_1\\
	L_2
	\end{matrix}\right]u.$$
\end{defn}
\begin{rem}\label{rem:impl}
	If $\Delta^u={\rm Impl}(\Lambda^{uz^2})$, where $\Lambda^{ uz^2}$ is the $v$-reduction of $\mathbf{\Lambda}^{uv}$, then $\mathbf{\Lambda}^{uv}\in \mathbf{Expl}(\Delta^u)$.
\end{rem} 
Then with the help of the above \emph{$v$-reduction and implicitation} procedure, we can regard the feedback canonical form \textbf{FBCF} for DACSs of {the} form $\Delta^u_{l,n,m}=(E,H,L)$ given in \cite{loiseau1991feedback} as a corollary of Theorem \ref{Thm:EMCF}. 
In the following, in order to save space and simplify notations, we denote
$$
\begin{array}{l}
K_i=\left[ {\begin{smallmatrix}
	0&I_{i-1}
	\end{smallmatrix}} \right]\in \mathbb{R}^{(i-1)  \times i }, \ L_i=\left[ {\begin{smallmatrix}
	I_{i-1}&0
	\end{smallmatrix}} \right]\in \mathbb{R}^{(i-1)  \times i}, \  N_i=\left[ {\begin{smallmatrix}
	0&0\\
	I_{i-1}&0
	\end{smallmatrix}} \right]\in \mathbb{R}^{i  \times i}, \  e_{i}=\left[ {\begin{smallmatrix}
	0\\
	1
	\end{smallmatrix}} \right]\in \mathbb{R}^{i} ,
\end{array}
$$
\blue{ where $\beta=(\beta_1,\ldots,\beta_k)$, $|\beta|=\beta_1+\cdots+\beta_k$, and}
$$
\begin{array}{ll}
N_{\beta}={\rm diag} \left\lbrace N_{\beta_1},\ldots,N_{\beta_k}\right\rbrace \in \mathbb R^{|\beta|\times |\beta|}& K_{\beta}={\rm diag}\left\lbrace  K_{\beta_1},\ldots,K_{\beta_k}\right\rbrace \in \mathbb R^{(|\beta|-k)\times |\beta|},\\ L_{\beta}={\rm diag}\left\lbrace L_{\beta_1},\ldots,L_{\beta_k}\right\rbrace \in \mathbb R^{(|\beta|-k)\times |\beta|},&\mathcal E_{\beta}={\rm diag}\left\lbrace  e_{\beta_1},\ldots,e_{\beta_k}\right\rbrace \in \mathbb R^{|\beta|\times k},
\end{array} 
$$
\begin{thm}[feedback canonical form of  {DACSs}  \cite{loiseau1991feedback}]\label{Cor:FCF}
	Any  {DACS} $\Delta^u_{l,n,m}=(E,H,L)$ is ex-fb-equivalent  to the following feedback canonical form \textbf{FBCF}: 
		\begin{align*}
		\left( \left[ {\begin{smallmatrix}
			{{I_{\left| \epsilon'  \right|}}}&0&0&0&0&0\\
			0&{{L_{\bar \epsilon' }}}&0&0&0&0\\
			0&0&{{I_{{n_\rho }}}}&0&0&0\\
			0&0&0&{K_{\sigma'}^T}&0&0\\
			0&0&0&0&{{N_{\bar \sigma' }}}&0\\
			0&0&0&0&0&{L_{\eta' }^T}
			\end{smallmatrix}} \right],\left[ \begin{smallmatrix}
		{N_{\epsilon'} ^T}&0&0&0&0&0\\
		0&{{K_{\bar \epsilon'}}}&0&0&0&0\\
		0&0&{{A_\rho }}&0&0&0\\
		0&0&0&{L_{\sigma'} ^T}&0&0\\
		0&0&0&0&I_{\left| {\bar \sigma' } \right|}&0\\
		0&0&0&0&0&K_{\eta'}^T
		\end{smallmatrix} \right], \left[ {\begin{smallmatrix}
			{{\mathcal E_{\epsilon'} }}&0&0\\
			0&0&0\\
			0&0&0\\
			0&{{\mathcal E_{\sigma'} }}&0\\
			0&0&0\\
			0&0&0
			\end{smallmatrix}} \right]\right) ,
		\end{align*} 
	where  $\epsilon'=(\epsilon'_1,\ldots,\epsilon'_{a'})\in (\mathbb N^+)^{a'}$, $\bar \epsilon' =(\bar \epsilon'_1,\ldots,\bar \epsilon'_{b'})\in (\mathbb N^+)^{b'} $, $\sigma'=(\sigma'_1,\ldots,\sigma'_{c'})\in (\mathbb N^+)^{c'}$, $\bar \sigma'=(\bar \sigma'_1,\ldots,\bar \sigma'_{d'})\in (\mathbb N^+)^{d'}$, $\eta'=(\eta'_1,\ldots,\eta'_{e'})\in (\mathbb N^+)^{e'}$ are multi-indices {and} the matrix $A_{\rho}$ is {given} up to similarity ( {and can always be put into real Jordan form}).
\end{thm}  
\begin{rem}\label{rem:5}
	(i) The above \blue{theorem} of the \textbf{FBCF} of DACSs is a corollary of Theorem \ref{Thm:EMCF}. Indeed, for any DACS $\Delta^u=(E,H,L)$, we can {construct} an ODECS $\Lambda^{uv}\in \mathbf{Expl}(\Delta^u)$. Then, by Theorem \ref{Thm:EMCF}, we have $\Lambda^{uv}\mathop  \sim \limits^{EM} \textbf{EMCF}$.  {It is not hard to see that the \textbf{FBCF} is the implicitation of the $v$-reduction (see Definition \ref{Def:red_impl}) of the \textbf{EMCF}.  A crucial observation is that $\textbf{EMCF}\in \mathbf{Expl}(\textbf{FBCF})$ (see Remark~\ref{rem:impl}).} Thus, by Theorem \ref{Thm:equivalence}, we {conclude}  $\Delta^u\mathop  \sim \limits^{ex-fb}\textbf{FBCF}$ (since $\Lambda^{uv}\mathop  \sim \limits^{EM} \textbf{EMCF}$).
	
	(ii) There exists a perfect correspondence between the six subsystems of the \textbf{EMCF} and {their counterparts} of the \textbf{FBCF}. Morse specifically, 
		$$
		\begin{small}
	\begin{array}{lll}
	(A^{cu},B^{cu})\leftrightarrow(I_{|\epsilon'|},N_{\epsilon'}^T,\mathcal E_{\epsilon'}),&(A^{cv},B^{cv})\leftrightarrow(L_{\bar {\epsilon}'}, K_{\bar {\epsilon}'},0),&A^{nn}\leftrightarrow(I_{n_{\rho}},A_{\rho}),\\
	(A^{pu},B^{pu},C^{pu},D^{pu})\leftrightarrow(K_{\sigma'}^T,L_{\sigma'}^T,\mathcal E_{\sigma'}),&(A^{pv},B^{pv},C^{pv})\leftrightarrow(N_{\bar {\sigma}'}, I_{|\bar {\sigma}'|},0),& (C^{o},A^{o})\leftrightarrow(L_{\eta'}^T,K_{\eta'}^T,0).
	\end{array}
	\end{small}
	$$
		 

 (iii) Since the \textbf{FBCF}  is the implicitation of {the $v$-reduction of} the \textbf{EMCF}, it is easy to observe that the indices  \blue{of} the \textbf{FBCF}  and \textbf{EMCF}  have the following relations:  $a=a'$ and $\epsilon_k=\epsilon'_k$ for $k=1,\ldots,a$;  $b=b'$ and $\bar \epsilon_k=\bar \epsilon'_k$ for $k=1,\ldots,b$; $n_2=n_{\rho}$ and $A^{nn} \approx A_{\rho}$ ( {similar matrices});
   $c+\delta=c'$ and $\sigma'_1=\sigma'_2=\cdots=\sigma'_{\delta}=1$, $\sigma'_{\delta+1}=\sigma_1+1$, $\sigma'_{\delta+2}=\sigma_2+1$, $\ldots$, $\sigma'_{\delta+c}=\sigma_c+1$; moreover, $d=d'$ and $\bar \sigma_k=\bar \sigma'_k$ for  $k=1,\ldots,d$;  $e=e'$ and $\eta_k+1=\eta'_k$ for $k=1,\ldots,e$.   
\end{rem}
 In an  algorithm below, we summarize how to construct the \textbf{FBCF} {for} a given DACS $\Delta^u_{l,n,m}=(E,H,L)$ based on the explicitation procedure.
 \setcounter{algorithm}{5}
 \begin{small}
\begin{algorithm}[htp!]\caption{the construction of the \textbf{FBCF} for linear DACSs via the explicitation}\label{Alg:2}
	\begin{algorithmic}[1]  
		\Require{ Consider a DACS $\Delta^u_{l,n.m}=(E,H,L)$ with  $E\in \mathbb{R}^{l\times n}$, $H\in \mathbb{R}^{l\times n}$, $L\in \mathbb{R}^{l\times m}$.} 
		\Stepa{  \blue{Construct} an ODECS  $\Lambda^{uv}$ such that $\Lambda^{uv}\in \mathbf{Expl}(\Delta^u)$ by Definition \ref{Def:Qvexpl_lin}:}\\
		Find $Q$ \blue{such that} $E_1$ of 	$QE=\left[ {\begin{smallmatrix}
				E_1\\
				0
		\end{smallmatrix}} \right]$ is of full row rank, denote $QH=\left[ {\begin{smallmatrix}
				H_1\\
				H_2
		\end{smallmatrix}} \right]$, $QL=\left[ {\begin{smallmatrix}
				L_1\\
				L_2
		\end{smallmatrix}} \right]$;\\ 
		Set	$ A=E^{\dagger}_1H_1$, $B^u=E^{\dagger}_1L_1$, $C=H_2$,   $D^u=L_2$ and find $B^v$ such that  ${\rm Im\,}B^v=\ker E_1=\ker E$;\\
		Set $\Lambda^{uv}=(A,B^u,B^v,C,D^v)$, \blue{then} we have $\Lambda^{uv}\in \mathbf{Expl}(\Delta^u)$.
		\Stepb{Find  $EM_{tran}$ \blue{such that} $\tilde \Lambda^{\tilde u \tilde v}=EM_{tran}(\Lambda^{uv})$ is in the \textbf{EMTF} by Theorem \ref{Thm:EMTF}:}\\
		Calculate the subspaces $\mathcal V^*$, $\mathcal U_u^*$, $\mathcal W^*$, $\mathcal Y^*$ for $\Lambda^{w}=\Lambda^{uv}$ by Lemma \ref{Lem:invarsubspaceseq};\\
		Construct $T_s$, $T_o$ by (\ref{Eq:mtftm}) and $T_w$ by (\ref{Eq:emtftm});\\
		Find $K_{MT}=T^{-1}_sKT_o$ and $F_{MT}=T^{-1}_iFT_s$ by (\ref{Eq:F12}) and (\ref{Eq:K12});\\Set  $T_x=T_s$,  $T_y=T_o$,  $F_w=F_{MT}$, $K_w=K_{MT}$ and
		$M_{trans}=(T_x,T_w,T_y,F_{w},K_{w})$, \blue{then} we have $\tilde \Lambda^{\tilde w}=M_{trans}(\Lambda^{w})$ is in the \textbf{MTF}, i.e., $\exists$ $EM_{tran}$: $\tilde \Lambda^{\tilde u\tilde v}=EM_{tran}(\Lambda^{uv})$ is in the \textbf{EMTF}.
		\Stepc{Find   $EM_{tran}$ \blue{such that} $\bar \Lambda^{\bar u\bar v}=EM_{tran}(\tilde \Lambda^{\tilde u \tilde v})$ is in the \textbf{EMNF} by Theorem \ref{Thm:EMNF}:}\\
		Construct $F_{MN}$, $K_{MN}$, $T_{MN}$ for $\tilde \lambda^{\tilde w}$ by the \textbf{MNF Algorithm \ref{Alg:1}}.\\
		Set $M_{tran}=(T_{MN},I_u,I_y,F_{MN},K_{MN})$, \blue{then}  we have $\bar \Lambda^{\bar w}=M_{tran}(\tilde \Lambda^{\tilde w})$ is in the \textbf{MNF}, i.e., $\exists$~$EM_{tran}$ \blue{such that} $\bar \Lambda^{\bar u\bar v}=EM_{tran}(\tilde \Lambda^{\tilde u\tilde v})$ is in the \textbf{EMNF}. 
		\Stepd{By the procedure shown in the proof of Theorem \ref{Thm:EMCF}, bring $\bar \Lambda^{\bar u\bar v}$ into the \textbf{EMCF}.}
		\Stepe{By  Definition \ref{Def:red_impl}, find the implicitation of the $v$-reduction  of $\bar \Lambda^{\bar u\bar v}$, denoted by $\bar \Delta^{\bar u}$.}
		\Result{ $\bar \Delta^{\bar u}$ is {in} the \textbf{FBCF} and $\Delta^u\mathop  \sim \limits^{ex-fb}\bar \Delta^{\bar u}$.} 
	\end{algorithmic}
\end{algorithm}
 \end{small}

\section{Example}\label{sec:5}
In this section, we {illustrate} the construction {of} Algorithm \ref{Alg:2} by an example taken from \cite{Berger2012}. Consider the following mathematical model of an electrical circuit (see Fig. 1.1 of \cite{Berger2012}), which is a DACS of the form $E\dot x=Hx+Lu$:
\begin{align*}
\left[  {\transparent{0.4}{
		\begin{smallmatrix}
		0&0&0&0&{\tr L}&0&0&0&0&0&0&0&0&0\\
		0&0&\tr -Ca&\tr Ca&0&0&0&0&0&0&0&0&0&0\\
		0&0&0&0&0&0&0&0&0&0&0&0&0&0\\
		0&0&0&0&0&0&0&0&0&0&0&0&0&0\\
		0&0&0&0&0&0&0&0&0&0&0&0&0&0\\
		0&0&0&0&0&0&0&0&0&0&0&0&0&0\\
		0&0&0&0&0&0&0&0&0&0&0&0&0&0\\
		0&0&0&0&0&0&0&0&0&0&0&0&0&0\\
		0&0&0&0&0&0&0&0&0&0&0&0&0&0\\
		0&0&0&0&0&0&0&0&0&0&0&0&0&0\\
		0&0&0&0&0&0&0&0&0&0&0&0&0&0\\
		0&0&0&0&0&0&0&0&0&0&0&0&0&0\\
		0&0&0&0&0&0&0&0&0&0&0&0&0&0\\
		\end{smallmatrix}}}\right]\dot x= \left[ {\transparent{0.4}\begin{smallmatrix}
	\tr 1&0&0&0&0&0&0&0&0&0&0&0&0&0\\
	0&0&0&0&0&0&0&0&0&0&0&0&\tr 1&0\\
	0&\tr -1&0&0&0&0&0&\tr R_G&0&0&0&0&0&0\\
	0&\tr 1&\tr -1&0&0&0&0&0&\tr R_F&0&0&0&0&0\\
	0&0&0&\tr -1&0&0&0&0&0&\tr R&0&0&0&0\\
	0&0&0&0&0&\tr 1&0&0&0&0&0&0&0&0\\
	0&0&0&0&0&0&\tr 1&0&0&0&0&0&0&0\\
	\tr 1&\tr -1&0&0&0&0&0&0&0&0&0&0&0&0\\
	0&0&0&0&\tr -1&\tr -1&0&0&0&0&0&0&0&0\\
	0&0&0&0&0&0&\tr 1&\tr 1&\tr -1&0&0&0&0&0\\
	0&0&0&0&0&0&0&0&\tr -1&0&\tr 1&\tr -1&\tr 1&0\\
	0&0&0&0&0&0&0&0&0&\tr 1&0&0&\tr 1&\tr 1\\
	0&0&\tr -1&0&0&0&0&0&0&0&0&0&0&0\\
	\end{smallmatrix}}\right]x+\left[ {\transparent{0.4} \begin{smallmatrix}
	0&0\\
	0&0\\
	0&0\\
	0&0\\
	0&0\\
	0&0\\
	0&0\\
	0&0\\
	\tr 1&0\\
	0&0\\
	0&0\\
	0&0\\
	0&\tr 1\\
	\end{smallmatrix}}\right]\left[ \begin{smallmatrix}
I\\V
\end{smallmatrix} \right],
\end{align*}
where $u=[I,V]^T$ is the control vector, $L,Ca,R,R_G,R_F$ are real scalars ({all assumed to be nonzero}). In \cite{Berger2012}, {only the matrix pencil $sE-H$ is transformed into a quasi-Kronecker form. Below, we will transform \footnote{The calculations
		of the invariant subspaces and the transformation matrices in the example are implemented by Matlab  and the source code is available on the webpage of the first
		author.} the whole DACS into its \textbf{FBCF}  via Algorithm \ref{Alg:2}. 

Step 1: Find an ODECS $\Lambda^{uv}\in \mathbf{Expl}(\Delta^u)$, which {we take as}
$$
\Lambda^{uv}:\left\lbrace \begin{array}{l@{\ }l@{\ }}
\dot x&=\left[ {\transparent{0.4}\begin{smallmatrix}
	0&0&0&0&0&0&0&0&0&0&0&0&0&0\\
	0&0&0&0&0&0&0&0&0&0&0&0&0&0\\
	0&0&0&0&0&0&0&0&0&0&0&0&0&0\\
	0&0&0&0&0&0&0&0&0&0&0&0&{\tr 1/Ca}&0\\
	\tr 1/L&0&0&0&0&0&0&0&0&0&0&0&0&0\\
	0&0&0&0&0&0&0&0&0&0&0&0&0&0\\
	0&0&0&0&0&0&0&0&0&0&0&0&0&0\\
	0&0&0&0&0&0&0&0&0&0&0&0&0&0\\
	0&0&0&0&0&0&0&0&0&0&0&0&0&0\\
	0&0&0&0&0&0&0&0&0&0&0&0&0&0\\
	0&0&0&0&0&0&0&0&0&0&0&0&0&0\\
	0&0&0&0&0&0&0&0&0&0&0&0&0&0\\
	0&0&0&0&0&0&0&0&0&0&0&0&0&0\\
	0&0&0&0&0&0&0&0&0&0&0&0&0&0
	\end{smallmatrix}}\right] x+\left[ {\transparent{0.4}\begin{smallmatrix}
	0&0\\
	0&0\\
	0&0\\
	0&0\\
	0&0\\
	0&0\\
	0&0\\
	0&0\\
	0&0\\
	0&0\\
	0&0\\
	0&0\\
	0&0\\
	0&0
	\end{smallmatrix}}\right]u+\left[ {\transparent{0.4}\begin{smallmatrix}
	\tr 1&0&0&0&0&0&0&0&0&0&0&0\\
	0&\tr 1&0&0&0&0&0&0&0&0&0&0\\
	0&0&\tr 1&0&0&0&0&0&0&0&0&0\\
	0&0&\tr 1&0&0&0&0&0&0&0&0&0\\
	0&0&0&0&0&0&0&0&0&0&0&0\\
	0&0&0&\tr 1&0&0&0&0&0&0&0&0\\
	0&0&0&0&\tr 1&0&0&0&0&0&0&0\\
	0&0&0&0&0&\tr 1&0&0&0&0&0&0\\
	0&0&0&0&0&0&\tr 1&0&0&0&0&0\\
	0&0&0&0&0&0&0&\tr 1&0&0&0&0\\
	0&0&0&0&0&0&0&0&\tr 1&0&0&0\\
	0&0&0&0&0&0&0&0&0&\tr 1&0&0\\
	0&0&0&0&0&0&0&0&0&0&\tr 1&0\\
	0&0&0&0&0&0&0&0&0&0&0&\tr 1
	\end{smallmatrix}}\right]v\\    \\
y&=\left[ {\transparent{0.4}\begin{smallmatrix}
	0&\tr -1&0&0&0&0&0&\tr R_G&0&0&0&0&0&0\\
	0&\tr 1&\tr -1&0&0&0&0&0&\tr R_F&0&0&0&0&0\\
	0&0&0&\tr -1&0&0&0&0&0&\tr R&0&0&0&0\\
	0&0&0&0&0&\tr 1&0&0&0&0&0&0&0&0\\
	0&0&0&0&0&0&\tr 1&0&0&0&0&0&0&0\\
	\tr 1&\tr -1&0&0&0&0&0&0&0&0&0&0&0&0\\
	0&0&0&0&\tr -1&\tr -1&0&0&0&0&0&0&0&0\\
	0&0&0&0&0&0&\tr 1&\tr 1&\tr -1&0&0&0&0&0\\
	0&0&0&0&0&0&0&0&\tr -1&0&\tr 1&\tr -1&\tr 1&0\\
	0&0&0&0&0&0&0&0&0&\tr 1&0&0&\tr 1&\tr 1\\
	0&0&\tr -1&0&0&0&0&0&0&0&0&0&0&0\\
	\end{smallmatrix}}\right]x+\left[ {\transparent{0.4} \begin{smallmatrix}
	0&0\\
	0&0\\
	0&0\\
	0&0\\
	0&0\\
	0&0\\
	\tr 1&0\\
	0&0\\
	0&0\\
	0&0\\
	0&\tr 1\\
	\end{smallmatrix}}\right]u.
\end{array}\right. 
$$
Step 2: Calculate the subspaces  $\mathcal V^*$, $\mathcal U_w^*$, $ \mathcal U_v^*$, $\mathcal W^*$,  $\mathcal Y^*$ of  $\Lambda^{w}=(A,B^w,C,D^w)$ by Lemma \ref{Lem:invarsubspaceseq} of {the} Appendix. They are $\mathcal W^*=\mathscr X=\mathbb R^{14}$, $\mathcal Y^*=\mathscr Y=\mathbb R^{11}$ and
$$
\mathcal V^*\!=\!{\rm Im\,}\left[ {\transparent{0.4} \begin{smallmatrix}
	\tr R_G&    0& 0&  0&  0\\
	\tr R_G&    0& 0&  0&  0\\
	\tr R_F + R_G&    0& 0&  0&  0\\
	0&   \tr R& 0&  0&  0\\
	0&    0& \tr 1&  0&  0\\
	0&    0& 0&  0&  0\\
	0&    0& 0&  0&  0\\
	\tr 1&    0& 0&  0&  0\\
	\tr 1&    0& 0&  0&  0\\
	0& \tr 1& 0&  0&  0\\
	0&    0& 0& \tr  1&  0\\
	0&    0& 0&  0&  \tr 1\\
	\tr 1&    0& 0& \tr -1&  \tr 1\\
	\tr -1& \tr -1& 0& \tr  1&\tr -1
	\end{smallmatrix}}\right], \ \mathcal U_w^*\!=\!{\rm Im\,}\left[ {\transparent{0.4} \begin{smallmatrix}
	0&  0&  0\\
	0&  0&  0\\
	\tr R*R_G&  0&  0\\
	\tr R*R_G&  0&  0\\
	\tr R*( R_F + R_G)&  0&  0\\
	0&  0&  0\\
	0&  0&  0\\
	\tr R&  0&  0\\
	\tr R&  0&  0\\
	\tr R_F + R_G&  0&  0\\
	0&  \tr 1&  0\\
	0&  0&  \tr 1\\
	\tr R& \tr -1& \tr 1\\
	\tr-(R + R_F + R_G)& \tr 1& \tr -1
	\end{smallmatrix}}\right], \  \mathcal U^*_v\!=\!{\rm Im\,}\left[ {\transparent{0.4} \begin{smallmatrix}
	\tr R*R_G&  0&  0\\
	\tr R*R_G&  0&  0\\
	\tr R*(R_F + R_G)&  0&  0\\
	0&  0&  0\\
	0&  0&  0\\
	\tr R&  0&  0\\
	\tr R&  0&  0\\
	\tr R_F + R_G&  0&  0\\
	0&  \tr 1&  0\\
	0&  0&  \tr 1\\
	\tr R& \tr -1& \tr 1\\
	\tr-(R + R_F + R_G)& \tr 1& \tr -1
	\end{smallmatrix}}\right].
$$
By the proof of Theorem \ref{Thm:EMTF} and Proposition \ref{Pro:MTF}, we can choose the following transformation matrices: $T_y=I_{11}$, $K_{MT}=0_{14\times 11}$,
$$
T_s\!=\!\left[ {\transparent{0.4} \begin{smallmatrix}  
	\tr 1& 0& 0& 0& 0& 0& 0& 0& 0& 0&  0&  0& 0& 0\\
	0& 0& 0& \tr 1& 0& 0& 0& 0& 0& 0&  0&  0& 0& 0\\
	0& 0& 0& 0& \tr 1& 0& 0& 0& 0& 0&  0&  0& 0& 0\\
	0& 0& 0& 0& 0& 0& 0& 0& 0& 0&  \tr 1&  0& 0& 0\\
	0& 0& 0& 0& 0& 0& 0& 0& 0& 0&  0&  \tr 1& 0& 0\\
	\tr -1& \tr 1& 0&0& 0& 0& 0& 0& 0& 0&  0&  0& 0& 0\\
	\tr -\frac{R_F + R_G}{R_G} & 0& \tr 1& 0& 0& 0& 0& 0& 0& 0& 0& 0& 0& 0\\
	0& 0& 0& 0& 0& 1& 0& 0& 0& 0&  0&  0& 0& 0\\
	0& 0& 0& 0& 0& 0& 1& 0& 0& 0&  0&  0& 0& 0\\
	\tr -1/R_G& 0& 0& 0& 0& 0& 0& \tr 1& 0& 0&  0&  0& 0& 0\\
	\tr -1/R_G& 0& 0& 0& 0& 0& 0& 0& \tr 1& 0&  0&  0& 0& 0\\
	0& 0& 0& \tr \frac{-1}{R}& 0& 0& 0& 0& 0& \tr 1&  0&  0& 0& 0\\
	\tr -1/R_G& 0& 0& 0& 0& 0& 0& 0& 0& 0& \tr 1& \tr -1&\tr  1& 0\\
	\tr 1/R_G& 0& 0& \tr \frac{-1}{R}& 0& 0& 0& 0& 0& 0& \tr -1&  \tr 1& 0& \tr 1	\end{smallmatrix}}\right], \ T_w\!=\!\left[ {\transparent{0.4} \begin{smallmatrix} 
	\tr	1& 0& 0& 0& 0& 0& 0& 0& 0& 0&  0&  0& 0& 0\\
	0& \tr 1& 0& 0& 0& 0& 0& 0& 0& 0&  0&  0& 0& 0\\
	0& 0& \tr 1& 0& 0& 0& 0& 0& 0& 0&  0&  0& 0& 0\\
	0& 0&0& 0& 0& 0& 0& 0& 0& 0& \tr 1&  0& 0& 0\\
	0& 0& 0& 0& 0& 0& 0& 0& 0& 0&  0& \tr 1& 0& 0\\
	0& 0& \tr -1& \tr 1& 0& 0& 0& 0& 0& 0&  0&  0& 0& 0\\
	0& 0& \tr -(R_F + R_G)/R_G& 0& \tr 1& 0& 0& 0& 0& 0&  0&  0& 0& 0\\
	0& 0& 0& 0& 0&\tr 1& 0& 0& 0& 0&  0&  0& 0& 0\\
	0& 0& 0& 0& 0& 0& \tr 1& 0& 0& 0&  0&  0& 0& 0\\
	0& 0& \tr -1/R_G& 0& 0& 0& 0& \tr 1& 0& 0&  0&  0& 0& 0\\
	0& 0& \tr -1/R_G& 0& 0& 0& 0& 0& \tr 1& 0&  0&  0& 0& 0\\
	0& 0&\tr  -(R_F + R_G)/R*R_G& 0& 0& 0& 0& 0& 0&\tr 1&  0&  0& 0& 0\\
	0& 0& \tr -1/R_G& 0& 0& 0& 0& 0& 0& 0&  \tr 1&\tr -1& \tr 1& 0\\
	0& 0&\tr (R + R_F + R_G)/(R*R_G)& 0& 0& 0& 0& 0& 0& 0&\tr -1&  \tr 1& 0& \tr 1 \end{smallmatrix}}\right],
$$
\begin{equation*}
F_{MT}= \left[ {\transparent{0.4} \begin{smallmatrix}            0& 0& 0& 0& \tr 1& 0& 0& 0& 0& 0&0&0& 0& 0\\
	\tr (R_F + R_G)/R_G& 0& 0& 0& 0& 0& 0& 0& 0& 0&0&0& 0& 0\\
	0& 0& 0& 0& 0& 0& 0& 0& 0& 0&0&0& 0& 0\\
	0& 0& 0& 0& 0& 0& 0& 0& 0& 0&0&0& 0& 0\\
	0& 0& 0& 0& 0& 0& 0& 0& 0& 0&0&0& 0& 0\\
	0& 0& 0& 0& 0& 0& 0& 0& 0& 0& 0&0& 0& 0\\
	0& 0& 0& 0& 0& 0& 0& 0& 0& 0&0&0& 0& 0\\
	0& 0& 0& 0& 0& 0& 0& 0& 0& 0&0& 0& 0& 0\\
	0& 0& 0& 0& 0& 0& 0& 0& 0& 0& 0&0& 0& 0\\
	\tr 1/(Ca*R*R_G)& 0& 0& 0& 0& 0& 0& 0& 0& 0& \tr -1/(Ca*R)& \tr 1/(Ca*R)& 0& 0\\
	0& 0& 0& 0& 0& 0& 0& 0& 0& 0&0&0& 0& 0\\
	0& 0& 0& 0& 0& 0& 0& 0& 0& 0&0&0& 0& 0\\
	0& 0& 0& 0& 0& 0& 0& 0& 0& 0& 0& 0& 0& 0\\
	\tr -1/(Ca*R*R_G)& 0& 0& 0& 0& 0& 0& 0& 0& 0& \tr 1/(Ca*R)&\tr -1/(Ca*R)& 0& 0
	\end{smallmatrix}}\right].
\end{equation*}
Then the Morse transformation $M_{trans}(T_s,T_w,T_y,F_{MT},K_{MT})$ {brings} $\Lambda^w$ {into}  $\tilde \Lambda^{\tilde w}=(\tilde A,\tilde B^{\tilde w},\tilde C,\tilde D^{\tilde w})$, which is in the \textbf{EMTF}, where  
\begin{footnotesize}
	\begin{align*}
&\left[ \begin{smallmatrix}
	\tilde A&\vline&\tilde B^{\tilde w}\\
	\hline
	\tilde C&\vline&\tilde D^{\tilde w}
\end{smallmatrix}\right]=\left[ \begin{smallmatrix}
	\tilde A&\vline&\tilde B^{\tilde u}&\vline&\tilde B^{\tilde v}\\
	\hline
	\tilde C&\vline&\tilde D^{\tilde u}&\vline&0
\end{smallmatrix}\right]=\left[ \begin{smallmatrix}
	\tilde A_{1}&\tilde A_{13}&\vline&0&\vline&\tilde B_{1}^{\tilde v}&\tilde B_{12}^{\tilde v}\\
	0&\tilde A_{3}&\vline&0&\vline&0&\tilde B_{3}^{\tilde v}\\
	\hline
	0&\tilde C_{3}&\vline&\tilde D_{3}^{\tilde u}&\vline&0&0
\end{smallmatrix}\right]=\\&	\left[ {\transparent{0.4} \begin{smallmatrix} 
		\tr 0& \tr 0&\tr  0&\tr  0&\tr 0& 0& 0& 0& 0& 0& 0& 0&0& 0	&\vline& 0& 0&\vline&\tr 1& \tr 0&\tr 0& 0&0& 0& 0& 0& 0& 0& 0& 0\\
		\tr 1/(Ca*R_G)& \tr 0& \tr 0& \tr -1/Ca&\tr 1/Ca& 0& 0& 0& 0& 0& 0& 0&  \tr 1/Ca& 0&\vline& 0& 0&\vline&\tr (R_F + R_G)/R_G& \tr 0&\tr 0& 0& \tr1& 0& 0& 0& 0& 0& 0& 0\\
		\tr 1/L&\tr  0& \tr 0&\tr 0& \tr 0& 0& 0& 0& 0& 0& 0& 0&0& 0&\vline& 0& 0&\vline& \tr 0&\tr 0&\tr 0& 0&0& 0& 0& 0& 0& 0& 0& 0\\
		\tr 0& \tr 0& \tr 0&\tr 0&\tr 0& 0& 0& 0& 0& 0& 0& 0&0& 0&\vline&0& 0&\vline& \tr 0& \tr 1&\tr 0& 0&0& 0& 0& 0& 0& 0& 0& 0\\
		\tr 0& \tr 0& \tr 0&\tr 0&\tr 0& 0& 0& 0& 0& 0& 0& 0&0& 0&\vline&0& 0&\vline&\tr 0& \tr 0& \tr 1& 0& 0& 0& 0& 0& 0& 0& 0& 0\\
		0& 0& 0&0&0& \tr 0& \tr 0&\tr  0& \tr 0& \tr 0& \tr 0& \tr 0&\tr 0& \tr 0&\vline&0& 0&\vline&0& 0& 0& \tr 1&\tr 0&\tr  0&\tr  0& \tr 0& \tr  0& \tr 0& \tr  0&\tr  0\\
		0& 0& 0&0&0& \tr 0& \tr 0& \tr 0& \tr 0& \tr 0& \tr 0& \tr 0&\tr 0& \tr 0&\vline&0& 0&\vline&0& 0& 0& \tr  0&\tr 1& \tr  0& \tr  0& \tr 0& \tr 0& \tr 0& \tr 0&\tr  0\\
		0& 0& 0&0&0& \tr 0& \tr 0& \tr 0& \tr 0& \tr 0& \tr 0& \tr 0&\tr 0& \tr 0&\vline&0& 0&\vline&0& 0& 0& \tr 0&\tr 0& \tr 1& \tr 0& \tr 0& \tr 0& \tr 0& \tr 0& \tr 0\\
		0& 0& 0&0&0&\tr 0& \tr 0& \tr 0& \tr 0&\tr 0&\tr 0& \tr 0&\tr 0& \tr 0& \vline&0& 0&\vline&0& 0& 0& \tr 0&\tr 0& \tr 0& \tr 1& \tr 0& \tr 0&\tr  0& \tr 0&\tr  0\\
		0& 0& 0&0&0& \tr 0& \tr 0& \tr 0& \tr 0&\tr  0& \tr 0&\tr 0&\tr 0& \tr 0&\vline&0& 0&\vline&0& 0& 0& \tr  0& \tr  0&\tr  0&\tr  0& \tr  1& \tr  0& \tr  0& \tr  0&\tr  0\\
		0& 0& 0& 0& 0& \tr 0& \tr 0& \tr 0&\tr 0& \tr 0&\tr 0& \tr 0&\tr 0& \tr 0&\vline&0& 0&\vline&0& 0& 0& \tr  0&\tr 0&\tr  0& \tr 0&\tr  0& \tr 1& \tr 0& \tr 0&\tr  0\\
		0& 0& 0& 0& 0& \tr 0& \tr 0& \tr 0& \tr 0& \tr 0& \tr 0&\tr  0& \tr -1/(Ca*R)& \tr 0&\vline&0& 0&\vline&0& 0& 0&\tr  0& \tr -1/R& \tr 0& \tr 0& \tr 0& \tr 0& \tr 1& \tr 0& \tr 0\\
		0& 0& 0&0&0&\tr 0& \tr 0& \tr 0& \tr 0& \tr 0& \tr 0& \tr 0&\tr 0& \tr 0&\vline&0& 0&\vline&0& 0& 0& \tr 0&  \tr 0& \tr 0& \tr 0& \tr 0& \tr 0& \tr 0& \tr 1& \tr 0\\
		0& 0& 0& 0&  0& \tr 0& \tr 0& \tr 0& \tr 0& \tr 0& \tr 0& \tr 0&  \tr 1/(Ca*R)& \tr 0&\vline& 0& 0&\vline& 0& 0& 0& \tr 0&  \tr 1/R& \tr 0& \tr 0& \tr 0& \tr 0& \tr 0& \tr 0& \tr 1\\
		\hline\\
		0& 0& 0& 0& 0& \tr -1& \tr  0& \tr  0& \tr 0& \tr R_G& \tr  0&\tr  0&\tr  0& \tr 0&\vline&0& 0&\vline& 0& 0& 0& 0& 0& 0& 0& 0& 0& 0& 0& 0\\
		0& 0& 0& 0& 0& \tr  1& \tr -1&  \tr 0& \tr 0&  \tr 0& \tr R_F& \tr 0& \tr 0& \tr 0&\vline&0& 0&\vline& 0& 0& 0& 0& 0& 0& 0& 0& 0& 0& 0& 0\\
		0& 0& 0& 0& 0&  \tr 0&  \tr 0&  \tr 0& \tr 0&  \tr 0&  \tr 0& \tr R& \tr 0& \tr 0&\vline&0& 0&\vline& 0& 0& 0& 0& 0& 0& 0& 0& 0& 0& 0& 0\\
		0& 0& 0& 0& 0&  \tr 0&  \tr 0&  \tr 1& \tr 0&  \tr 0&  \tr 0& \tr 0& \tr 0& \tr 0&\vline&0& 0&\vline& 0& 0& 0& 0& 0& 0& 0& 0& 0& 0& 0& 0\\
		0& 0& 0& 0& 0&  \tr 0&  \tr 0&  \tr 0& \tr 1&  \tr 0&  \tr 0& \tr 0& \tr 0& \tr 0&\vline& 0& 0&\vline& 0& 0& 0& 0& 0& 0& 0& 0& 0& 0& 0& 0\\
		0& 0& 0& 0& 0& \tr -1&  \tr 0&  \tr 0& \tr 0&  \tr 0&  \tr 0& \tr 0& \tr 0& \tr 0&\vline&0& 0&\vline& 0& 0& 0& 0& 0& 0& 0& 0& 0& 0& 0& 0\\
		0& 0& 0& 0& 0&  \tr 0&  \tr 0& \tr -1& \tr 0&  \tr 0&  \tr 0& \tr 0& \tr 0& \tr 0&\vline& \tr 1& 0&\vline& 0& 0& 0& 0& 0& 0& 0& 0& 0& 0& 0& 0\\
		0& 0& 0& 0& 0&  \tr 0&  \tr 0&  \tr 0& \tr 1&  \tr 1& \tr -1& \tr 0& \tr 0& \tr 0&\vline&0& 0&\vline& 0& 0& 0& 0& 0& 0& 0& 0& 0& 0& 0& 0\\
		0& 0& 0& 0& 0&  \tr 0&  \tr 0&  \tr 0& \tr 0&  \tr 0& \tr -1& \tr 0& \tr 1& \tr 0&\vline&0& 0&\vline& 0& 0& 0& 0& 0& 0& 0& 0& 0& 0& 0& 0\\
		0& 0& 0& 0& 0&  \tr 0&  \tr 0&  \tr 0& \tr 0& \tr  0&  \tr 0& \tr 1& \tr 1& \tr 1&\vline&0& 0&\vline& 0& 0& 0& 0& 0& 0& 0& 0& 0& 0& 0& 0\\
		0& 0& 0& 0& 0&  \tr 0& \tr -1&  \tr 0& \tr 0&  \tr 0&  \tr 0& \tr 0& \tr 0& \tr 0&\vline&0& \tr 1&\vline& 0& 0& 0& 0& 0& 0& 0& 0& 0& 0& 0& 0
		\end{smallmatrix}}\right].
	\end{align*}
\end{footnotesize}
Step 3: By \textbf{MNF} Algorithm \ref{Alg:1}, {set}
\begin{align*}
K_{MN}=\left[ {\transparent{0.4} \begin{smallmatrix}
	0& 0& 0& 0&     0&         0& 0&    0&     0& 0& 0\\
	\tr -1/(Ca*R_G)& 0& 0& 0& \tr -1/Ca&\tr 1/(Ca*R_G)& 0&\tr 1/Ca&\tr -1/Ca& 0& 0\\
	0& 0& 0& 0&     0&         0& 0&    0&     0& 0& 0\\
	0& 0& 0& 0&     0&         0& 0&    0&     0& 0& 0\\
	0& 0& 0& 0&     0&         0& 0&    0&     0& 0& 0\\
	0& 0& 0& 0&     0&         0& 0&    0&     0& 0& 0\\
	0& 0& 0& 0&     0&         0& 0&    0&     0& 0& 0\\
	0& 0& 0& 0&     0&         0& 0&    0&     0& 0& 0\\
	0& 0& 0& 0&     0&         0& 0&    0&     0& 0& 0\\
	0& 0& 0& 0&     0&         0& 0&    0&     0& 0& 0\\
	0& 0& 0& 0&     0&         0& 0&    0&     0& 0& 0\\
	0& 0& 0& 0&     0&         0& 0&    0&     0& 0& 0\\
	0& 0& 0& 0&     0&         0& 0&    0&     0& 0& 0\\
	0& 0& 0& 0&     0&         0& 0&    0&     0& 0& 0
	\end{smallmatrix}}\right],\ \ F_{MN}=0_{14\times 14}.
\end{align*}
Then \blue{find} $T^2_{MN}$ via the following constrained Sylvester equation,  
\begin{align*}
\bar A_{1}{T^2_{MN}} - {T^2_{MN}}\bar A_{3} =  - \bar A_{1},\ \ \ T^2_{MN}\bar B^{\bar w}_{3}=-\bar B^{\bar w}_{12},
\end{align*}
where $\bar A=\tilde A+K_{MN}\tilde C$, $\bar B^{\bar w}=\tilde B^{\tilde w}+K_{MN}\tilde D^{\tilde w}$. The above {equation} is solvable and {the} solution is
\begin{align*}
T^2_{MN} =\left[ {\transparent{0.4} \begin{smallmatrix}
	0& 0& 0& 0& 0& 0& 0& 0& 0\\
	0& \tr -1& 0& 0& 0& 0& 0& 0& 0\\
	0& 0& 0& 0& 0& 0& 0& 0& 0\\
	0& 0& 0& 0& 0& 0& 0& 0& 0\\
	0& 0& 0& 0& 0& 0& 0& 0& 0
	\end{smallmatrix}}\right].
\end{align*}
Thus the Morse transformation $M_{tran}=(T_{MN},I_{14},I_{11},F_{MN},K_{MN})$, where $T_{MN}=\left[ \begin{smallmatrix}
I&T^2_{MN}\\0&I
\end{smallmatrix}\right] $, {brings} $\tilde \Lambda^{\tilde w}$ {into} $\bar \Lambda^{\bar w}=(\bar A,\bar B^{\bar w},\bar C,\bar D^{\bar w})$, which is in the \textbf{EMNF}, where 	
\begin{footnotesize}
	\begin{align*}
&\left[ \begin{smallmatrix}
	\bar A&\vline&\bar B^{\bar w}\\
	\hline
	\bar C&\vline&\bar D^{\bar w}
\end{smallmatrix}\right]=\left[ \begin{smallmatrix}
	\bar A&\vline&\bar B^{\bar u}&\vline&\bar B^{\bar v}\\
	\hline
	\bar C&\vline&\bar D^{\bar u}&\vline&0
\end{smallmatrix}\right]=\left[ \begin{smallmatrix}
	\bar A_{11}&0&\vline&0&\vline&\bar B_{11}^{\bar v}&0\\
	0&\bar A_{33}&\vline&0&\vline&0&\bar B_{32}^{\bar v}\\
	\hline
	0&\bar C_{13}&\vline&\bar D_{12}^{\bar u}&\vline&0&0
\end{smallmatrix}\right]=\\
&\left[ {\transparent{0.4} \begin{smallmatrix} 
		\tr 0& \tr 0&\tr  0&\tr  0&\tr 0& 0& 0& 0& 0& 0& 0& 0&0& 0	&\vline& 0& 0&\vline&\tr 1& \tr 0&\tr 0& 0&0& 0& 0& 0& 0& 0& 0& 0\\
		\tr 1/(Ca*R_G)& \tr 0& \tr 0& \tr -1/Ca&\tr 1/Ca& 0& 0& 0& 0& 0& 0& 0&  0& 0&\vline& 0& 0&\vline&\tr (R_F + R_G)/R_G& \tr 0&\tr 0& 0& 0& 0& 0& 0& 0& 0& 0& 0\\
		\tr 1/L&\tr  0& \tr 0&\tr 0& \tr 0& 0& 0& 0& 0& 0& 0& 0&0& 0&\vline& 0& 0&\vline& \tr 0&\tr 0&\tr 0& 0&0& 0& 0& 0& 0& 0& 0& 0\\
		\tr 0& \tr 0& \tr 0&\tr 0&\tr 0& 0& 0& 0& 0& 0& 0& 0&0& 0&\vline&0& 0&\vline& \tr 0& \tr 1&\tr 0& 0&0& 0& 0& 0& 0& 0& 0& 0\\
		\tr 0& \tr 0& \tr 0&\tr 0&\tr 0& 0& 0& 0& 0& 0& 0& 0&0& 0&\vline&0& 0&\vline&\tr 0& \tr 0& \tr 1& 0& 0& 0& 0& 0& 0& 0& 0& 0\\
		0& 0& 0&0&0& \tr 0& \tr 0&\tr  0& \tr 0& \tr 0& \tr 0& \tr 0&\tr 0& \tr 0&\vline&0& 0&\vline&0& 0& 0& \tr 1&\tr 0&\tr  0&\tr  0& \tr 0& \tr  0& \tr 0& \tr  0&\tr  0\\
		0& 0& 0&0&0& \tr 0& \tr 0& \tr 0& \tr 0& \tr 0& \tr 0& \tr 0&\tr 0& \tr 0&\vline&0& 0&\vline&0& 0& 0& \tr  0&\tr 1& \tr  0& \tr  0& \tr 0& \tr 0& \tr 0& \tr 0&\tr  0\\
		0& 0& 0&0&0& \tr 0& \tr 0& \tr 0& \tr 0& \tr 0& \tr 0& \tr 0&\tr 0& \tr 0&\vline&0& 0&\vline&0& 0& 0& \tr 0&\tr 0& \tr 1& \tr 0& \tr 0& \tr 0& \tr 0& \tr 0& \tr 0\\
		0& 0& 0&0&0&\tr 0& \tr 0& \tr 0& \tr 0&\tr 0&\tr 0& \tr 0&\tr 0& \tr 0& \vline&0& 0&\vline&0& 0& 0& \tr 0&\tr 0& \tr 0& \tr 1& \tr 0& \tr 0&\tr  0& \tr 0&\tr  0\\
		0& 0& 0&0&0& \tr 0& \tr 0& \tr 0& \tr 0&\tr  0& \tr 0&\tr 0&\tr 0& \tr 0&\vline&0& 0&\vline&0& 0& 0& \tr  0& \tr  0&\tr  0&\tr  0& \tr  1& \tr  0& \tr  0& \tr  0&\tr  0\\
		0& 0& 0& 0& 0& \tr 0& \tr 0& \tr 0&\tr 0& \tr 0&\tr 0& \tr 0&\tr 0& \tr 0&\vline&0& 0&\vline&0& 0& 0& \tr  0&\tr 0&\tr  0& \tr 0&\tr  0& \tr 1& \tr 0& \tr 0&\tr  0\\
		0& 0& 0& 0& 0& \tr 0& \tr 0& \tr 0& \tr 0& \tr 0& \tr 0&\tr  0& \tr -1/(Ca*R)& \tr 0&\vline&0& 0&\vline&0& 0& 0&\tr  0& \tr -1/R& \tr 0& \tr 0& \tr 0& \tr 0& \tr 1& \tr 0& \tr 0\\
		0& 0& 0&0&0&\tr 0& \tr 0& \tr 0& \tr 0& \tr 0& \tr 0& \tr 0&\tr 0& \tr 0&\vline&0& 0&\vline&0& 0& 0& \tr 0&  \tr 0& \tr 0& \tr 0& \tr 0& \tr 0& \tr 0& \tr 1& \tr 0\\
		0& 0& 0& 0&  0& \tr 0& \tr 0& \tr 0& \tr 0& \tr 0& \tr 0& \tr 0&  \tr 1/(Ca*R)& \tr 0&\vline& 0& 0&\vline& 0& 0& 0& \tr 0&  \tr 1/R& \tr 0& \tr 0& \tr 0& \tr 0& \tr 0& \tr 0& \tr 1\\
		\hline\\
		0& 0& 0& 0& 0& \tr -1& \tr  0& \tr  0& \tr 0& \tr R_G& \tr  0&\tr  0&\tr  0& \tr 0&\vline&0& 0&\vline& 0& 0& 0& 0& 0& 0& 0& 0& 0& 0& 0& 0\\
		0& 0& 0& 0& 0& \tr  1& \tr -1&  \tr 0& \tr 0&  \tr 0& \tr R_F& \tr 0& \tr 0& \tr 0&\vline&0& 0&\vline& 0& 0& 0& 0& 0& 0& 0& 0& 0& 0& 0& 0\\
		0& 0& 0& 0& 0&  \tr 0&  \tr 0&  \tr 0& \tr 0&  \tr 0&  \tr 0& \tr R& \tr 0& \tr 0&\vline&0& 0&\vline& 0& 0& 0& 0& 0& 0& 0& 0& 0& 0& 0& 0\\
		0& 0& 0& 0& 0&  \tr 0&  \tr 0&  \tr 1& \tr 0&  \tr 0&  \tr 0& \tr 0& \tr 0& \tr 0&\vline&0& 0&\vline& 0& 0& 0& 0& 0& 0& 0& 0& 0& 0& 0& 0\\
		0& 0& 0& 0& 0&  \tr 0&  \tr 0&  \tr 0& \tr 1&  \tr 0&  \tr 0& \tr 0& \tr 0& \tr 0&\vline& 0& 0&\vline& 0& 0& 0& 0& 0& 0& 0& 0& 0& 0& 0& 0\\
		0& 0& 0& 0& 0& \tr -1&  \tr 0&  \tr 0& \tr 0&  \tr 0&  \tr 0& \tr 0& \tr 0& \tr 0&\vline&0& 0&\vline& 0& 0& 0& 0& 0& 0& 0& 0& 0& 0& 0& 0\\
		0& 0& 0& 0& 0&  \tr 0&  \tr 0& \tr -1& \tr 0&  \tr 0&  \tr 0& \tr 0& \tr 0& \tr 0&\vline& \tr 1& 0&\vline& 0& 0& 0& 0& 0& 0& 0& 0& 0& 0& 0& 0\\
		0& 0& 0& 0& 0&  \tr 0&  \tr 0&  \tr 0& \tr 1&  \tr 1& \tr -1& \tr 0& \tr 0& \tr 0&\vline&0& 0&\vline& 0& 0& 0& 0& 0& 0& 0& 0& 0& 0& 0& 0\\
		0& 0& 0& 0& 0&  \tr 0&  \tr 0&  \tr 0& \tr 0&  \tr 0& \tr -1& \tr 0& \tr 1& \tr 0&\vline&0& 0&\vline& 0& 0& 0& 0& 0& 0& 0& 0& 0& 0& 0& 0\\
		0& 0& 0& 0& 0&  \tr 0&  \tr 0&  \tr 0& \tr 0& \tr  0&  \tr 0& \tr 1& \tr 1& \tr 1&\vline&0& 0&\vline& 0& 0& 0& 0& 0& 0& 0& 0& 0& 0& 0& 0\\
		0& 0& 0& 0& 0&  \tr 0& \tr -1&  \tr 0& \tr 0&  \tr 0&  \tr 0& \tr 0& \tr 0& \tr 0&\vline&0& \tr 1&\vline& 0& 0& 0& 0& 0& 0& 0& 0& 0& 0& 0& 0
		\end{smallmatrix}}\right].
	\end{align*}
\end{footnotesize}

Step 4: Transform each subsystem of $\bar \Lambda^{\bar w}$ into its canonical form as in Theorem \ref{Thm:EMCF} {to obtain} 
\begin{align*}
\textbf{EMCF}: \left[ \begin{smallmatrix}
A^{cv}&0&\vline&0&\vline&B^{cv}&0\\
0&A^{pv}&\vline&0&\vline&0& B^{pv}\\
\hline
0&0&\vline&D^{pu}&\vline&0&0\\
0&\bar C^{pv}&\vline&0&\vline&0&0\\
\end{smallmatrix}\right]=
\left[ {\transparent{0.4} \begin{smallmatrix} 
	\tr 0& \tr 1&\tr  0&\tr  0&\tr 0& 0& 0& 0& 0& 0& 0& 0&0& 0	&\vline& 0& 0&\vline&\tr 0& \tr 0&\tr 0& 0&0& 0& 0& 0& 0& 0& 0& 0\\
	\tr 0& \tr 0& \tr 0& \tr 0&\tr 0& 0& 0& 0& 0& 0& 0& 0&  0& 0&\vline& 0& 0&\vline&\tr 1& \tr 0&\tr 0& 0& 0& 0& 0& 0& 0& 0& 0& 0\\
	\tr 0&\tr  0& \tr 0&\tr 1& \tr 0& 0& 0& 0& 0& 0& 0& 0&0& 0&\vline& 0& 0&\vline& \tr 0&\tr 0&\tr 0& 0&0& 0& 0& 0& 0& 0& 0& 0\\
	\tr 0& \tr 0& \tr 0&\tr 0&\tr 0& 0& 0& 0& 0& 0& 0& 0&0& 0&\vline&0& 0&\vline& \tr 0& \tr 1&\tr 0& 0&0& 0& 0& 0& 0& 0& 0& 0\\
	\tr 0& \tr 0& \tr 0&\tr 0&\tr 0& 0& 0& 0& 0& 0& 0& 0&0& 0&\vline&0& 0&\vline&\tr 0& \tr 0& \tr 1& 0& 0& 0& 0& 0& 0& 0& 0& 0\\
	0& 0& 0&0&0& \tr 0& \tr 0&\tr  0& \tr 0& \tr 0& \tr 0& \tr 0&\tr 0& \tr 0&\vline&0& 0&\vline&0& 0& 0& \tr 1&\tr 0&\tr  0&\tr  0& \tr 0& \tr  0& \tr 0& \tr  0&\tr  0\\
	0& 0& 0&0&0& \tr 0& \tr 0& \tr 0& \tr 0& \tr 0& \tr 0& \tr 0&\tr 0& \tr 0&\vline&0& 0&\vline&0& 0& 0& \tr  0&\tr 1& \tr  0& \tr  0& \tr 0& \tr 0& \tr 0& \tr 0&\tr  0\\
	0& 0& 0&0&0& \tr 0& \tr 0& \tr 0& \tr 0& \tr 0& \tr 0& \tr 0&\tr 0& \tr 0&\vline&0& 0&\vline&0& 0& 0& \tr 0&\tr 0& \tr 1& \tr 0& \tr 0& \tr 0& \tr 0& \tr 0& \tr 0\\
	0& 0& 0&0&0&\tr 0& \tr 0& \tr 0& \tr 0&\tr 0&\tr 0& \tr 0&\tr 0& \tr 0& \vline&0& 0&\vline&0& 0& 0& \tr 0&\tr 0& \tr 0& \tr 1& \tr 0& \tr 0&\tr  0& \tr 0&\tr  0\\
	0& 0& 0&0&0& \tr 0& \tr 0& \tr 0& \tr 0&\tr  0& \tr 0&\tr 0&\tr 0& \tr 0&\vline&0& 0&\vline&0& 0& 0& \tr  0& \tr  0&\tr  0&\tr  0& \tr  1& \tr  0& \tr  0& \tr  0&\tr  0\\
	0& 0& 0& 0& 0& \tr 0& \tr 0& \tr 0&\tr 0& \tr 0&\tr 0& \tr 0&\tr 0& \tr 0&\vline&0& 0&\vline&0& 0& 0& \tr  0&\tr 0&\tr  0& \tr 0&\tr  0& \tr 1& \tr 0& \tr 0&\tr  0\\
	0& 0& 0& 0& 0& \tr 0& \tr 0& \tr 0& \tr 0& \tr 0& \tr 0&\tr  0& \tr 0& \tr 0&\vline&0& 0&\vline&0& 0& 0&\tr  0& \tr 0& \tr 0& \tr 0& \tr 0& \tr 0& \tr 1& \tr 0& \tr 0\\
	0& 0& 0&0&0&\tr 0& \tr 0& \tr 0& \tr 0& \tr 0& \tr 0& \tr 0&\tr 0& \tr 0&\vline&0& 0&\vline&0& 0& 0& \tr 0&  \tr 0& \tr 0& \tr 0& \tr 0& \tr 0& \tr 0& \tr 1& \tr 0\\
	0& 0& 0& 0&  0& \tr 0& \tr 0& \tr 0& \tr 0& \tr 0& \tr 0& \tr 0&  \tr 0& \tr 0&\vline& 0& 0&\vline& 0& 0& 0& \tr 0&  \tr 0& \tr 0& \tr 0& \tr 0& \tr 0& \tr 0& \tr 0& \tr 1\\
	\hline\\
	0& 0& 0& 0& 0&  0&  0&   0&  0&  0&   0&  0& 0&  0&\vline&\tr 1&\tr 0&\vline& 0& 0& 0& 0& 0& 0& 0& 0& 0& 0& 0& 0\\
	0& 0& 0& 0& 0& 0& 0&  0& 0&  0& 0& 0&0& 0&\vline&\tr 0&\tr 1&\vline& 0& 0& 0& 0& 0& 0& 0& 0& 0& 0& 0& 0\\
	0& 0& 0& 0& 0&  \tr 1&  \tr 0&  \tr 0& \tr 0&  \tr 0&  \tr 0& \tr 0& \tr 0& \tr 0&\vline&0& 0&\vline& 0& 0& 0& 0& 0& 0& 0& 0& 0& 0& 0& 0\\
	0& 0& 0& 0& 0&  \tr 0&  \tr 1&  \tr 0& \tr 0&  \tr 0&  \tr 0& \tr 0& \tr 0& \tr 0&\vline&0& 0&\vline& 0& 0& 0& 0& 0& 0& 0& 0& 0& 0& 0& 0\\
	0& 0& 0& 0& 0&  \tr 0&  \tr 0&  \tr 1& \tr 0&  \tr 0&  \tr 0& \tr 0& \tr 0& \tr 0&\vline& 0& 0&\vline& 0& 0& 0& 0& 0& 0& 0& 0& 0& 0& 0& 0\\
	0& 0& 0& 0& 0& \tr 0&  \tr 0&  \tr 0& \tr 1&  \tr 0&  \tr 0& \tr 0& \tr 0& \tr 0&\vline&0& 0&\vline& 0& 0& 0& 0& 0& 0& 0& 0& 0& 0& 0& 0\\
	0& 0& 0& 0& 0&  \tr 0&  \tr 0& \tr 0& \tr 0&  \tr 1&  \tr 0& \tr 0& \tr 0& \tr 0&\vline&  0& 0&\vline& 0& 0& 0& 0& 0& 0& 0& 0& 0& 0& 0& 0\\
	0& 0& 0& 0& 0&  \tr 0&  \tr 0&  \tr 0& \tr 0&  \tr 0& \tr 1& \tr 0& \tr 0& \tr 0&\vline&0& 0&\vline& 0& 0& 0& 0& 0& 0& 0& 0& 0& 0& 0& 0\\
	0& 0& 0& 0& 0&  \tr 0&  \tr 0&  \tr 0& \tr 0&  \tr 0& \tr 0& \tr 1& \tr 0& \tr 0&\vline&0& 0&\vline& 0& 0& 0& 0& 0& 0& 0& 0& 0& 0& 0& 0\\
	0& 0& 0& 0& 0&  \tr 0&  \tr 0&  \tr 0& \tr 0& \tr  0&  \tr 0& \tr 0& \tr 1& \tr 0&\vline&0& 0&\vline& 0& 0& 0& 0& 0& 0& 0& 0& 0& 0& 0& 0\\
	0& 0& 0& 0& 0&  \tr 0& \tr 0&  \tr 0& \tr 0&  \tr 0&  \tr 0& \tr 0& \tr 0& \tr 1&\vline&0&  0&\vline& 0& 0& 0& 0& 0& 0& 0& 0& 0& 0& 0& 0
	\end{smallmatrix}}\right].
\end{align*}
The \textbf{EMCF} indices are $\bar \epsilon_1=2, \bar \epsilon_2=2,\bar \epsilon_3=1$, $\delta=2$, $\bar \sigma_1=\bar \sigma_2=,\dots,=\bar \sigma_9=1$. Note that $n_2, a, c, e$ are all zeros {and} we have 3 subsystems {only}.

Step 5: {Using} the {$v$-reduction and implicitation of Definition \ref{Def:red_impl}}, we get  the following DACS from the above \textbf{EMCF}:
\begin{align*}
\left[  {\transparent{0.4}{
		\begin{smallmatrix}
		\tr 1&\tr 0&\tr 0&\tr 0&{\tr 0}&0&0&0&0&0&0&0&0&0\\
		\tr 0&\tr 0&\tr 1&\tr 0&\tr 0&0&0&0&0&0&0&0&0&0\\
		0&0&0&0&0& 0&0&0&0&0&0&0&0&0\\
		0&0&0&0&0& 0& 0&0&0&0&0&0&0&0\\
		0&0&0&0&0&\tr 0&\tr 0&\tr 0&\tr 0&\tr 0&\tr 0&\tr 0&\tr 0&\tr 0\\
		0&0&0&0&0&\tr 0&\tr 0&\tr 0&\tr 0&\tr 0&\tr 0&\tr 0&\tr 0&\tr 0\\
		0&0&0&0&0&\tr 0&\tr 0&\tr 0&\tr 0&\tr 0&\tr 0&\tr 0&\tr 0&\tr 0\\
		0&0&0&0&0&\tr 0&\tr 0&\tr 0&\tr 0&\tr 0&\tr 0&\tr 0&\tr 0&\tr 0\\
		0&0&0&0&0&\tr 0&\tr 0&\tr 0&\tr 0&\tr 0&\tr 0&\tr 0&\tr 0&\tr 0\\
		0&0&0&0&0&\tr 0&\tr 0&\tr 0&\tr 0&\tr 0&\tr 0&\tr 0&\tr 0&\tr 0\\
		0&0&0&0&0&\tr 0&\tr 0&\tr 0&\tr 0&\tr 0&\tr 0&\tr 0&\tr 0&\tr 0\\
		0&0&0&0&0&\tr 0&\tr 0&\tr 0&\tr 0&\tr 0&\tr 0&\tr 0&\tr 0&\tr 0\\
		0&0&0&0&0&\tr 0&\tr 0&\tr 0&\tr 0&\tr 0&\tr 0&\tr 0&\tr 0&\tr 0\\
		\end{smallmatrix}}}\right]\dot z= \left[ {\transparent{0.4}\begin{smallmatrix}
	\tr 0&\tr 1&\tr 0&\tr 0&\tr 0&0&0&0&0&0&0&0&0&0\\
	\tr 0&\tr 0&\tr 0&\tr 1&\tr 0&0&0&0&0&0&0&0&0&0\\
	0&0&0&0&0& 0&0&0&0&0&0&0&0&0\\
	0&0&0&0&0& 0& 0&0&0&0&0&0&0&0\\
	0&0&0&0&0&\tr 1&\tr 0&\tr 0&\tr 0&\tr 0&\tr 0&\tr 0&\tr 0&\tr 0\\
	0&0&0&0&0&\tr 0&\tr 1&\tr 0&\tr 0&\tr 0&\tr 0&\tr 0&\tr 0&\tr 0\\
	0&0&0&0&0&\tr 0&\tr 0&\tr 1&\tr 0&\tr 0&\tr 0&\tr 0&\tr 0&\tr 0\\
	0&0&0&0&0&\tr 0&\tr 0&\tr 0&\tr 1&\tr 0&\tr 0&\tr 0&\tr 0&\tr 0\\
	0&0&0&0&0&\tr 0&\tr 0&\tr 0&\tr 0&\tr 1&\tr 0&\tr 0&\tr 0&\tr 0\\
	0&0&0&0&0&\tr 0&\tr 0&\tr 0&\tr 0&\tr 0&\tr 1&\tr 0&\tr 0&\tr 0\\
	0&0&0&0&0&\tr 0&\tr 0&\tr 0&\tr 0&\tr 0&\tr 0&\tr 1&\tr 0&\tr 0\\
	0&0&0&0&0&\tr 0&\tr 0&\tr 0&\tr 0&\tr 0&\tr 0&\tr 0&\tr 1&\tr 0\\
	0&0&0&0&0&\tr 0&\tr 0&\tr 0&\tr 0&\tr 0&\tr 0&\tr 0&\tr 0&\tr 1\\
	\end{smallmatrix}}\right]z+\left[ {\transparent{0.4} \begin{smallmatrix}
	0&0\\
	0&0\\
	\tr 1&0\\
	0&\tr 1\\
	0&0\\
	0&0\\
	0&0\\
	0&0\\
	0&0\\
	0&0\\
	0&0\\
	0&0\\
	0&0\\
	\end{smallmatrix}}\right]\tilde u,
\end{align*}
where $z$ and $\tilde u$ is the new ``generalized'' state and the new input, respectively. Obviously, the above DACS {is in} the \textbf{FBCF} with indices $\bar \epsilon'_1=2, \bar \epsilon'_2=2,\bar \epsilon'_3=1$, $\sigma'_1=\sigma'_2=1$, $\bar \sigma'_1=\bar \sigma'_2=,\dots,=\bar \sigma'_9=1$. Moreover, $a'=n_{\rho}=e'=0$, $c'=\delta=2$.
\section{Proofs of the results}\label{sec:proofs}
\subsection{{Proofs of Proposition \ref{Pro:LinDAEexpl}, Proposition \ref{Pro:solution}   and Theorem \ref{Thm:equivalence}}}\label{ProofThm:equivalence}
\begin{proof}[Proof of Proposition \ref{Pro:LinDAEexpl}]
 	\emph{If.} Suppose that $\Lambda^{uv}$ and $\tilde \Lambda^{u\tilde v}$ are equivalent via \blue{a} transformation given \blue{by} (\ref{Eq:mapDACS}). First, ${\rm Im\,}\tilde B^{\tilde v}\mathop=\limits^{(\ref{Eq:mapDACS})}{\rm Im\,} B^{v}T^{-1}_v=\ker E_1=\ker E$ implies that $\tilde B^{\tilde v}$ is another choice  such that ${\rm Im\,} \tilde B^{\tilde v}=\ker E$. Observe that
	$$
	\tilde \Lambda^{u\tilde v}:\left\{ {\begin{array}{r@{\,}l}
		\dot x&=\tilde Ax+\tilde B^uu+\tilde B^{\tilde v}\tilde v\mathop=\limits^{(\ref{Eq:mapDACS})}(A+K C+B^vF_v)x+(B^u+ K D^u+B^vR)u+ B^vT^{-1}_v\tilde v  \\
		\tilde y&=\tilde Cx+\tilde D^uu\mathop=\limits^{(\ref{Eq:mapDACS})} T_y Cx+T_y D^uu.
		\end{array}}\right.	
	$$
Then pre-multiply the differential part of \blue{$\tilde \Lambda^{u\tilde v}$} by $E_1$, to get (notice that $A=E^{\dagger}_1H_1$, $B^u=E^{\dagger}_1L_1$, ${\rm Im\,}B^v=\ker E_1$ and $C=H_2$, $D^u=L_2$)
	$$
	\left\{ {\begin{array}{r@{\ }l}
		E_1\dot x&=(H_1+E_1K H_2)x+(L_1+E_1K L_2)u \\
		\tilde y&=T_y H_2x+T_y L_2u.
		\end{array}}\right.	
	$$
	Thus  $\tilde \Lambda^{u\tilde v}$ is {an} $(I_l,\tilde v)$-explicitation of the following DACS:
	$$
	\left[ {\begin{smallmatrix}
		E_1\\
		0
		\end{smallmatrix}} \right]\dot x=\left[ {\begin{smallmatrix}
		H_1+E_1K H_2\\
		T_y H_2
		\end{smallmatrix}} \right]x+\left[ {\begin{smallmatrix}
		L_1+E_1K L_2\\
		T_y L_2
		\end{smallmatrix}} \right]u.
	$$
	Since the above DACS can be transformed from $\Delta^u$ via $\tilde Q=Q'Q$, where $Q'=\left[ {\begin{smallmatrix}
		I_q&E_1K\\
		0&T_y
		\end{smallmatrix}}\right]$, it proves that $\tilde \Lambda^{u\tilde v}$ is {a} $(\tilde Q,\tilde v)$-explicitation of $\Delta^u$ corresponding {to} the choice of  invertible matrix $\tilde Q$. Finally, by $E_1\tilde A=H_1+E_1KH_2$, $E_1\tilde B^u=L_1+E_1KL_2$, we get $\tilde A=\tilde E^{\dagger}_1(H_1+KH_2)$ and $\tilde B^u=\tilde E^{\dagger}_1(L_1+KL_2)$ for another choice of right inverse $\tilde E^{\dagger}_1$ of $E_1$. 
	
 	\emph{Only if.} Suppose that $\tilde \Lambda^{u\tilde v}\in \mathbf{Expl}(\Delta^u)$ via $\tilde Q$, $\tilde E^{\dagger}_1$ and $\tilde B^{\tilde v}$. First, by ${\rm Im\,}\tilde B^{\tilde v}=\ker E={\rm Im\,}B^v$, there exists an invertible matrix $T^{-1}_v$ such that $\tilde B^{\tilde v}=B^vT^{-1}_v$. Moreover, since $E^{\dagger}_1$ is a right inverse of $E_1$ if and only if any solution $\dot x$ of $E_1\dot x=w$ is given by $E^{\dagger}_1w$, we have $E_1E_1^{\dagger}(H_1x+L_1u)=H_1x+L_1u$ and $E_1\tilde E_1^{\dagger}(H_1x+L_1u)=H_1x+L_1u$. It follows that $E_1(\tilde E^{\dagger}_1-E^{\dagger}_1)(H_1x+L_1u)=0$, so $(\tilde E^{\dagger}_1-E^{\dagger}_1)H_1\in \ker E_1$, $(\tilde E^{\dagger}_1-E^{\dagger}_1)L_1\in \ker E_1$. Since $\ker E_1={\rm Im\,}B^v$, it follows that $(\tilde E^{\dagger}_1-E^{\dagger}_1)H_1=B^v
	F_v$  and  $(\tilde E^{\dagger}_1-E^{\dagger}_1)L_1=B^v
	R$ for suitable $F_v$ and {$R$}. Furthermore, since $Q$ is such that $E_1$ of
	$
	QE=\left[ {\begin{smallmatrix}
		E_1\\
		0
		\end{smallmatrix}} \right]$ is of full row rank,  it follows that  any other $\tilde Q$, such that $\tilde E_1$ of $\tilde QE=\left[ {\begin{smallmatrix}
		\tilde E_1\\
		0
		\end{smallmatrix}} \right]$ is full row rank, must be of the form $\tilde Q=Q'Q$, where $Q'=\left[ {\begin{smallmatrix}
		Q_1&Q_2\\
		0&Q_4
		\end{smallmatrix}}\right]$.
	Thus via $\tilde Q$, $\Delta^u$ is ex-equivalent to
	\begin{align*}
 	Q'\left[ {\begin{smallmatrix}
		E_1\\
		0
		\end{smallmatrix}} \right]\dot x=Q'\left[ {\begin{smallmatrix}
		H_1\\
		H_2
		\end{smallmatrix}} \right]+Q'\left[ {\begin{smallmatrix}
		L_1\\
		L_2
		\end{smallmatrix}} \right]u \Rightarrow \left[ {\begin{smallmatrix}
		Q_1E_1\\
		0
		\end{smallmatrix}} \right]\dot x=\left[ {\begin{smallmatrix}
		Q_1H_1+Q_2H_2\\
		Q_4H_2
		\end{smallmatrix}} \right]+\left[ {\begin{smallmatrix}
		Q_1L_1+Q_2L_2\\
		Q_4L_2
		\end{smallmatrix}} \right]u.
	\end{align*} 
We {obtain} the following equations, using $\tilde E_1^{\dagger}$ and $\tilde B^{\tilde v}$, and based on the right-hand side of the above:
	\begin{align*}
	\left\{ {\begin{aligned} 
		\dot x&=(\tilde E_1^{\dagger}H_1+\tilde E_1^{\dagger}Q^{-1}_1Q_2H_2)x+(\tilde E_1^{\dagger}L_1+\tilde E_1^{\dagger}Q^{-1}_1Q_2L_2)u+\tilde B^{\tilde v}v\\&=(E^{\dagger}_1H_1+B^vF_v+E_1^{\dagger}Q^{-1}_1Q_2C)x+(E^{\dagger}_1H_1+B^vR+E_1^{\dagger}Q^{-1}_1Q_2D^u)u+B^vT^{-1}_v\tilde v\\
		0 &= Q_4H_2+Q_4L_2=Q_4Cx+Q_4D^u.
		\end{aligned}}\right.	
	\end{align*}
	Thus the explicitation of $\Delta^u$ via $\tilde Q$, $\tilde E^{\dagger}_1$ and $\tilde B^{\tilde v}$ is 
	$$
	\tilde \Lambda^{u\tilde v}:\left\{ {\begin{array}{*{20}{l}}
		\dot x=Ax+ K( Cx+D^uu) +B^v(F_vx +R u+T^{-1}_v\tilde v)=\tilde Ax+\tilde B^uu+\tilde B^{\tilde v}\tilde v\\
		\tilde y= T_y Cx+T_y D^uu=\tilde Cx+\tilde D^uu.
		\end{array}}\right.	
	$$
	where $K=E_1^{\dagger}Q^{-1}_1Q_2$, $T_y=Q_4$. Now we can see {that} $\Lambda^{uv}$ and $\tilde \Lambda^{u\tilde v}$ are equivalent via transformations {listed} in (\ref{Eq:mapDACS}).  
\end{proof}

\begin{proof}[Proof of Proposition \ref{Pro:solution}]
 Consider equation (\ref{Eq:expl1}) {of} the $(Q,v)$-explicitation procedure. Since $Q$-transformations preserve solutions of $\Delta^u$,  equation (\ref{Eq:expl1}) resulting from a  $Q$-transformation of $\Delta^u$ has the same solutions {as} $\Delta^u$. Thus we need to prove that equations  (\ref{Eq:expl1}) and  (\ref{Eq:expl3}) have corresponding solutions for any choices of $E^{\dagger}_1$ and $B^v$. Moreover, the second equation $0=H_2x+L_2u$ of (\ref{Eq:expl1}) coincides  with $0=Cx+D^uu$ of (\ref{Eq:expl3}) (since $C=H_2$ and $D^u=L_2$).  So we only need to prove that $(x(t),u(t))$ with $x(t)\in \mathcal C^1$ and $u(t)\in \mathcal C^0$  is a solution of (\ref{Eq:expl1a}) if and only if there exists $v(t)\in\mathcal C^0$ such that $(x(t),u(t),v(t))$ is a solution of (\ref{Eq:ODEcontrolnooutput1})   { independently of the choice of $E^{\dagger}_1$, defining  $A=E^{\dagger}_1H$ and $B^u=E^{\dagger}_1L_1$, and of the choice of $B^v$ satisfying ${\rm Im\,}B^v=\ker E_1$.} 
	
 \emph{If.} Suppose that  $(x(t),u(t),v(t))$ is a solution of (\ref{Eq:ODEcontrolnooutput1}). Then we have $
		\dot x(t)=Ax(t)+B^uu(t)+B^vv(t)$.
		Pre-multiplying the {last} equation by $E_1$, we conclude ({recall} that $A=E^{\dagger}_1H_1$ , $B^u=E^{\dagger}_1L_1$, $\ker E_1={\rm Im\,}B_v$) {that} $E_1\dot x(t)=H_1x(t)+L_1u(t)$, which proves that $(x(t),u(t))$ is a solution of (\ref{Eq:expl1a}). 
	
	\emph{Only if.} Suppose that $(x(t),u(t))$ is a solution of (\ref{Eq:expl1a}). Rewrite $E_1\dot x$ as 
		$
		\left[ \begin{smallmatrix}
		E^1_1&E^2_1
		\end{smallmatrix}\right] \left[ \begin{smallmatrix}
		\dot x_1\\\dot x_2
		\end{smallmatrix}\right] 
		$, where $E^1_1\in \mathbb R^{q\times q}$ and $x= \left[ \begin{smallmatrix}
		x_1\\  x_2
		\end{smallmatrix}\right] $. Then, without loss of generality, we assume that the matrix $E^1_1$ is invertible (if not, we permute the components of $x$ such that the first $q$ columns of $E_1$ are independent). Thus, a choice of the right inverse of $E_1$ is $E_1^{\dagger}= \left[ \begin{smallmatrix}
		(E^1_1)^{-1}\\0
		\end{smallmatrix}\right] $ (since $\left[ \begin{smallmatrix}
		E^1_1&E^2_1
		\end{smallmatrix}\right]\left[ \begin{smallmatrix}
		(E^1_1)^{-1}\\0
		\end{smallmatrix}\right] =I_q$), which gives the matrices $A$, $B^u$, $B^v$ of (\ref{Eq:ODEcontrolnooutput1}) \blue{to be, respectively,}
		$$
		A:=E^{\dagger}_1 H_1=\left[ \begin{smallmatrix}
		(E^1_1)^{-1}H_1\\0
		\end{smallmatrix}\right], \ \ \ B^u:=E^{\dagger}_1 L_1=\left[ \begin{smallmatrix}
		(E^1_1)^{-1}L_1\\0
		\end{smallmatrix}\right], \ \ \ B^v:=\left[ \begin{smallmatrix}
		-(E^1_1)^{-1}E^2_1\\I_{s}
		\end{smallmatrix}\right].
		$$
		Let $v(t)=\dot x_2(t)$, {then $v\in \mathcal C^0$} and it is clear that if $(x(t),u(t))=((x_1(t),x_2(t)), u(t))$ is a solution of (\ref{Eq:expl1a}), then $(x(t), u(t), v(t))$ solves (\ref{Eq:ODEcontrolnooutput1})   {with $(A,B^u,B^v)$ as above,} since
		$$
		\left[ \begin{smallmatrix}
		E^1_1&E^2_1
		\end{smallmatrix}\right] \left[ \begin{smallmatrix}
		\dot x_1(t)\\\dot x_2(t)
		\end{smallmatrix}\right] =H_1x_1(t)+L_1u(t) \Rightarrow \dot x_1(t)= (E^1_1)^{-1}H_1x(t)+(E^1_1)^{-1}L_1u(t)-(E^1_1)^{-1}E^2_1\dot x_2(t).
		$$
		Notice that if we choose another right inverse $\tilde E^{\dagger}_1$ of $E_1$ and {another} matrix $\tilde B^v$ such that ${\rm Im\,}\tilde B^v=\ker E_1$, then by Proposition \ref{Pro:LinDAEexpl},   equation (\ref{Eq:ODEcontrolnooutput1}) becomes
		$$
		\dot x=\tilde Ax+\tilde B^uu+\tilde B^{\tilde v}\tilde v \Leftrightarrow \dot x=  Ax+ B^uu+ B^{v}(F_vx+Ru+T^{-1}_v\tilde v).
		$$
		We thus {conclude} that there exists $\tilde v(t)=-T_vF_vx(t)-T_vRu(t)+T_vv(t)=-T_vF_vx(t)-T_vRu(t)+T_v\dot x_2(t)$ such that $(x(t),u(t),\tilde v(t))$ solves equation  (\ref{Eq:ODEcontrolnooutput1}). Therefore, $\Delta^u$ has corresponding solutions with  any $(Q,v)$-explicitation {independently of the} choice of $Q$, $E^{\dagger}_1$ and $B^v$. 
\end{proof}

\begin{proof}[Proof of Theorem \ref{Thm:equivalence}]
	Without loss of generality, we assume that the system matrices of $\Delta^u=(E,H,L)$ and $\tilde \Delta^{\tilde u}=(\tilde E,\tilde H,\tilde L)$ are of the following form:
	$$
	\begin{array}{llllll}
	E=\left[ \begin{smallmatrix}
	I_q&0\\
	0&0
	\end{smallmatrix}\right],&	H=\left[ \begin{smallmatrix}
	H_1\\
	H_2
	\end{smallmatrix}\right],&	L=\left[ \begin{smallmatrix}
	L_1\\
	L_2
	\end{smallmatrix}\right],&
	\tilde E=\left[ \begin{smallmatrix}
	I_{\tilde q}&0\\
	0&0
	\end{smallmatrix}\right], & \tilde H=\left[ \begin{smallmatrix}
	\tilde H_1\\
	\tilde H_2
	\end{smallmatrix}\right],&	L=\left[ \begin{smallmatrix}
	\tilde L_1\\
	\tilde L_2
	\end{smallmatrix}\right],
	\end{array}
	$$
	where $H_1\in \mathbb R^{q\times n}$, $L_1\in\mathbb R^{q\times m}$, $\tilde H_1\in \mathbb R^{\tilde q\times n}$, $\tilde L_1\in\mathbb R^{\tilde q\times m}$, $q={\rm rank\,} E$, $\tilde q={\rm rank\,} \tilde E$. Since if not, we can always find $Q,\tilde Q\in Gl(l,\mathbb R)$, $P,\tilde P\in Gl(n,\mathbb R)$ such that  $$(QEP^{-1},QHP^{-1},QL) \ \ {\rm and}  \ \ (\tilde Q \tilde E\tilde P^{-1},\tilde Q \tilde H\tilde P^{-1},\tilde Q \tilde L)$$ are of the above desired form  {and it} is easily seen that the ex-fb-equivalence of $(E,H,L)$ and $(\tilde E, \tilde H, \tilde L)$ is equivalent  to (implied by and implying) {that} of  $(QEP^{-1},QHP^{-1},QL)$ and $(\tilde Q \tilde E\tilde P^{-1},\tilde Q \tilde H\tilde P^{-1},\tilde Q \tilde L)$. Thus we can use the above system matrices to represent $\Delta^u$ and $\tilde \Delta^{\tilde u}$ in the remaining  {part of} proof.
	
	{By} the assumptions that $\Lambda^{uv}\in \mathbf{Expl}(\Delta^u)$ and $\tilde \Lambda^{\tilde u \tilde v}\in \mathbf{Expl}(\tilde \Delta^{\tilde u})$, we have 
	\begin{align}\label{Eq:chosedsysmatrics}
	\left[ {\begin{smallmatrix}
		{{A}}&{ B^u}& B^v\\
		{{C}}&{ D^u}&0
		\end{smallmatrix}} \right] = \left[ {\begin{smallmatrix}
		H_1&\vline& L_1&\vline&0\\
		0&\vline& 0&\vline&I_{n-q}\\
		\hline
		H_2&\vline& L_2&\vline&0
		\end{smallmatrix}} \right], \ \ \left[ {\begin{smallmatrix}
		{\tilde{A}}&{\tilde B^{\tilde u}}&\tilde B^{\tilde v}\\
		{\tilde{C}}&{\tilde D^{\tilde u}}&0
		\end{smallmatrix}} \right] = \left[ {\begin{smallmatrix}
		\tilde H_1&\vline&\tilde L_1&\vline&0\\
		0&\vline& 0&\vline&I_{n-\tilde q}\\
		\hline
		\tilde H_2&\vline& \tilde L_2&\vline&0
		\end{smallmatrix}} \right].
	\end{align}
	{We have chosen $\Lambda^{uv}$ and $\tilde \Lambda^{\tilde u \tilde v}$ as above for convenience, any other choice based on the explicitation procedure could have been made.}  Since {any} two ODECSs in an explicitation class {are} {EM-equivalent}, {the} choice of  a $(Q,v)$-explicitation  makes no difference {when proving} {EM-equivalence}. {Therefore,} we will use the system matrices in  (\ref{Eq:chosedsysmatrics}) for the following proof.
	
	\emph{If}. Suppose $\Lambda^{uv} \mathop  \sim \limits^{EM} \tilde \Lambda^{\tilde u \tilde v} $. Then there exist transformation matrices $T_x$, $T_u$, $T_v$, $T_y$, $F_u$, $F_v$, $R$, $K$ such that   (\ref{Eq:EME}) holds. {Substituting} the system matrices of (\ref{Eq:chosedsysmatrics}) into (\ref{Eq:EME}), we have 
	\begin{align}\label{Eq:pfTh1}
	\setlength{\arraycolsep}{4pt}
	\left[ {\begin{smallmatrix}
		\tilde H_1&\vline&\tilde L_1&\vline&0\\
		0&\vline& 0&\vline&I_{n-q}\\
		\hline
		\tilde H_2&\vline& \tilde L_2&\vline&0
		\end{smallmatrix}} \right]\! =\! \left[ {\begin{smallmatrix}
		T_x&{T_xK}\\
		0&{T_y}
		\end{smallmatrix}} \right]\left[ {\begin{smallmatrix}
		H_1&\vline& L_1&\vline&0\\
		0&\vline& 0&\vline&I_{n-\tilde q}\\
		\hline
		H_2&\vline& L_2&\vline&0
		\end{smallmatrix}} \right]\left[ {\begin{smallmatrix}
		{T_x^{-1}}&0&0\\
		{F_uT_x^{-1}}&{T^{-1}_u}&0\\
		(F_v+RF_u)T_x^{-1}&RT^{-1}_u&T^{-1}_v
		\end{smallmatrix}} \right].
	\end{align}
{Represent} $T_x\!=\!\left[ {\begin{smallmatrix}
		T_x^1&T_x^2\\
		T_x^3&T_x^4
		\end{smallmatrix}} \right]$, where $T_x^1\in \mathbb R^{q\times q}$. By 
	$
	\tilde B^{\tilde v}\!=\!T_xB^vT_v^{-1}$, we get $\left[ {\begin{smallmatrix}
		0\\
		I
		\end{smallmatrix}} \right]\!=\!\left[ {\begin{smallmatrix}
		T_x^1&T_x^2\\
		T_x^3&T_x^4
		\end{smallmatrix}} \right]\left[ {\begin{smallmatrix}
		0\\
		I
		\end{smallmatrix}} \right]T_v^{-1},$
	hence it can be deduced that  $q=\tilde q$ and $T^2_x=0$. Moreover, $T_x^4T^{-1}_v=I$ implies {that} $T^4_x$ is invertible. Thus by the invertibility of $T_x$, we have $T_x^1$ is  invertible as well.
	
	Subsequently, premultiply equation (\ref{Eq:pfTh1}) by $\left[ {\begin{smallmatrix}
		(T_x^1)^{-1}&0&0\\
		0&0&I_{l-q}
		\end{smallmatrix}} \right]$ {and} we get
	\begin{align*}
	\setlength{\arraycolsep}{2pt}
	\left[ {\begin{smallmatrix}
		(T_x^1)^{-1}&0\\
		0&I_{l-q}
		\end{smallmatrix}} \right]\left[ {\begin{smallmatrix}
		\tilde H_1&\vline&\tilde L_1&\vline&0\\
		\tilde H_2&\vline& \tilde L_2&\vline&0
		\end{smallmatrix}} \right] = \left[ {\begin{smallmatrix}
		I_q&{K_1}\\
		0&{T_y}
		\end{smallmatrix}} \right]\left[ {\begin{smallmatrix}
		H_1&\vline& L_1&\vline&0\\
		H_2&\vline& L_2&\vline&0
		\end{smallmatrix}} \right]\left[ {\begin{smallmatrix}
		{T_x^{-1}}&0&0\\
		{F_uT_x^{-1}}&{T^{-1}_u}&0\\
		(F_v+RF_u)T_x^{-1}&RT^{-1}_u&T^{-1}_v
		\end{smallmatrix}} \right],
	\end{align*}
	where $K_1=\left[ {\begin{smallmatrix}
		I_q&{(T_x^{1})}^{-1}T_x^{2}\end{smallmatrix}} \right]K$. It {follows} that 
	\begin{align*}
	\left[ {\begin{smallmatrix}
		\tilde H_1&\vline&\tilde L_1\\
		\tilde H_2&\vline& \tilde L_2
		\end{smallmatrix}} \right] = \left[ {\begin{smallmatrix}
		T_x^1&T_x^1{K_1}\\
		0&{T_y}
		\end{smallmatrix}} \right]\left[ {\begin{smallmatrix}
		H_1&\vline& L_1\\
		H_2&\vline& L_2
		\end{smallmatrix}} \right]\left[ {\begin{smallmatrix}
		{T_x^{-1}}&0\\
		{F_uT_x^{-1}}&{T^{-1}_u}
		\end{smallmatrix}} \right].
	\end{align*}
	Thus $ \Delta^u\mathop  \sim \limits^{ex-fb} \tilde \Delta^{\tilde u}$ via
	\begin{align*}
	Q=\left[ {\begin{smallmatrix}
		T_x^1&T_x^1{K_1}\\
		0&{T_y}
		\end{smallmatrix}} \right], \ \ \ P=T_x, \ \ \ F=F_u, \ \ \ G=T^{-1}_u.
	\end{align*} 
	
	\emph{Only if}. Suppose $ \Delta^u\mathop  \sim \limits^{ex-fb} \tilde \Delta^{\tilde u}$. Then there exist invertible matrices $Q$, $P$, and matrices $F$, $G$ of appropriate sizes such that equation (\ref{Eq:ex-fb-eq}) holds. {Represent} $Q=\left[ {\begin{smallmatrix}
		Q_1&Q_2\\
		Q_3&Q_4
		\end{smallmatrix}} \right]$, where $Q_1\in \mathbb R^{q\times q}$, and  $P^{-1}=\left[ {\begin{smallmatrix}
		P_1&P_2\\
		P_3&P_4
		\end{smallmatrix}} \right]$, where $P_1\in \mathbb R^{q\times q}$. Then by 	
	\begin{align*}
	\tilde E=QEP^{-1}\Rightarrow\left[ {\begin{smallmatrix}
		I_{\tilde q}&0\\
		0&0
		\end{smallmatrix}} \right]=\left[ {\begin{smallmatrix}
		Q_1&Q_2\\
		Q_3&Q_4
		\end{smallmatrix}} \right]\left[ {\begin{smallmatrix}
		I_{q}&0\\
		0&0
		\end{smallmatrix}} \right]\left[ {\begin{smallmatrix}
		P_1&P_2\\
		P_3&P_4
		\end{smallmatrix}} \right],
	\end{align*}
	we immediately get $q=\tilde q$ and $Q_1P_1=I$, $Q_1P_2=0$, $Q_3P_1=0$, which implies {that} $Q_1$, $P_1$ are invertible matrices, $P_2=0$, and $Q_3=0$. Thus by the invertibility of $Q$ and $P$, we have $Q_4$ and $P_4$ are invertible matrices as well.  Then by equation (\ref{Eq:ex-fb-eq}), we get
	\begin{align*}
	\left[ {\begin{smallmatrix}
		\tilde H_1&\vline&\tilde L_1\\
		\tilde H_2&\vline& \tilde L_2
		\end{smallmatrix}} \right] = \left[ {\begin{smallmatrix}
		Q_1&Q_2\\
		0&Q_4
		\end{smallmatrix}} \right]\left[ {\begin{smallmatrix}
		H_1&\vline& L_1\\
		H_2&\vline& L_2
		\end{smallmatrix}} \right]\left[ {\begin{smallmatrix}
		{P^{-1}}&0\\
		{FP^{-1}}&{G}
		\end{smallmatrix}} \right],
	\end{align*}
	which implies {that} the following equation holds:
	\begin{align*}
	\left[ {\begin{smallmatrix}
		\tilde H_1&\vline&\tilde L_1&\vline&0\\
		0&\vline& 0&\vline&I_{n-q}\\
		\hline
		\tilde H_2&\vline& \tilde L_2&\vline&0
		\end{smallmatrix}} \right] = \left[ {\begin{smallmatrix}
		Q_1&0&\vline&Q_2\\
		X&P_4^{-1}&\vline&0\\
		\hline
		0&0&\vline&Q_4
		\end{smallmatrix}} \right]\left[ {\begin{smallmatrix}
		H_1&\vline& L_1&\vline&0\\
		0&\vline& 0&\vline&I_{n-\tilde q}\\
		\hline
		H_2&\vline& L_2&\vline&0
		\end{smallmatrix}} \right]\left[ {\begin{smallmatrix}
		{P^{-1}}&0&0\\
		{FP^{-1}}&G&0\\
		Y&Z&P_4
		\end{smallmatrix}} \right],
	\end{align*}
	where $X=-P^{-1}_4P_3P^{-1}_1$, $Y=(P_3P_1^{-1}H_1+P_3P_1^{-1}L_1F)P^{-1}$, $Z=P_3P^{-1}_1L_1G$. So $\Lambda^{uv} \mathop  \sim \limits^{EM} \tilde \Lambda^{\tilde u \tilde v} $ via 
	$$
	\begin{array}{llll}
	T_x=P,& T_u=G^{-1},& T_v=P^{-1}_4,& T_y=Q_4,\\
	F_u=F,& F_v=P_3P_1^{-1}H_1,&R=P_3P_1^{-1}L_1,& K=\left[ {\begin{smallmatrix}
		P_1Q_2\\
		P_3Q_2
		\end{smallmatrix}} \right].
	\end{array}
	$$
\end{proof}
\subsection{{Proof of Proposition \ref{Pro:subspacesrelation}}}\label{ProofPro:subspacesrelation}
\begin{proof}
	Without loss of generality, we {may} assume {that} $\Delta^u_{l,n,m}=(E,H,L)$ is of the following form:
	\begin{align*}
	\left[ {\begin{smallmatrix}
		I_q&0\\
		0&0
		\end{smallmatrix}} \right]\left[ {\begin{smallmatrix}
		\dot x_1\\
		\dot x_2
		\end{smallmatrix}} \right]=\left[ {\begin{smallmatrix}
		H_1&H_2\\
		H_3&H_4
		\end{smallmatrix}} \right]\left[ {\begin{smallmatrix}
		x_1\\
		x_2
		\end{smallmatrix}} \right]+\left[ {\begin{smallmatrix}
		L_1\\
		L_2
		\end{smallmatrix}} \right]u,
	\end{align*}
	where $q={\rm rank\,} E$ and $H_1\in\mathbb R^{q\times q}$, $H_2\in\mathbb R^{q\times (n-q)}$, $H_3\in\mathbb R^{p\times q}$, $H_4\in\mathbb R^{p\times (n-q)}$, $L_1\in\mathbb R^{q\times m}$, $L_2\in\mathbb R^{p\times m}$, where $p=l-q$. Since if not, we can always find $Q\in Gl(l,\mathbb R)$, $P\in Gl(n,\mathbb R)$ such that  $ \tilde\Delta^{\tilde u}=(QEP^{-1},QHP^{-1},QL)$ is of the above form. Then, it is not hard to  {check}  that $\mathscr V_i( \tilde \Delta^{\tilde u})=P\mathscr V_i(\Delta^u)$, $\mathscr W_i( \tilde \Delta^{\tilde u})=P\mathscr W_i(\Delta^u)$, $\hat{\mathscr W}_i( \tilde \Delta^{\tilde u})=P\hat{\mathscr W}_i(\Delta^u)$. Moreover, for two ODECSs $\Lambda^w=\Lambda^{uv}\in \mathbf{Expl}(\Delta^u)$, $\tilde \Lambda^{\tilde w}=\tilde \Lambda^{\tilde u\tilde v}\in \mathbf{Expl}( \tilde \Delta^{\tilde u})$, we can verify {that} $\mathcal V_i(\tilde \Lambda^{\tilde w})=P\mathcal V_i(\Lambda^{w})$, $\mathcal W_i(\tilde \Lambda^{\tilde w})=P\mathcal W_i(\Lambda^w)$, $\hat{\mathcal W}_i(\tilde \Lambda^{\tilde w})=P\hat{\mathcal W}_i(\Lambda^w)$. Therefore, in order to show {that} the relations of the subspaces {(as claimed in Proposition \ref{Pro:subspacesrelation}) hold}, {replacing $\Delta^u$ by  $\tilde \Delta^{\tilde u}$ makes no difference and thus we will assume that $\Delta^u$ is of the above form in what follows.} 
	
	The following {system}, denoted  {$\Lambda^w=\Lambda^{uv}$}, is a $(Q,v)$-explicitation of $\Delta^u$,
	\begin{align}\label{Eq:controlsys4}
	\Lambda^w=\Lambda^{uv}:\left\lbrace \begin{array}{c@{\ }l}
	\left[ {\begin{smallmatrix}
		\dot x_1\\
		\dot x_2\\
		\end{smallmatrix}}\right] &=\left[ {\begin{smallmatrix}
		H_1&H_2\\
		0&0\\
		\end{smallmatrix}} \right]\left[ {\begin{smallmatrix}
		x_1\\
		x_2
		\end{smallmatrix}} \right]+\left[ {\begin{smallmatrix}
		L_1\\
		0
		\end{smallmatrix}} \right]u+\left[ {\begin{smallmatrix}
		0\\
		I_{n-q}
		\end{smallmatrix}} \right]v \\ 
	y&=H_3x_1+H_4x_2+L_2u.
	\end{array} \right.
	\end{align}
	Firstly, we calculate $\mathcal V_i(\Lambda^w)$ through equation (\ref{Eq:barV}) of the Appendix:
	\begin{align*}
	\setlength{\arraycolsep}{1.5pt}
	{{\mathcal V}_{i + 1}}(\Lambda^w) &= {\left[ {\begin{smallmatrix}
			A\\
			C
			\end{smallmatrix}} \right]^{ - 1}}\left(  {\left[ {\begin{smallmatrix}
			I\\
			0
			\end{smallmatrix}}\right]  {{\mathcal V}_i}(\Lambda^w) + {\rm Im\,}\left[ {\begin{smallmatrix}
			{B^w}\\
			{D^w}
			\end{smallmatrix}} \right]}\right) = {\left[  {\begin{smallmatrix}
			H_1&H_2\\
			0&0\\
			H_3&H_4
			\end{smallmatrix}}  \right]^{ - 1}}\left( {\left[ {\begin{smallmatrix}
			{{\mathcal V}_i}(\Lambda^w)\\
			0
			\end{smallmatrix}} \right] + {\rm Im\,}\left[ {\begin{smallmatrix}
			L_1&0\\
			0&I_{n-q}\\
			L_2&0
			\end{smallmatrix}} \right]}\right) \\
	&={\left[  {\begin{smallmatrix}
			H_1&H_2\\
			H_3&H_4
			\end{smallmatrix}}  \right]^{ - 1}}\left(  {\left[ {\begin{smallmatrix}
			[I_q,0]	{{\mathcal V}_i}(\Lambda^w)\\
			0
			\end{smallmatrix}} \right] + {\rm Im\,}\left[ {\begin{smallmatrix}
			L_1&0\\
			L_2&0
			\end{smallmatrix}} \right]}\right) =H^{-1}(E{{\mathcal V}_i}(\Lambda^w)+{\rm Im\,}L).
	\end{align*}
	{Comparing} the above {expression} with equation (\ref{Eq:V}) of {the} Appendix, it is {easily} seen that the subspace sequences $\mathcal V_{i+1}(\Lambda^w)$ and $\mathscr V_{i+1}(\Delta^u)$ are calculated in the same  {way}. {Since} $\mathscr V_0(\Delta^u)={\mathcal V_0}(\Lambda^w)=\mathbb R^n$, we conclude that $\mathscr V_i(\Delta^u)=\mathcal V_i(\Lambda^w)$ for $i\in \mathbb N$.
	
	Then calculate $\mathscr W_{i+1}(\Delta^u)$ via equation (\ref{Eq:W}) of {the} Appendix:
	\begin{align*}
	\mathscr W_{i+1}(\Delta^u)&=E^{-1}(H\mathscr W_i(\Delta^u)+{\rm Im\,}L)={\left[  {\begin{smallmatrix}
			I_q&0\\
			0&0
			\end{smallmatrix}}  \right]^{ - 1}}\left(  \left[  {\begin{smallmatrix}
		H_1&H_2\\
		H_3&H_4
		\end{smallmatrix}}  \right]\mathscr W_i(\Delta^u)+{\rm Im\,} \left[ {\begin{smallmatrix}
		L_1\\
		L_2
		\end{smallmatrix}} \right]\right)  \\
	&={\left[  {\begin{smallmatrix}
			I_q&0\\
			0&0
			\end{smallmatrix}}  \right]^{ - 1}}\left(  \left[  {\begin{smallmatrix}
		H_1&H_2&L_1&0\\
		H_3&H_4&L_2&0
		\end{smallmatrix}}  \right]\left[ {\begin{smallmatrix}
		\mathscr W_i(\Delta^u)\\
		\mathscr U_w
		\end{smallmatrix}} \right]\right) \\&
	=\left[  {\begin{smallmatrix}
		H_1&H_2&L_1&0\\
		0&0&0&0
		\end{smallmatrix}}  \right]\left(  \left[ {\begin{smallmatrix}
		\mathscr W_i(\Delta^u)\\
		\mathscr U_w
		\end{smallmatrix}} \right]\cap \ker \left[  {\begin{smallmatrix}
		H_3&H_4&L_2&0
		\end{smallmatrix}}  \right] \right) +{\rm Im\,}\left[ {\begin{smallmatrix}
		0\\
		I_{n-q}
		\end{smallmatrix}} \right].
	\end{align*}
	In the above formula, according to the special form of $E$, we directly calculate the preimage. Moreover, we can express 
	\begin{align*}
	\left[ {\begin{smallmatrix}
		0\\
		I_{n-q}
		\end{smallmatrix}} \right]= \left[  {\begin{smallmatrix}
		0&0&0&0\\
		0&0&0&I_{n-q}
		\end{smallmatrix}}  \right]\left( \left[ {\begin{smallmatrix}
		\mathscr W_i(\Delta^u)\\
		\mathscr U_w
		\end{smallmatrix}} \right]\cap \ker \left[  {\begin{smallmatrix}
		H_3&H_4&L_2&0
		\end{smallmatrix}}  \right]\right).
	\end{align*}
	It follows that
	\begin{align*}
	\mathscr W_{i+1}(\Delta^u)&=\left[  {\begin{smallmatrix}
		H_1&H_2&L_1&0\\
		0&0&0&I_{n-q}
		\end{smallmatrix}}  \right]\left( \left[ {\begin{smallmatrix}
		\mathscr W_i(\Delta^u)\\
		\mathscr U_w
		\end{smallmatrix}} \right]\cap \ker \left[  {\begin{smallmatrix}
		H_3&H_4&L_2&0
		\end{smallmatrix}}  \right] \right) \\&=\left[ {\begin{smallmatrix}
		A&B^w
		\end{smallmatrix}} \right]\left( {\left[ \begin{smallmatrix}
		{{\mathscr{W}}}_i(\Delta^u)\\
		{\mathscr{U}_w}
		\end{smallmatrix} \right] \cap \ker \left[ \begin{smallmatrix}
		C&D^w
		\end{smallmatrix} \right]}\right) .
	\end{align*}
	It is seen from the above equation and (\ref{Eq:barW}) of Appendix that the subspace sequences $\mathcal W_{i+1}(\Lambda^w)$ and ${\mathscr W_{i+1}(\Delta^u)}$ are calculated in the same  {way}. Since the initial conditions  ${\mathcal W_{0}}(\Lambda^w)=\mathscr W_{0}(\Delta^u)=\{0\}$, we conclude that $\mathcal W_{i+1}(\Lambda^w)={\mathscr W_{i+1}(\Delta^u)}$ for all $i\in \mathbb N$.
	
	Then from (\ref{Eq:W}) and (\ref{Eq:Waddition}), it is seen that the subspaces sequences $\mathscr W_i$ and $\hat{\mathscr W}_i$ are calculated in the same form, their difference comes from their initial conditions only. Similarly, from (\ref{Eq:barW}) and (\ref{Eq:barWaddition}), it is seen that $\mathcal W_i$ and $\hat{\mathcal W}_i$ have different initial conditions but evolve in the same way. Thus, by $ {\hat{\mathcal W}_{1}}(\Lambda^w)={\hat{\mathscr W}_{1}(\Delta^u)}=\ker E={\rm Im\,} B^v$, we get  $ {\hat{\mathcal W}_{i}}(\Lambda^w)={\hat{\mathscr W}_{i}(\Delta^u)}$ for all $i\in \mathbb N^+$.	
\end{proof}
\subsection{{Proof of Proposition \ref{Pro:MTF}}}\label{ProofPro:MTF}
\begin{proof}
	Observe that the transformation matrix $T_s$ {decomposes} the state space $\mathscr X$ of $\Lambda^u$ into {$\mathscr X=\mathscr X_1\oplus\mathscr X_2\oplus\mathscr X_3\oplus\mathscr X_4$}, where $\mathscr X_1=\mathcal V^* \cap {\mathcal W^*}$, $\mathscr X_1\oplus \mathscr X_2=\mathcal V^*$, $\mathscr X_1\oplus \mathscr X_3=\mathcal W^*$, ${\left( {{\mathcal V^*} + {\mathcal W^*}} \right) \oplus  {\mathscr X_4} = {\mathscr X}}$. The transformation matrix $T_i$ decomposes the input space $\mathscr U_u$ into {$\mathscr U_u=\mathscr U_1\oplus\mathscr U_2$}, where $\mathscr U_1=\mathcal U_u^*$, $\mathscr U_1\oplus \mathscr U_2=\mathscr U_u$. The transformation matrix $T_o$ decomposes the output space $\mathscr Y$ into {$\mathscr Y=\mathscr Y_1\oplus\mathscr Y_2$}, where $\mathscr Y_1={\mathcal Y}^*$, $\mathscr Y_1\oplus \mathscr Y_2=\mathscr Y$. Let $\Lambda'=(A',B',C',D')=M_{tran}(\Lambda^u)$, where $M_{tran}$ is the Morse transformation $M_{tran}=(T_s,T_i,T_o,0,0)$. Then consider the following equation and subspaces:
	\begin{align*}
	\begin{smallmatrix}
	\left[ {\begin{smallmatrix}
		A'&B'\\
		C'&D'
		\end{smallmatrix}} \right]=\left[ {\begin{smallmatrix}
		{{T_s}}&0\\
		0&{{T_o}}
		\end{smallmatrix}} \right]\left[ {\begin{smallmatrix}
		A&B^u\\
		C&D^u
		\end{smallmatrix}} \right]\left[ {\begin{smallmatrix}
		{T_s^{ - 1}}&0\\
		0&{T_i^{ - 1}}
		\end{smallmatrix}} \right] = \left[ {\begin{smallmatrix}
		{A_1^1}&{A_1^2}&{A_1^3}&{A_1^4}&\vline& {B_1^1}&{B_1^2}\\
		{A_2^1}&{A_2^2}&{A_2^3}&{A_2^4}&\vline& {B_2^1}&{B_2^2}\\
		{A_3^1}&{A_3^2}&{A_3^3}&{A_3^4}&\vline& {B_3^1}&{B_3^2}\\
		{A_4^1}&{A_4^2}&{A_4^3}&{A_4^4}&\vline& {B_4^1}&{B_4^2}\\
		\hline \\
		{C_3^1}&{C_3^2}&{C_3^3}&{C_3^4}&\vline& {D_3^1}&{D_3^2}\\
		{C_4^1}&{C_4^2}&{C_4^3}&{C_4^4}&\vline& {D_4^1}&{D_4^2}
		\end{smallmatrix}} \right],&\begin{smallmatrix}
	{{\mathcal V}^*}(\Lambda'):\left[ {\begin{smallmatrix}
		*\\
		*\\
		0\\
		0
		\end{smallmatrix}} \right],\\
	{{\mathcal U_u}^*}(\Lambda'):\left[ {\begin{smallmatrix}
		*\\
		0
		\end{smallmatrix}} \right],
	\end{smallmatrix}&\begin{smallmatrix}
	{{\mathcal W}^*}(\Lambda'):\left[ {\begin{smallmatrix}
		*\\
		0\\
		*\\
		0
		\end{smallmatrix}} \right],\\
	{{\mathcal Y}^*}(\Lambda'):\left[ {\begin{smallmatrix}
		*\\
		0
		\end{smallmatrix}} \right].
	\end{smallmatrix}
	\end{smallmatrix}
	\end{align*}
	Now, {applying (\ref{Eq:barI}), for $i=n$, to both $\Lambda'$ and the dual system of $\Lambda'$} ( {see Appendix}), we have
	\begin{align*}
	{\left[ {\begin{smallmatrix}
			B'\\
			D'
			\end{smallmatrix}} \right]} {{\mathcal U_u}^*} \subseteq \left[ {\begin{smallmatrix}
		{{{\mathcal V}^*}}\\
		0
		\end{smallmatrix}} \right], {\kern 10pt}{\left[ {\begin{smallmatrix}
			(C')^T\\
			(D')^T
			\end{smallmatrix}} \right]} {({\mathcal Y}^*)^{\bot}} \subseteq \left[ {\begin{smallmatrix}
		({\mathcal W}^*)^{ \bot }\\
		0
		\end{smallmatrix}} \right].
	\end{align*}
	It follows that $B^1_3$, $B^1_4$, $C^1_4$, $C^3_4$, $D^1_3$, $D^1_4$, $D_2^4$ are all zero.
	
	Then  {applying (\ref{Eq:barV}) for $i=n$, to both $\Lambda'$ and its dual system,} we have
	\begin{align}
	{\left[ {\begin{smallmatrix}
			A'{{\mathcal V}^*}\\
			C'{{\mathcal V}^*}
			\end{smallmatrix}} \right]} &\subseteq  {\left[ {\begin{smallmatrix}
			{{\mathcal V}^*}\\
			0
			\end{smallmatrix}} \right] + {\mathop{\rm Im\,}\nolimits}\left[ {\begin{smallmatrix}
			{B'}\\
			{D'}
			\end{smallmatrix}} \right]},\label{barV*} 	
	\\ {\left[ {\begin{smallmatrix}
			(A')^T{({\mathcal W}^*)^{\bot}}\\
			(B')^T{({\mathcal W}^*)^{\bot}}
			\end{smallmatrix}} \right]}& \subseteq  {\left[ {\begin{smallmatrix}
			{({\mathcal W}^*)^{\bot}}\\
			0
			\end{smallmatrix}} \right] + {\mathop{\rm Im\,}\nolimits}\left[ {\begin{smallmatrix}
			(C')^T\\
			(D')^T
			\end{smallmatrix}} \right]}.\label{barW*} 
	\end{align}
	The lower parts of equations (\ref{barV*}) and (\ref{barW*}) give  $C'{{\mathcal V}^*}\subseteq {\mathop{\rm Im\,}\nolimits}D'$ and $(B')^T{({\mathcal W}^*)^{\bot}}\subseteq {\mathop{\rm Im\,}\nolimits}(D')^T$,  which implies that $B_2^1$ {and} $C_2^4$ are zero. On the other hand, equation (\ref{barV*}) gives that 
	\begin{align*}
	{\mathop{\rm Im\,}\nolimits} \left[ {\begin{smallmatrix}
		{A_3^1}\\
		{A_4^1}\\
		{C_3^1}
		\end{smallmatrix}} \right] \subseteq {\mathop{\rm Im\,}\nolimits} \left[ {\begin{smallmatrix}
		{B_3^2}\\
		{B_4^2}\\
		{D_3^2}
		\end{smallmatrix}} \right] \ \ \ {\rm and} \ \ \ {\mathop{\rm Im\,}\nolimits} \left[ {\begin{smallmatrix}
		{A_3^2}\\
		{A_4^2}\\
		{C_3^2}
		\end{smallmatrix}} \right] \subseteq {\mathop{\rm Im\,}\nolimits} \left[ {\begin{smallmatrix}
		{B_3^2}\\
		{B_4^2}\\
		{D_3^2}
		\end{smallmatrix}} \right],
	\end{align*}
	{implying} that there exist matrices $F_1\in \mathbb{R}^{m_3\times n_1} $ and $F_2\in\mathbb{R}^{m_3\times n_2} $ such that 
	\begin{align}\label{Eq:F12}
	\left[ {\begin{smallmatrix}
		{A_3^1}\\
		{A_4^1}\\
		{C_3^1}
		\end{smallmatrix}} \right] = -\left[ {\begin{smallmatrix}
		{B_3^2}\\
		{B_4^2}\\
		{D_3^2}
		\end{smallmatrix}} \right]F_1 \ \ \ {\rm and} \ \ \  \left[ {\begin{smallmatrix}
		{A_3^2}\\
		{A_4^2}\\
		{C_3^2}
		\end{smallmatrix}} \right] =  -\left[ {\begin{smallmatrix}
		{B_3^2}\\
		{B_4^2}\\
		{D_3^2}
		\end{smallmatrix}} \right]F_2.
	\end{align} 
	Then {setting} $F = \left[ {\begin{smallmatrix}
		0&0&0&0\\
		{{F_1}}&{{F_2}}&0&0
		\end{smallmatrix}} \right]$, we have
	$$
	\left[ {\begin{smallmatrix}
		{{T_s}}&0\\
		0&{{T_o}}
		\end{smallmatrix}} \right]\left[ {\begin{smallmatrix}
		A&B^u\\
		C&D^u
		\end{smallmatrix}} \right]\left[ {\begin{smallmatrix}
		{T_s^{ - 1}}&0\\
		{T_i^{ - 1}F}&{T_i^{ - 1}}
		\end{smallmatrix}} \right]= \left[ {\begin{smallmatrix}
		{A_1^1+B_1^2F_1}&{A_1^2+B_1^2F_2}&{A_1^3}&{A_1^4}&\vline& {B_1^1}&{B_1^2}\\
		{A_2^1+B_2^2F_1}&{A_2^2+B_2^2F_1}&{A_2^3}&{A_2^4}&\vline& 0&{B_2^2}\\
		0&0&{A_3^3}&{A_3^4}&\vline& 0&{B_3^2}\\
		0&0&{A_4^3}&{A_4^4}&\vline& 0&{B_4^2}\\
		\hline \\
		0&0&{C_3^3}&{C_3^4}&\vline& 0&{D_3^2}\\
		0&0&0&{C_4^4}&\vline& 0&0
		\end{smallmatrix}} \right].
	$$
	Since $\mathcal W^*$ is feedback invariant, equation (\ref{barW*}) also holds for the above transformed system. Thus the upper part of (\ref{barW*}) becomes
	\begin{align*}
	(A'+B'F)^T{({\mathcal W}^*(\Lambda'))^{\bot}}
	\subseteq  
	{({\mathcal W}^*(\Lambda'))^{\bot}}
	+ {\mathop{\rm Im\,}\nolimits}(C')^T,	
	\end{align*}
	which gives that $({A_2^1}+B_1^2F_1)^T=0$,
	\begin{align*}
	{\mathop{\rm Im\,}\nolimits} \left[ {\begin{smallmatrix}
		({A_2^3})^T\\
		({B_2^2})^T
		\end{smallmatrix}} \right] \subseteq {\mathop{\rm Im\,}\nolimits} \left[ {\begin{smallmatrix}
		({C_1^3})^T\\
		({D_1^2})^T
		\end{smallmatrix}} \right]\ \ \ {\rm and} \ \ \ {\mathop{\rm Im\,}\nolimits} \left[ {\begin{smallmatrix}
		({A_4^3})^T\\
		({B_4^2})^T
		\end{smallmatrix}} \right] \subseteq {\rm Im\,} \left[ {\begin{smallmatrix}
		({C_3^3})^T\\
		({D_3^2})^T
		\end{smallmatrix}} \right].
	\end{align*}
	It follows {that} there exist $K_1\in\mathbb{R}^{n_2\times p_3} $ and $K_2\in \mathbb{R}^{n_4\times p_3}$ such that 
	\begin{align}\label{Eq:K12}
	\left[ {\begin{smallmatrix}
		({A_2^3})^T\\
		({B_2^2})^T
		\end{smallmatrix}} \right]=-\left[ {\begin{smallmatrix}
		({C_3^3})^T\\
		({D_3^2})^T
		\end{smallmatrix}} \right]K_1^T\ \ \ {\rm and} \ \ \ \left[ {\begin{smallmatrix}
		({A_4^3})^T\\
		({B_4^2})^T
		\end{smallmatrix}} \right]=-\left[ {\begin{smallmatrix}
		({C_3^3})^T\\
		({D_3^2})^T
		\end{smallmatrix}} \right]K^T_2.
	\end{align}
	Let $K={\left[ {\begin{smallmatrix}
			0&{K_1^T}&0&{K_2^T}\\
			0&0&0&0
			\end{smallmatrix}} \right]^T}$, {which implies} that
	\begin{align*}
	\left[ {\begin{smallmatrix}
		{{T_s}}&KT_o\\
		0&{{T_o}}
		\end{smallmatrix}} \right]\left[ {\begin{smallmatrix}
		A&B\\
		C&D
		\end{smallmatrix}} \right]\left[ {\begin{smallmatrix}
		{T_s^{ - 1}}&0\\
		{T_i^{ - 1}F}&{T_i^{ - 1}}
		\end{smallmatrix}} \right]\!=\!\left[ {\begin{smallmatrix}
		{A_1^1+B_1^2F_1}&{A_1^2+B_1^2F_2}&{A_1^3}&{A_1^4}&\vline& {B_1^1}&{B_1^2}\\
		0&{A_2^2+B_2^2F_1}&0&{A_2^4+K_1C^4_1}&\vline& 0&{0}\\
		0&0&{A_3^3}&{A_3^4}&\vline& 0&{B_3^2}\\
		0&0&{0}&{A_4^4+K_2C_3^4}&\vline& 0&{0}\\
		\hline\\
		0&0&{C_3^3}&{C_3^4}&\vline& 0&{D_3^2}\\
		0&0&0&{C_4^4}&\vline& 0&0
		\end{smallmatrix}} \right].
	\end{align*}  
	Now it is seen that there exist $K_{MT}=T_s^{-1}KT_o$ and $F_{MT}=T_i^{-1}FT_s$ such that $\tilde \Lambda^{\tilde u}=(\tilde A, \tilde B^{\tilde u}, \tilde C, \tilde D^{\tilde u})$ has the form   (\ref{Eq:MTF}), where
	\begin{align*}
	\left[ {\begin{smallmatrix}
		{\tilde{A}}&{\tilde B^{\tilde u}}\\
		{\tilde{C}}&{\tilde D^{\tilde u}}
		\end{smallmatrix}} \right]=\left[ {\begin{smallmatrix}
		{{T_s}}&{{T_s}K_{MT}}\\
		0&{{T_o}}
		\end{smallmatrix}} \right]\left[ {\begin{smallmatrix}
		A&B^u\\
		C&D^u
		\end{smallmatrix}} \right]\left[ {\begin{smallmatrix}
		{{T_s}^{ - 1}}&0\\
		{F_{MT}{T_s}^{ - 1}}&{{T_i}^{ - 1}}
		\end{smallmatrix}} \right].
	\end{align*}
	The system matrices of $\tilde \Lambda^u$, see (\ref{Eq:MTF}), are $\tilde A_1=A_1^1+B_1^2F_1$, $\tilde A_1^2=A_1^2$, $\tilde A_1^3=A_1^3$, $\tilde A_1^4=A_1^4$, ${\tilde B_1=B_1^1}$, ${\tilde B_1^2=B_1^2}$, $\tilde A_2 =A_2^2+B_2^2F_1$, $\tilde A_2^4={A_2^4+K_1C^4_1}$, ${\tilde A_3=A_3^3}$, ${\tilde A_3^4=A_3^4}$, $\tilde B_3={B_3^2}$, $\tilde A_4={A_4^4+K_2C_1^4}$, $ {\tilde C_3}={C_3^3}$, $\tilde C_3^4={C_3^4}$, $\tilde D_3={D_3^2}$, $\tilde C_4={C_4^4}$.
	
	{Now} we {will} show {that} $(\tilde A_1, \tilde B_1)$ is controllable. By   Lemma 4 of \cite{molinari1978structural} \red{applied} to $\tilde \Lambda^{\tilde u}$, we get
	\begin{align}\label{WiV*_U*}
	{\mathcal W_i}(\tilde \Lambda^{\tilde u}|_{\mathcal U_u^*})=\mathcal W_i(\tilde \Lambda^{\tilde u})\cap \mathcal V^*(\tilde \Lambda^{\tilde u}),
	\end{align}
	where  {${\mathcal W_i}(\tilde \Lambda^{\tilde u}|_{\mathcal U_u^*})$}  \blue{denotes} the subspace ${\mathcal W_i}$  when {the} input is restricted to ${\mathcal U_u^*}$. Use \blue{the} system matrices (\ref{Eq:MTF}) to calculate  {${\mathcal W_i}(\tilde \Lambda^{\tilde u}|_{\mathcal U_u^*})$}  and $\mathcal W_i(\tilde \Lambda^{\tilde u})\cap \mathcal V^*(\tilde \Lambda^{\tilde u})$,  {which gives}
	\begin{align}\label{Eq:barWI*}
	 {{\mathcal W_n}(\tilde \Lambda^{\tilde u}|_{\mathcal U_u^*})}=\mathscr{B}_1+\tilde A_1\mathscr{B}_1+ \cdots +(\tilde A_1)^{n-1}\mathscr{B}_1\mathop=\limits^{(\ref{WiV*_U*})}\mathcal W_n(\tilde\Lambda^{\tilde u})\cap \mathcal V^*(\tilde\Lambda^{\tilde u}),
	\end{align}
	where $\mathscr{B}_1={\mathop{\rm Im\,}\nolimits} {[ {\begin{smallmatrix}
			{\tilde B_1}&0&0&0
			\end{smallmatrix}} ]^T}$.
	We can see from the above equation that the reachability space of $(\tilde A_1, \tilde B_1)$ is $\mathcal W^*(\tilde\Lambda^{\tilde u})\cap \mathcal V^*(\tilde\Lambda^{\tilde u})=\mathscr X_1$, which implies {that} $(\tilde A_1, \tilde B_1)$ is controllable. Since the proof of  the observability of $( \tilde C_4, \tilde A_4)$ is completely dual {to} the above proof, we omit  {that} part.
	
	Subsequently, we  prove {that} the system $\Lambda^3=(\tilde A_3,\tilde B_3,\tilde C_3,\tilde D_3)$, given by (\ref{Eq:MTF}), is prime. {Using} the system matrices of $\tilde \Lambda^{\tilde u}$ to calculate  {$	{\mathcal W^*}(\tilde \Lambda^{\tilde u}|_{(\mathcal U_u^*)^{\bot}})$}, we get
	\begin{align*}
{\mathcal W^*}(\tilde \Lambda^{\tilde u}|_{(\mathcal U_u^*)^{\bot}})=\mathscr R\times \left\lbrace 0\right\rbrace \times
	{\mathcal W^*}(\tilde \Lambda^{3})\times \left\lbrace 0\right\rbrace ,
	\end{align*}
	where  {$\mathscr R$} denotes \blue{a subspace whose explicit form is} irrelevant. From ${\mathcal W^*}(\tilde \Lambda^{\tilde u})={\mathcal W^*}(\tilde \Lambda^{\tilde u}|_{\mathcal U_u^*})\oplus{\mathcal W^*}(\tilde \Lambda^{\tilde u}|_{(\mathcal U_u^*)^{\bot}})$ and equation (\ref{Eq:barWI*}), we can deduce  that
	$
	{\mathcal W^*}(\tilde \Lambda^{3})=\mathscr X(\tilde \Lambda^{3})=\mathscr X_3(\tilde \Lambda^{\tilde u}).
	$
	Moreover, by {a} direct calculation, we get 
	\begin{align*}
	\mathcal Y^*(\tilde \Lambda^3)=\mathscr Y(\tilde \Lambda^3)=\tilde C_3{\mathcal W^*}(\tilde \Lambda^3)+\tilde D_3\mathscr U_w(\tilde \Lambda^3), \ \ \mathcal V^*(\tilde \Lambda^3)=0,  \ \ \mathcal U_u^*(\tilde \Lambda^3)=0.
	\end{align*}	
	Finally, by Theorem 10 of \cite{molinari1978structural}, we {conclude that} $\tilde \Lambda^3=(\tilde A_3,\tilde B_3,\tilde C_3,\tilde D_3)$ is prime.
\end{proof}
\subsection{{Proof of Proposition \ref{Pro:MMNF}}}\label{ProofPro:MMNF}
\begin{proof}
	First, by \textbf{MNF} Algorithm \ref{Alg:1} and {a direct calculation}, we have	
	\begin{align*}
	\begin{array}{ll}
	\bar A_1=\tilde A_1+\tilde B_1F^1_{MN},&\bar A_1^3=\tilde A_1^3+\tilde B^2_1F^2_{MN}+K^1_{MN}\tilde C_3+K^1_{MN}\tilde D_3F^2_{MN},\\
	\bar A_4= \tilde A_4+K^3_{MN}\tilde C_3^4,&\bar A_3=\tilde A_3+K^2_{MN}\tilde C_3+\tilde B_3F^2_{MN}+K^2_{MN}\tilde D_3F^2_{MN},\\\bar B_3=\tilde B_3+K^2_{MN}\tilde D_3, &
	\bar A_1^4=\tilde A_1^4+\tilde B^2_1F^3_{MN}+K^1_{MN}\tilde C_3+K^1_{MN}\tilde D_3F^3_{MN},\\
	\bar B_1^2=\tilde B_1^2+K^1_{MN}\tilde D_3,&\bar A_3^4=\tilde A_3^4+\tilde B_3F^3_{MN}+K^2_{MN}\tilde C^4_3+K^2_{MN}\tilde D_3F^3_{MN},\\
	\bar C_3=\tilde C_3+\tilde D_3F^2_{MN},&\bar C_3^4=\tilde C_3^4+\tilde D_3F^3_{MN}.
	\end{array} 	\end{align*}		
	{We will show that we can always assume $\tilde D_3=0$. To this end, we can find  a change of coordinates in the input and output spaces to obtain $\tilde D_3=\left[ {\begin{smallmatrix}
			0&0\\
			0&{{I_\delta }}
			\end{smallmatrix}} \right]$.} Then by suitable choice of feedback and output injection transformation, the 5-tuple $(\tilde B_1^2,\tilde B_3,\tilde C_3,\tilde C_3^4,\tilde D_3)$  {can be brought into the following form:}
	\begin{align*}
	\left[ {\begin{smallmatrix}
		{*}&{*}&\vline& {\tilde B^2_1}\\
		{*}&{*}&\vline& {\tilde B_3}\\
		\hline \\
		{\tilde C_3}&{\tilde C_3^4}&\vline& {\tilde D_3}
		\end{smallmatrix}} \right]	\Rightarrow\left[ {\begin{smallmatrix}
		{*}&{*}&\vline& {\hat B^2_1}&0\\
		{*}&{*}&\vline& {\hat B_3}&0\\
		\hline \\
			{\hat C_3}&{\hat C_3^4}&\vline& 0& 0\\
		0&0&\vline& 0& I_{\delta}	
		\end{smallmatrix}} \right].
	\end{align*}
 The zero columns of $\hat B$ and the zero rows of $\hat C$ which correspond to the static relations $y_i=u_i$, $1\le i\le\sigma$, we will {be kept} unchanged. Now, by neglecting the zero columns of $\hat B$ and the zero rows of $\hat C$, we \blue{may assume that}
		$$
		\left[ {\begin{smallmatrix}
		{*}&{*}&\vline& {\tilde B^2_1}\\
		{*}&{*}&\vline& {\tilde B_3}\\
		\hline \\
		{\tilde C_3}&{\tilde C_3^4}&\vline& \tilde D_3
		\end{smallmatrix}} \right]=	\left[ {\begin{smallmatrix}
			{*}&{*}&\vline& {\hat B^2_1}\\
			{*}&{*}&\vline& {\hat B_3}\\
			\hline \\
			{\hat C_3}&{\hat C_3^4}&\vline&0
			\end{smallmatrix}} \right],
		$$
i.e., $\tilde D_3$-matrix is $\hat D_3=0$.	
	
	Now {with} the assumption $\tilde D_3=0$, we show that the constrained Sylvester equations of (\ref{Eq:CSlyvester}) can be reduced to normal Sylvester equations by {a} suitable choice of $F_{MN}$ and $K_{MN}$. We claim that the following matrix equation 
	\begin{align}\label{Eq:hatG2}
	\tilde B^2_1=-\hat T^2_{MN}\tilde B_3
	\end{align}
	is solvable for $\hat T^2_{MN}$. This claim can be proved by  {observing that}
	\begin{align}\label{Eq:barIbot}
	\left[ \begin{smallmatrix}
	\tilde B^{\tilde u} (\mathcal U_u^*)^{\bot}\\
	\tilde D^{\tilde u} (\mathcal U_u^*)^{\bot}
	\end{smallmatrix}\right] \cap \left[ \begin{smallmatrix} 
	\mathcal V^*\\
	0 
	\end{smallmatrix}\right] =0.
	\end{align} 
	Note that the above equation is {a consequence of} the definition of $\mathcal U_u^*$ (see equation (\ref{Eq:barI})). Now by (\ref{Eq:barIbot}), we have
	\begin{align*}
	{\rm Im\,}  ({\rm col}\left[ \begin{smallmatrix}
	\tilde B_1^2& 0&\tilde B_3&0&\tilde D_3&0
	\end{smallmatrix}\right])\cap \left[ \begin{smallmatrix} 
	\mathcal V^*\\
	0 
	\end{smallmatrix}\right] =0.
	\end{align*}	
	Since $\tilde D_3$ is already zero, the above equation  {implies} that (\ref{Eq:hatG2}) is solvable for $\hat T^2_{MN}$. Consequently, substitute 	(\ref{Eq:hatG2}) into the upper equations of (\ref{Eq:CSlyvester}) {and} we get
	\begin{align}\label{Eq:Csylvester2}
	\bar A_1\bar T^2_{MN} - \bar T^2_{MN}\bar A_3 =  - \bar A_1^3+\bar A^1\hat T^2_{MN}-\hat T^2_{MN}\bar A_3,\ \ \  \bar T^2_{MN}\bar B_3=0,
	\end{align}
	where $\bar T^2_{MN}=T^2_{MN}+\hat T^2_{MN}$. 
	
	Furthermore, {since} $(\tilde A_3,\tilde B_3,\tilde C_3,\tilde D_3)$ is prime ( {a consequence} of Proposition \ref{Pro:MTF}), we can always assume $\tilde B_3=[I_{m_3},0]^T$ and $\tilde C_3=[I_{p_3},0]$ (if not, use coordinates transformations such that $\tilde B_3$ and $\tilde C_3$ are of that form), where $m_3={\rm rank\,}\tilde B_3=\dim\, (\mathcal U_u^*)^{\bot}=p_3={\rm rank\,}\tilde C_3=\dim\, { \mathcal Y}^*$ . Then, it is possible to choose $K^1_{MN}$, $K^2_{MN}$, $F^2_{MN}$ such that the 4-tuple ($\bar A^3_1,\bar A_3,\bar B_1^2,\bar C_3$) {is transformed into} the following form:
	\begin{align*}
	\left[ {\begin{smallmatrix}
		{\bar A_1^{3}}&\vline&\\
		\hline \\
		{\bar A_3}&\vline& {\bar B_3}\\
		\hline \\
		{\tilde C_3}&\vline& 
		\end{smallmatrix}} \right]=\left[ {\begin{smallmatrix}
		0&{\bar A_1^{3'}}&\vline& \\
		\hline \\
		0&0&\vline& I_{m_3}\\
		0&{\bar A_3'}&\vline& 0\\
		\hline \\
		I_{p_3}&0&\vline& 
		\end{smallmatrix}} \right].
	\end{align*}
	Thus $\bar T^2_{MN}$ in equation (\ref{Eq:Csylvester2}) {is} of the form $\bar T^2_{MN}=[ {\begin{smallmatrix}
		0&\hat T^2_{MN}\\
		\end{smallmatrix}}]$  {because}  $\bar T^2_{MN}\bar B_3=0$. Hence, solving $\bar T^2_{MN}$ {via} equation (\ref{Eq:Csylvester2}) is equivalent to solving $\hat T^2_{MN}$ {via}
	\begin{align*}
	\bar A_1\left[  {\begin{smallmatrix}
		0&\hat T^2_{MN}\
		\end{smallmatrix}}\right] - \left[  {\begin{smallmatrix}
		0&\hat T^2_{MN}\\
		\end{smallmatrix}}\right] \left[ {\begin{smallmatrix}
		0&0\\
		0&\bar A_3'
		\end{smallmatrix}}\right]=  \left[ {\begin{smallmatrix}
		0&{\hat A_1^{3'}}\end{smallmatrix}}\right].
	\end{align*}
	Therefore, the upper part of {the} constrained Sylvester equations of (\ref{Eq:CSlyvester}) can be reduced to the above normal Sylvester equation. The reduction
	of {the} lower part of (\ref{Eq:CSlyvester}) to a normal Sylvester equation follows dually from the above result {and} we will omit that proof.
	
	Moreover, from Proposition \ref{Pro:MTF}, we have that the pair $(\tilde A_1, \tilde B_1)$ is controllable {and} the pair $(\tilde C_4,\tilde A_4)$ is observable. By {the} standard matrix theory, we can choose $F_{MN}$ {and} $K_{MN}$ such that the  {spectra} of $\bar A_1$, $\bar A_2$, $\bar A'_3$ ,and $\bar A_4^4$ are  {mutually} disjoint ( {that of $A_2$ is fixed but, the three others can be made arbitrary}). Then there exist unique solutions for $T^1_{MN}$, $T^2_{MN}$, $T^3_{MN}$, $T^4_{MN}$, $T^5_{MN}$ in (\ref{Eq:Slyvester}) and (\ref{Eq:CSlyvester}).  Furthermore, it is not hard to see that the state coordinates transformation matrix  {$T_{MN}$} brings $\tilde \Lambda^{\tilde u}$ {into} $\bar \Lambda^{\bar u}$. Feedback transformations preserve controllability, so  the controllability of $(\tilde A_1, \tilde B_1)$ implies the controllability of $(\bar A_1, \bar B_1)$; output injection preserves observability, so the observability of $(\tilde C_4,\tilde A_4)$  implies the observability of $(\bar C_4,\bar A_4)$.  The fact that the 4-tuple $(\bar A_3,\bar B_3,\bar C_3,\bar D_3)$ is prime is inherited from the fact that $(\tilde A_3,\tilde B_3,\tilde C_3,\tilde D_3)$ is prime since $(\tilde A_3,\tilde B_3,\tilde C_3,\tilde D_3)\mathop  \sim \limits^{M  }(\bar A_3,\bar B_3,\bar C_3,\bar D_3)$  (see  this property of prime systems in \cite{molinari1978structural}).
\end{proof}
\subsection{{Proofs of Theorem \ref{Thm:EMTF} and Theorem \ref{Thm:EMNF}}}\label{ProofThm:EMTFNF}
\begin{proof}[Proof of Theorem \ref{Thm:EMTF} ]
	Recall Remark \ref{rem:EM-equi}(iii) {that} there exists  an extended Morse transformation $EM_{tran}$ such that $\tilde \Lambda^{\tilde u \tilde v}=EM_{tran}(\Lambda^{uv})$ {is} of the \textbf{EMTF} if and only if there exists a Morse transformation $M_{tran}$ with a triangular ({and not just any}) input coordinates transformation bringing  $\Lambda^w_{n,m+s,p}=(A,B^w,C,D^w)$ into the \textbf{MTF}. Now we use the result of Proposition \ref{Pro:MTF}  for $\Lambda^w$ {with} a more subtle way to construct the input coordinates transformation matrix $T_w$.  {More specifically, set $T_x=T_s$, $T_y=T_o$, $F_w=F_{MT}$, $K_w=K_{MT}$ as in Proposition \ref{Pro:MTF} and define}
		\begin{align}\label{Eq:emtftm}
	T_w = [ {\begin{smallmatrix}
		{{T_u^1}}&{{T_u^3}}&T_v^1&T_v^3
		\end{smallmatrix}} ]^{-1}\in \mathbb{R}^{(m+s)\times (m+s)},
	\end{align}
where	   $T_u^1\in\mathbb{R}^{(m+s)\times m_1}$, $T_u^3\in\mathbb{R}^{(m+s)\times m_3}$, $T_v^1\in\mathbb{R}^{(m+s)\times s_1}$, $T_v^3\in\mathbb{R}^{(m+s)\times s_3}$ with $m_1+m_3=m$, $s_1+s_3=s$ are full rank matrices such that 
	\begin{align*}
	\begin{array}{ll}
	{\rm Im\,} {T_v^1} = \mathcal U^*_v,&	{\rm Im\,} {T_v^1} \oplus {\rm Im\,} {T_v^3} = { \mathscr U}_v,\\
	{\rm Im\,} {T_u^1} \oplus {\rm Im\,} {T_v^1} =\mathcal U_{uv}^*= \mathcal U_w^* ,&	{\rm Im\,} {T_u^1} \oplus{\rm Im\,} {T_u^3}\oplus{\rm Im\,} {T_v^1}\oplus{\rm Im\,} {T_v^3}= {\mathscr U}_{uv}={\mathscr U}_{w},
	\end{array} 
	\end{align*}
	 {where $\mathcal U_v^*$ is $\mathcal U_{uv}^*$ when the input {$w=[u^T \ v^T]^T$} is {restricted} to $v$ {(i.e., we put $u=0$)}. Notice that $T_w$ has a triangular form since ${\rm Im\,} {T_v^1} \oplus {\rm Im\,} {T_v^3} = { \mathscr U}_v$ and thus preserves $\mathscr U_u$.} Now the Morse transformation $M_{trans}=(T_x,T_w,T_y,F_{w},K_{w})$ brings $\Lambda^w$ into the desired form of (\ref{Eq:EMTF}). Hence, it proves that there exists an $EM_{tran}$ transforming $\Lambda^{uv}$ {into} the  \textbf{EMTF}. The claims that  $(\tilde A_{1}, \tilde B^{\tilde w}_{1})$ is controllable, $(\tilde C_{4},\tilde A_{4})$ is observable and $(\tilde A_{3},\tilde B^{\tilde w}_{3},\tilde C_{3},\tilde D^{\tilde w}_{3})$ is prime are inherited from the corresponding results of Proposition~\ref{Pro:MTF}.
\end{proof}
\begin{proof}[Proof of Theorem \ref{Thm:EMNF}]
There exists  an $EM_{tran}$ such that $\bar \Lambda^{\bar u\bar v}=EM_{tran}(\tilde \Lambda^{\tilde u \tilde v})$ is in the \textbf{EMNF} if and only if there exists a Morse transformation $M_{tran}$ with a triangular   input transformation matrix $T_w$ bringing {the} system $\tilde \Lambda^{\tilde w}$,  {given in \textbf{MTF},} into the \textbf{MNF}. Then as shown in Proposition \ref{Pro:MMNF}, the input coordinates transformation matrix of the Morse transformation, which brings the \textbf{MTF} into the \textbf{MNF}, is  {the identity matrix, thus  triangular, as we need}. Therefore, with the transformation matrices shown in Proposition \ref{Pro:MMNF}, we can always bring $\tilde \Lambda^{\tilde w}$ into the \textbf{EMNF}. Moreover,  the claims that $(\bar A_{1}, \bar B^{\bar w}_{1})$ is controllable,  $(\bar C_{4},\bar A_{4})$ is observable, $(\bar A_{3},\bar B^{\bar w}_{3},\bar C_{3},\bar D^{\bar w}_{3})$ is prime  {follow from} the corresponding results of Proposition~\ref{Pro:MMNF}.
\end{proof}
\subsection{{Proof of Theorem \ref{Thm:EMCF}}}\label{ProofThm:EMCF}
\begin{proof}
	By Theorem \ref{Thm:EMNF}, for a given ODECS $\Lambda^{uv}_{n,m,s,p}=(A,B^u,B^v,C,D^u)$, there exists an extended Morse transformation $EM_{tran}$ such that $EM_{tran}(\Lambda^{uv})$ is {in} the \textbf{EMNF}. Therefore, the starting point of this proof is the \textbf{EMNF} {given by} (\ref{Eq:EMNF}). Since the system { represented in the \textbf{EMNF}} is already decoupled {into} four independent subsystems, we only need to transform each subsystem into its {corresponding} canonical form.
	
	(i) We will prove that any controllable $\Lambda^{uv}_{n,m,s}=(A,B^u,B^v)$ can be transformed into the Brunovsk{\' y} canonical form with indices $(\epsilon_1,\ldots,\epsilon_m)$ and $(\bar \epsilon_1,\ldots,\bar \epsilon_s)$, then  the transformation from $(\bar A_1,\bar B^u_1,\bar B^v_1)$ to $ \left(\left[ {\begin{smallmatrix}
		A^{cu}&0\\
		0&A^{cv}
		\end{smallmatrix}} \right],\  \left[ {\begin{smallmatrix}
		B^{cu}\\
		0
		\end{smallmatrix}} \right], \ \left[ {\begin{smallmatrix}
		0\\
		B^{cv}
		\end{smallmatrix}} \right] \right) $ is straightforward to see.
	Since $ \Lambda^{uv}=(A,B^u,B^v)$ is a control system without output, in view of the extended Morse equivalence of Definition \ref{Def:EM-equi}, we just need to prove {that} there exist transformation matrices $T_x$, $T_u$, $T_v$, $F_u$, $F_v$, $R$ such that the transformed system matrices 
	\begin{align*}
	\left( T_x\left(A+ B^uF_u+B^v\left( F_v+RF_u\right)  \right)T^{-1}_x, T_x\left( B^u+ B^vR \right)T^{-1}_u, T_xB^vT^{-1}_v  \right)
	\end{align*}
	are {in} the Brunovsk{\'y} canonical form ({notice a triangular form of input transformation acting on $[B^u \ B^v]$}). First, from the classical linear system theory (see, e.g., \cite{brunovsky1970classification}), {using only a} state coordinates transformation and state feedback, i.e., {choosing} suitable $T_x$,  $F_v$, $F_u$,  and {setting} $T_u=I_{m}$, $T_v=I_s$, $R=0$, we can transform $\Lambda^{uv}$ into the following form:
	\begin{align}\label{Eq:chainform}
	\left\lbrace \begin{array}{c@{\ }lr}
	\dot x^j_i&=x^{j+1}_i, \ \ \ \ 1\le i \le m+s, \ \ 1\le j \le \kappa_i-1,\\
	\dot x_i^{\kappa_i}&=b^1_iu_1+\cdots+b^m_iu_m+\bar b^1_iv_1+\cdots+\bar b^s_iv_s, \ \ \ \ 1\le i \le m+s.
	\end{array}\right.
	\end{align}
	Moreover, without loss of generality, we assume ${\rm rank\,}B^w=m+s$ (if not, we can always permute the variables of $u$ and $v$ such that the first $m_1$ columns of $B^u$ and the first $s_1$ columns of $B^v$ are independent, where $m_1={\rm rank\,}B^u$ and $s_1={\rm rank\,}B^v$, then we will work with the matrices with these independent columns  {only, the remaining ones being zero by suitable transformations  $T_u$ and $T_v$}).   {Thus the matrix $\Gamma=[\Gamma_u \ \Gamma_v]$, where $\Gamma_u=(b^l_i)$ and  $\Gamma_v=(\bar b^{\bar l}_i)$, where $1\le i\le m+s$, $1\le l \le m$ and $1\le \bar l \le s$, is invertible.} Then we  suppose that the controllability indices $\kappa_i$ satisfy
	$$
	\kappa_1\ge \kappa_2\ge\dots \ge \kappa_{m+s}\ge1.
	$$
%
 Note that in the case of the Brunovsk{\'y} form for classical ODECS (with one kind of inputs), we {could} use $T_w=\Gamma$  as an input coordinates transformation matrix. However, $\Delta^{uv}$ has two kinds of inputs {and} the input coordinates transformation matrix should have a triangular form (see Remark \ref{rem:EM-equi}(ii)). In order to have such an input coordinates transformation matrix, we implement the following procedure.
%
%
\\Step $i=1$: two cases are possible: either for all $1\le j \le s$, we have $\bar b_1^j=0$ or there exists $1\le j\le s$ such that $\bar b_1^j\neq0$. In the first case, by the invertibility of $\Gamma$, there exists $1\le j\le m$ such that $b_1^j\neq0$. We assume $b^1_1\neq 0$ (if not, we permute the $u_j$'s), set $\ell_1=1$, $\epsilon_1=\kappa_1$, and $\bar \ell_1=0 $ and define
\begin{align*}
\left\lbrace \begin{array}{l}
 \tilde u_1=b^1_1u_1+\dots+b^m_1u_m\\ \tilde u_j=u_j, \ 2\le j\le m
\end{array}\right., \ \  \ \
\tilde v_j=v_j, \  1\le j\le s,
\end{align*}	
the system becomes (we delete "tildes" over $u_j$ and $v_j$)
$$
	\left\lbrace \begin{array}{c@{\ }l}
\dot x^j_i&=x^{j+1}_i,  \ \ 1\le i \le m+s,\ \   1\le j \le \kappa_i-1,\\
\dot x^{\epsilon_1}_1&=  u_1,\\
\dot x_i^{\kappa_i}&=b_i^1  u_1+\dots+b_i^m  u_m+\bar b_i^1v_1+\dots+\bar b_i^sv_s, \ \ 2\le i\le m+s.
\end{array}\right.
$$
In the second case, assume $\bar b^1_1\neq0$ (if not, we permute the $v_j$'s),   set $\bar \ell_1=1$, $\bar \epsilon_1=\kappa_1$, and $\ell_1=0$, and  {define
$$	
\left\lbrace \begin{array}{ll}
\tilde v_1=b^1_1u_1+\dots+b^m_1u_m+\bar b^1_1v_1+\dots+\bar b^m_1v_s,\\
\tilde v_i=v_i, \ \ 2 \le i \le s,
\end{array}\right. 
$$
and we get}
$$
\left\lbrace \begin{array}{c@{\ }l} 
\dot {\bar x}^{\kappa_1}_1&= \tilde v_1,\\
\dot x_i^{\kappa_i}&=\tilde b_i^1 u_1+\dots+\tilde b_i^m u_m+\tilde  {\bar b}_i^1\tilde v_1+\tilde  {\bar b}_i^2\tilde v_2+\dots+\tilde {\bar b}_i^s\tilde v_s, \ \ 2\le i\le m+s.
\end{array}\right.
$$
Set
$$
\left\lbrace \begin{array}{l}
\bar x^j_1=x_1^j, \ \ 1\le j \le \bar \epsilon_1,\\
\tilde x^j_i=x^j_i-\tilde {\bar b}^1_ix^{\kappa_1-\kappa_i+j}_1, \ \ 2\le i \le m+s,\ \ 1\le j \le   \kappa_i, 
\end{array}\right. 
$$
to get (we delete "tildes" over $x_i$,   $v_j$,  {$b_i$ and $\bar {b}_i$})
$$
\left\lbrace \begin{array}{c@{\ }l}
\dot x^j_i&=x^{j+1}_i, \ \ 1\le i \le m+s,\ \   1\le j \le \kappa_i-1, \\
\dot {\bar x}^{\bar \epsilon_1}_1&= v_1,\\
\dot x_i^{\kappa_i}&=b_i^1 u_1+\dots+b_i^m u_m+0+\bar b_i^2v_2+\dots+\bar b_i^sv_s, \ \ 2\le i\le m+s.
\end{array}\right.
$$
Step $i=k+1$: Assume that after $k$ steps, we have defined $\ell_k$ and $\epsilon_i$, for $1\le i\le \ell_k$, as well as $\bar \ell_k$ and $\bar \epsilon_i$  {for $1\le i\le\bar \ell_k$}, such that $\ell_k+\bar \ell_k=k$, and the system reads ( {the term ``0'' is to indicate that $v_1,\dots, v_{\bar \ell_k}$ are missing}) 
\begin{equation*}
\left\lbrace \begin{array}{l}
\begin{array}{c@{\ }ll}
\dot x^j_i&=x^{j+1}_i, & 1\le i \le \ell_k,\ \   1\le j \le \epsilon_i-1, \\
\dot x^{\epsilon_i}_i&=u_i, &1\le i \le \ell_k, \\
\dot {\bar x}^j_i&=\bar x^{j+1}_i, & 1\le i \le \bar \ell_k, \ \   1\le j \le \bar \epsilon_i-1, \\
\dot {\bar x}^{\bar \epsilon_i}_1&= v_i,& 1\le i \le \bar \ell_k,\\
\dot x^j_i&=x^{j+1}_i, & k+1\le i \le m+s, \ \   1\le j \le \kappa_i-1, \\
\end{array}\\
\ \ \dot x_i^{\kappa_i}=b_i^1 u_1+\dots+b_i^m u_m+0+\bar b_i^{\bar \ell_k+1}v_{\bar \ell_k+1}+\dots+\bar b_i^sv_s, \ \ k+1\le i\le m+s.
\end{array} \right.
\end{equation*}
Then two cases are possible, either for all $\bar \ell_k+1\le j \le s$, we have $\bar b^j_{k+1}=0$ or there exists $\bar \ell_k+1 \le j \le s$ such that $\bar b^j_{k+1}\ne0$. In the first case, set $\ell_{k+1}=\ell_k+1$, $\epsilon_{\ell_{k+1}}=\kappa_{k+1}$, $\bar \ell_{k+1}=\bar \ell_k$ and set
$$
\left\lbrace \begin{array}{l}
\tilde u_j=b^1_{k+1}u_1+\dots+b^m_{k+1}u_m, \ \ j=\ell_{k+1},\\ 
\tilde u_j=u_j, \ \ \ell_{k+1}+1\le j\le m,\\
\tilde v_j=v_j, \ \ 1\le j\le s,
\end{array}\right. 
$$
 {which is well-defined because, by controllability, at least one $b^j_{k+1}\ne0$, for $j>\ell_k$.} We get  (we delete "tildes" over $x_i$, $u_j$ and $v_j$)
$$
\left\lbrace \begin{array}{l}
\begin{array}{c@{\ }ll}
\dot x^j_i&=x^{j+1}_i,&1\le i \le \ell_{k+1}, \ \   1\le j \le \epsilon_i-1, \\
\dot x^{\epsilon_i}_i&=u_i, &1\le i \le \ell_{k+1} \\
\dot {\bar x}^j_i&=\bar x^{j+1}_i, &1\le i \le \bar \ell_{k+1}=\bar \ell_k,\ \   1\le j \le \bar \epsilon_i-1, \\
\dot {\bar x}^{\bar \epsilon_i}_1&= v_i,&1\le i \le \bar \ell_{k+1}=\bar\ell_k,\\
\dot x^j_i&=x^{j+1}_i, & k+2\le i \le m+s, \ \   1\le j \le \kappa_i-1, \\
\end{array}\\
\ \ \dot x_i^{\kappa_i}=b_i^1 u_1+\dots+b_i^m u_m+0+\bar b_i^{\bar \ell_k+1}v_{\bar \ell_k+1}+\dots+\bar b_i^sv_s, \ \ k+2\le i\le m+s.
\end{array}\right. 
$$
In the second case, assume \blue{$\bar b^{\bar \ell_{k}+1}_{k+1}\neq0$} (if not, we permute the $v_j$'s), set $\bar \ell_{k+1}=\bar \ell_k+1$, $\bar \epsilon_{\bar \ell_{k+1}}=\kappa_{k+1}$, and $\ell_{k+1}=\ell_k$, and define
$$
\left\lbrace \begin{array}{l}
	\tilde v_j=b^1_{k+1}u_1+\dots+b^m_{k+1}u_m+b_{k+1}^{\bar \ell_k+1}v_{\bar \ell_k+1}+\dots+\bar b_{k+1}^sv_s, \ \ j=\bar \ell_{k+1},\\ 
\tilde v_j=v_j, \ \ j\ne\ell_{k+1},
\end{array}\right. 
$$
 {we get
$$
\left\lbrace \begin{array}{c@{\ }l} 
\dot {\bar x}^{\kappa_{k+1}}_{k+1}&= \tilde v_{\bar \ell_{k+1}},\\
\dot x_i^{\kappa_i}&=\tilde b_i^1 u_1+\dots+\tilde b_i^m u_m+\tilde  {\bar b}_i^1\tilde v_1+\tilde  {\bar b}_i^2\tilde v_2+\dots+\tilde {\bar b}_i^s\tilde v_s, \ \ k+1\le i\le m+s.
\end{array}\right.
$$}
Set
$$	
\left\lbrace \begin{array}{l}
	\tilde x^j_i=x^j_i-\tilde {\bar b}^{\bar \ell_{k+1}}_ix^{\kappa_{k+1}-\kappa_i+j}_{k+1}, \ \ k+2\le i \le m+s,\ \ 1\le j \le \bar \kappa_i\\
	\bar x^j_i=x^j_{k+1}, \ \ i=\bar \ell_{k+1}, \ \ 1\le j\le \bar \epsilon_{\bar \ell_{k+1}},
\end{array}\right. 
$$
to get (we delete "tildes" over $x_i$, $v_j$, $b_i$, $\bar b_i$)
$$
\left\lbrace 
\begin{array}{l}
\begin{array}{c@{\ }ll}
\dot x^j_i&=x^{j+1}_i,& 1\le i \le \ell_{k+1}=\ell_k, \ \   1\le j \le \epsilon_i-1,\\
\dot x^{\epsilon_i}_i&=u_i, & 1\le i \le \ell_{k+1}=\ell_k, \\
\dot {\bar x}^j_i&=\bar x^{j+1}_i, &1\le i \le \bar \ell_{k+1},\ \  1\le j \le \bar \epsilon_i-1, \\
\dot {\bar x}^{\bar \epsilon_i}_1&= v_i,&1\le i \le \bar \ell_{k+1},\\
\dot x^j_i&=x^{j+1}_i, & k+2\le i \le m+s,\ \   1\le j \le \kappa_i-1,\\
\end{array}\\
\ \ \dot x_i^{\kappa_i}=b_i^1 u_1+\dots+b_i^m u_m+0+\bar b_i^{\bar \ell_{k+1}+1}v_{\bar \ell_{k+1}+1}+\dots+\bar b_i^sv_s, \ \ k+2\le i\le m+s.
\end{array}\right.
$$
After $m+s$ steps, we have  {$\ell_{m+s}=m$ and $\bar \ell_{m+s}=s$} and we get the Brunovsk{\'y} canonical form  of $\Lambda^{uv}$ with indices $(\epsilon_1,\ldots,\epsilon_m)$ and $(\bar \epsilon_1,\ldots,\bar \epsilon_s)$:
$$
\left\lbrace \begin{array}{c@{\ }l}
\dot x^j_i&=x^{j+1}_i, \ \   1\le j \le \epsilon_i-1, \ \ 1\le i \le \ell_{m+s}=m,\\
\dot x^{\epsilon_i}_i&=u_i, \ \ 1\le i \le \ell_{m+s}=m, \\
\dot {\bar x}^j_i&=\bar x^{j+1}_i, \ \   1\le j \le \bar \epsilon_i-1, \ \ 1\le i \le \bar \ell_{m+s}=s,\\
\dot {\bar x}^{\bar \epsilon_i}_1&= v_i,\ \ 1\le i \le \bar \ell_{m+s}=s.
\end{array}\right.
$$
	
	(ii) {The $A^{nn}$-matrix (corresponding to the uncontrollable and unobservable system) is} $A^{nn}=\bar A_{2}$.
	
	(iii) First,  we can find a Morse transformation  $M^1_{tran}$ with {a} triangular $T_w$ such that
	\begin{align*}
	M^1_{tran}\left( {\begin{smallmatrix}
		{\bar A_{3}}&\vline& \bar B^u_{3}&\vline& \bar B^v_{3}\\
		\hline \\
		\bar C_{3}&\vline& \bar D^u_{3}&\vline& {}
		\end{smallmatrix}} \right)=\left( {\begin{smallmatrix}
		{ A_p}&\vline& B^u_p&0&\vline& B^v_p\\
		\hline \\
		{C_p}&\vline& 0&0&\vline& {}\\
		0&\vline& 0&I_{\delta}&\vline& {}
		\end{smallmatrix}} \right).
	\end{align*}
	Since $(\bar A_{3},\bar B^w_{3},\bar C_{3},\bar D^w_{3})$ is prime, by Theorem 10 of \cite{molinari1978structural}, $(A_p, B^w_p,C_p)$  enjoys the properties:
	\begin{align}
	&\mathcal V^*(A_p,B^w_p,C_p)=0,  \ \ \  \mathcal U_w^*(A_p,B^w_p,C_p)=0. \label{Eq:VI=0} \\
	&\mathcal W^*(A_p,B^w_p,C_p)=\mathbb{R}^{n_3}, \ \  \  { \mathcal Y}^*(A_p,B^w_p,C_p)=\mathscr Y. \label{Eq:WO=R}
	\end{align}
	A little thought (or see Lemma 2 of \cite{molinari1978structural}) and equation (\ref{Eq:VI=0}) give that $\left[ {\begin{smallmatrix}
		{ A_p}& B^w_p\\
		{C_p}& 0
		\end{smallmatrix}} \right]$ is of full column rank. Then by  $\mathcal V^*(A_p,B^w_p,C_p)=(\mathcal W^*((A_p)^T,(C_p)^T,(B^w_p)^T))^{\bot}$ (see {also} the results of (\ref{Eq:dualinvarsubsp}) {below}) and  equation (\ref{Eq:WO=R}), we have $\left[ {\begin{smallmatrix}
		{ A_p}& B^w_p\\
		{C_p}& 0
		\end{smallmatrix}} \right]$ is of full row rank. Thus $\left[ {\begin{smallmatrix}
		{ A_p}& B^w_p\\
		{C_p}& 0
		\end{smallmatrix}} \right]$ is square and invertible.
	
	Moreover, by item (i) of this proof, there exists a Morse transformation $M^2_{tran}$ with triangular $T_w$ such that the pairs $(\hat A^{pu},\hat B^{pu})$ and $(A^{pv},B^{pv})$ below are {in} the Brunovsk{\'y} form with indices $(\sigma_{1},\dots,\sigma_{c})$ and $(\bar \sigma_{1},\dots,\bar \sigma_{d})$, respectively
	\begin{align*}
	M^2_{tran}\left( {\begin{smallmatrix}
		{ A_p}&\vline& B^u_p&\vline& B^v_p\\
		\hline \\
		{C_p}&\vline& 0&\vline& {}
		\end{smallmatrix}} \right)=\left( {\begin{smallmatrix}
		{ \hat A^{pu}}&0&\vline& \hat B^{pu}&\vline& 0\\
		{ 0}&A^{pv}&\vline& 0&\vline& B^{pv}\\
		\hline \\
		{\hat C^{u}}&{C^{v}}&\vline& 0&\vline& {}
		\end{smallmatrix}} \right).
	\end{align*}
	Then, {according} to the block-diagonal structure of $\hat A^{pu}$ and $A^{pv}$, the matrices $\hat C^{u}$ and $C^{v} $ above have the form:
	\begin{align*}
	\begin{small}
\hat C^{u}=\left[ {\begin{matrix}
	{\hat C_1^{u}}&\vline& {\hat C_2^{u}}&\vline&  \cdots &\vline& {\hat C_c^{u}}
	\end{matrix}} \right],\ \ \ \ 	C^{v}=\left[ {\begin{matrix}
	{C_1^{v}}&\vline& {C_2^{v}}&\vline&  \cdots &\vline& {C_d^{v}}
	\end{matrix}} \right],
	\end{small}
	\end{align*}
	where $\hat C_i^{u}\in \mathbb R^{p_3\times \sigma_{i}}$, $1\le i\le c$ and $C_i^{v}\in \mathbb R^{p_3\times \bar \sigma_{i}}$, $1\le i\le d$. 
	
	Now  the diagonal  submatrices $(\hat A^{pu}_i,\hat B^{pu}_i,\hat C_i^{u})$ of $(\hat A^{pu},\hat B^{pu},\hat C^{u})$, for $1\le i\le c$,  and $(A^{pv}_{i},B^{pv}_{i},C_i^{v})$ of $(A^{pv},B^{pv},C^{v})$, for $1\le i\le d $,  have to satisfy
	\begin{align}\label{Eq:WR}
	\mathcal W^*(\hat A^{pu}_{i},\hat B^{pu}_{i},\hat C_i^{u})=\mathbb{R}^{\sigma_{i}}, \ \ \ \
	\mathcal W^*(A^{pv}_{i},B^{pv}_{i},C_i^{v})=\mathbb{R}^{\bar \sigma_{i}},
	\end{align}
	since if not, equation (\ref{Eq:WO=R}) {does} not hold. 
	
	By {a direct} calculation, we have $\mathcal W_1(\hat A^{pu}_{i},\hat B^{pu}_{i},\hat C_i^u)={\rm Im\,}\hat B^{pu}_{i}$ and $\mathcal W_1(A^{pv}_{i},B^{pv}_{i},C_i^v)={\rm Im\,}B^{pv}_{i}$. Then the subspaces $\mathcal W_2(\hat A^{pu}_{i},\hat B^{pu}_{i},\hat C_i^u,0)$ and $\mathcal W_2(A^{pv}_{i},B^{pv}_{i},C_i^v,0)$  {coincide with} ${\rm Im\,}\hat B^{pu}_{i}$ and ${\rm Im\,}B^{pv}_{i}$, respectively, unless the last columns of $\hat C_i^u$ and $C_i^v$ are zero vectors. By similar arguments, we can deduce {that} $\hat C_i^u$, $1\le i\le c$ and  $C_i^v$, $1\le i\le d$ have the following form:
	\begin{align*}
	\begin{small}
 \hat C^u_i=\left[ {\begin{matrix}
	{\hat c_i^u}&\vline& 0&\vline&  \cdots &\vline& 0
	\end{matrix}} \right],\ \ \ \ 	C^v_i=\left[ {\begin{matrix}
	{c_i^v}&\vline& 0&\vline&  \cdots &\vline& 0
	\end{matrix}} \right],
	\end{small}
	\end{align*}
	where $\hat c_i^u\in \mathbb R^{p_3}$ and $c_i^v\in \mathbb R^{p_3}$. Furthermore, since the columns of $\hat A^{pu}_{i}$ and $A^{pv}_{i}$ corresponding to  $\hat c_i^u$ and $c_i^v$ are all zero, so  {by the inveritibility of $\left[ {\begin{smallmatrix}
		{ A_p}& B^w_p\\
		{C_p}& 0
		\end{smallmatrix}} \right]$,}  we {see that} the following matrix
	\begin{align*}
	T^{-1}_y=\left[ {\begin{matrix}
		{\hat c_1^u}&\hat c_2^u& \dots&\hat c_c^u& \vline& {c_1^v}&c_2^v& \dots&c_d^v
		\end{matrix}} \right]
	\end{align*}
	is invertible. Finally, {using} $T_y$ as {an} output coordinates transformation matrix,
	we get the following canonical form {for} $C_p$
	$$
	T_yC_p= {T_y}\left[ {\begin{matrix}
		{\hat C^{u}}&C^{v}
		\end{matrix}} \right]= \left[ {\begin{matrix}
		{\hat {C}^{pu}}&0\\
		0&{{C^{pv}}}
		\end{matrix}} \right].
	$$
	
	(iv) The proof of  {transforming} $(\bar A_4^4,\bar C_2^4)$  {into} $(A^o,C^o)$ is omitted since it is well-known in {the} linear control theory.
\end{proof}
\section{Conclusion}\label{sec:7}
In this paper, on one hand, for linear ODECSs, we modify and simplify the construction of the \textbf{MCF} {given in} \cite{molinari1978structural} by proposing {the} Morse triangular form \textbf{MTF}. {On} the other hand,  {a} bridge from the \textbf{MTF} of ODECSs to  the \textbf{FBCF} of DACSs is constructed via the explicitation with driving {variables} procedure. It is shown that, after attaching a class of ODECSs with two kinds of inputs to a DACS, we can find  connections between their geometric subspaces and canonical forms.  Finally, an {explicit} algorithm  {for} constructing  transformations from the \textbf{MTF}  {into} the \textbf{FBCF} is proposed via the explicitation procedure and an example is given to show {how our results and algorithms can be applied to physical systems}. 
\section*{Appendix} 
Recall the following  geometric subspaces for DACSs  (see e.g. \cite{ozccaldiran1986geometric},\cite{berger2013controllability}) {of the form $\Delta^u:E\dot x=Hx+Lu$. }
\begin{defn}
	Consider a DACS $\Delta^u_{l,n,m}=(E,H,L)$. A subspace $\mathscr V\subseteq \mathbb R^{n}$ is called $(H,E; {\rm Im\,}L)$-invariant if
	\begin{align*}
	H\mathscr V \subseteq E\mathscr V+ {\rm Im\,}L.
	\end{align*}
	A subspace $\mathscr W \subseteq \mathbb R^{n}$ is called restricted $(E,H;{\rm Im\,} L)$-invariant if 
	\begin{align*}
	\mathscr W  {=}E^{-1}(H\mathscr V+ {\rm Im\,}L).
	\end{align*}
\end{defn}
\begin{defn}\label{Def:augWong}
	For a DACS $\Delta^u_{l,n,m}=(E,H,L)$, define the augmented Wong sequences as follows:
	\begin{align}\label{Eq:V}
	\begin{array}{clcl}
	\mathscr V_0=\mathbb R^n, \ \ &\mathscr V_{i+1}=H^{-1}(E\mathscr V_{i}+{\rm Im\,}L),
	\end{array}\\
	\begin{array}{clcl}
	\mathscr W_0= 0,  \ \ &\mathscr W_{i+1}=E^{-1}(H\mathscr W_{i}+{\rm Im\,}L). 
	\end{array}	\label{Eq:W}
	\end{align}
	Additionally, define the {sequence of subspaces} $\hat{\mathscr W}_i$ as follows:
	\begin{align}\label{Eq:Waddition}
	\hat{\mathscr W}_1= \ker E, \ \  \hat{\mathscr W}_{i+1}=E^{-1}(H\hat{\mathscr W}_{i}+{\rm Im\,}L).
	\end{align}	
\end{defn}
{Consider} an ODECS $\Lambda^{uv}_{n,m,s,p}=(A,B^u,B^v,C,D)$ of the form$$\Lambda^{uv}:\left\lbrace  {\begin{array}{*{20}{l}}
	\dot x=Ax+B^uu+B^vv\\
	y = Cx+D^uu.
	\end{array}}\right.	$$The state, input and output space of $\Lambda^{uv}$ will be denoted by $\mathscr X$, $\mathscr{U}_{uv}$ and $\mathscr{Y}$, respectively. The input subspaces of $u$ and $v$ will be denoted by $\mathscr{U}_u$ and  $\mathscr{U}_v$, respectively. Thus we have $\mathscr{U}_{uv}=\mathscr{U}_u\oplus\mathscr{U}_v$. Recall that $\Lambda^{uv}$ can be expressed as a classical ODECS $\Lambda^w_{n,m+s,p}=(A,B^w,C,D^w)$ of the form (\ref{Eq:ODEcontrol}). The input space of $\Lambda^w$ is denoted by $\mathscr U_w$, and, clearly, $\mathscr U_w=\mathscr U_{uv}=\mathscr U_{u}\oplus\mathscr U_{v}$. We now recall the invariant subspaces $\mathcal V $ and $\mathcal W$ {defined} in \cite{molinari1976strong} and \cite{molinari1978structural} for $\Lambda^w$ (generalizing the classical invariant subspaces \cite{basile1992controlled,wonham1970decoupling,wonham1974linear} \red{given} for $D^u=0$).
\begin{defn}\label{Def:invariantsub}
	For an ODECS $\Lambda^{w}_{n,m+s,p}=(A,B^w,C,D^w)$, a subspace $\mathcal V \subseteq \mathbb{R}^n$ is called a null-output $(A,B^w)$-controlled invariant subspace if there exists $F\in \mathbb{R}^{(m+s)\times n}$ such that 
	\[(A+B^wF)\mathcal V\subseteq\mathcal V \ \ \ {\rm and }\ \ \ (C+D^wF)\mathcal V=0\]
	and a subspace $\mathcal U_w \subseteq \mathbb{R}^{s+m}$ is called a null-output $(A,B^w)$-controlled invariant input subspace if
	\[ \mathcal U_w=(B^w)^{-1}\mathcal V\cap \ker D^w.\]  
	Denote by $\mathcal V^*$ (respectively $\mathcal U_w^*$)  the largest null-output $(A,B^w)$ controlled invariant subspace (respectively input subspace).
	
	Correspondingly, a subspace $\mathcal W \subseteq \mathbb{R}^n$ is called an unknown-input $(C,A)$-conditioned invariant subspace if there exists $K\in \mathbb{R}^{n\times p}$ such that 
	$$(A+KC)\mathcal W+(B^w+KD^w) {\mathscr{U}_w}=\mathcal W$$ 
	and a subspace $\mathcal Y \subseteq \mathbb{R}^p$ is called an unknown-input $(C,A)$-conditioned invariant output subspace if 
	$$ \mathcal Y=C\mathcal W+ D^w{\mathscr{U}_w}.$$ 
	Denote by $\mathcal W^*$ (respectively $\mathcal Y^*$) the smallest unknown-input $(C,A)$-conditioned invariant subspace (respectively output subspace).
\end{defn}
\begin{lem}{\rm \cite{molinari1976strong}}\label{Lem:invarsubspaceseq}
	Initialize $\mathcal V_0={\mathscr{X}}=\mathbb{R}^n$ and, for $i\in 
	\mathbb{N}$, define inductively
	\begin{align}\label{Eq:barV}
	{{\mathcal V}_{i + 1}} = {\left[ {\begin{smallmatrix}
			A\\
			C
			\end{smallmatrix}} \right]^{ - 1}}\left( {\left[ {\begin{smallmatrix}
			I\\
			0
			\end{smallmatrix}} \right]{{\mathcal V}_i} + {\rm Im\,}\left[ {\begin{smallmatrix}
			{B^w}\\
			{D^w}
			\end{smallmatrix}} \right]} \right) 
	\end{align}
	and $\mathcal U_i\subseteq {\mathscr{U}}$ for $i\in \mathbb{N}$ are given by
	\begin{align}\label{Eq:barI}
	{{\mathcal U}_i} = {\left[ {\begin{smallmatrix}
			B^w\\
			D^w
			\end{smallmatrix}} \right]^{ - 1}}\left[ {\begin{smallmatrix}
		{{{\mathcal V}_i}}\\
		0
		\end{smallmatrix}} \right].
	\end{align}
	{Then}  $\mathcal V^*=\mathcal V_n$ and $\mathcal U_w^*=\mathcal U_n$.
	
	Correspondingly, initialize $\mathcal W_0=\{0\}$ and, for $i\in 
	\mathbb{N}$, define inductively
	\begin{align}\label{Eq:barW}
	\mathcal W_{i + 1}= \left[ {\begin{smallmatrix}
		A&B^w
		\end{smallmatrix}} \right]\left(  {\left[ {\begin{smallmatrix}
			\mathcal W_{i}\\
			{\mathscr{U}_w}
			\end{smallmatrix}} \right] \cap \ker \left[ {\begin{smallmatrix}
			C&D^w
			\end{smallmatrix}} \right]} \right) 
	\end{align}
	and $\mathcal Y_i\subseteq{\mathscr{Y}}$ for $i\in \mathbb{N}$ are given by
	\begin{align}\label{Eq:barO}
	{{\mathcal Y}_i} = \left[ {\begin{smallmatrix}
		C&D^w
		\end{smallmatrix}} \right]\left[ {\begin{smallmatrix}
		{{{\mathcal W}_i}}\\
		\mathscr{U}_w
		\end{smallmatrix}} \right].
	\end{align}
	Additionally, define  {a sequence $ \hat{\mathcal W}_i $ of subspaces} as
	\begin{align}\label{Eq:barWaddition}
	\hat{\mathcal W}_{1}={{\rm Im\,} B^v}, \ \ \
	{\hat{\mathcal W}_{i + 1}} = \left[ {\begin{smallmatrix}
		A&B^w
		\end{smallmatrix}} \right]\left( {\left[ {\begin{smallmatrix}
			{{\hat{\mathcal W}_i}}\\
			{\mathscr{U}}_w
			\end{smallmatrix}} \right] \cap \ker \left[ {\begin{smallmatrix}
			C&D^w
			\end{smallmatrix}} \right]} \right) .
	\end{align}
	 {Then $\mathcal W^*=\mathcal W_n=\hat{\mathcal W}_n$ and $\mathcal Y^*=\mathcal Y_n$}.
\end{lem} 
Note that {when considering} the above defined invariant subspaces for the dual system $(\Lambda^{w})^d$ of $\Lambda^w$, given by $(\Lambda^{w})^d=(A^T,C^T,(B^w)^T,(D^w)^T)$, we have the following results \cite{morse1973structural},\cite{molinari1978structural}:
\begin{align}\label{Eq:dualinvarsubsp}
\begin{aligned}
\mathcal V^*(\Lambda^w)=\left( \mathcal W^*((\Lambda^{w})^d)\right) ^{\bot},& ~~~\mathcal W^*(\Lambda^w)=\left( \mathcal V^*((\Lambda^{w})^d)\right) ^{\bot}, \\
\mathcal U_w^*(\Lambda^w)=\left( \mathcal Y^*((\Lambda^{w})^d)\right) ^{\bot},& ~~~\mathcal Y^*(\Lambda^w)=\left( \mathcal U_w^*((\Lambda^{w})^d)\right) ^{\bot}.
\end{aligned}
\end{align} 
\\ \textbf{Acknowledgment.}   The first author of the paper is currently supported by Vidi-grant 639.032.733.

\bibliographystyle{siamplain}
\bibliography{siamref}

\end{document}